\documentclass[a4paper,10pt]{amsart}
\usepackage[utf8x]{inputenc}
\usepackage{amssymb,amsmath, amsthm, amsfonts,enumerate, wrapfig, tikz, color, tensor, lscape}
\usetikzlibrary{decorations.markings}
\usetikzlibrary{patterns}
\newcounter{res}[section]
\numberwithin{res}{section}
\newtheorem{thm}[res]{Theorem}
\newtheorem{lem}[res]{Lemma}
\newtheorem{prop}[res]{Proposition}
\newtheorem{cor}[res]{Corollary}
\theoremstyle{definition}
\newtheorem{notation}[res]{Notation}
\newtheorem{dfn}[res]{Definition}
\newtheorem{req}[res]{Remark}

\newcommand{\id}{\ensuremath{\mathrm{id}}}
\newcommand{\NN}{\mathbb N} 
\newcommand{\ZZ}{\mathbb Z} 
\newcommand{\CC}{\mathbb C} 
\newcommand{\QQ}{\mathbb Q}
\newcommand{\RR}{\mathbb R} 

\newcommand{\D}{\mathcal{D}}
\newcommand{\F}{\mathcal{F}}

\newcommand{\cjg}[1]{\overline{#1}}

\renewcommand {\epsilon}{\varepsilon}

\renewcommand {\leq}{\leqslant}
\renewcommand {\geq}{\geqslant}

\renewcommand {\bar}{\cjg}

\renewcommand\marginpar[1]{}
\newcommand\eqdef{\ensuremath{\stackrel{\textrm{def}}{=}}}
\newcommand\kup[1]{\ensuremath\left\langle #1 \right\rangle}

\newcommand{\sll}{\ensuremath{\mathfrak{sl}}}
\newcommand{\imagesfolder}{.}
\newcommand{\ie}{i.~e.~}

\newcommand{\ed}{\ensuremath{\textrm{ed}}}
\newcommand{\resp}{resp.~{}}
\newcommand{\FR}[1][]{\textbf{FR}#1}
\newcommand{\G}{\ensuremath{\mathcal{G}}}

\newcommand{\websquare}[1][0.6]{\vcenter{\hbox{\!
 \begin{tikzpicture}[scale={#1},decoration={markings, mark=at
     position 0.5 with {\arrow{>}}},postaction={decorate}]
      \draw[postaction = {decorate}] (0.5,0.5)--(1,1);
      \draw[postaction = {decorate}] (2.5,2.5)--(2,2);
      \draw[postaction = {decorate}] (2,1)--(1,1);
      \draw[postaction = {decorate}] (2,1)--(2,2);
      \draw[postaction = {decorate}] (2,1)--(2.5,0.5);
      \draw[postaction = {decorate}] (1,2)--(2,2);
      \draw[postaction = {decorate}] (1,2)--(1,1);
      \draw[postaction = {decorate}] (1,2)--(0.5,2.5);
 \end{tikzpicture}}}}
\newcommand{\webtwovert}[1][0.6]{
\vcenter{\hbox{\!
 \begin{tikzpicture}[scale={#1},decoration={markings, mark=at
     position 0.5 with {\arrow{>}}},postaction={decorate}]
      \draw[postaction= {decorate}] (0.5,0.5) .. controls (1,1) and (1,2) .. (0.5, 2.5);
      \draw[postaction= {decorate}] (2.5,2.5) .. controls (2,2) and (2,1) .. (2.5, 0.5);
 \end{tikzpicture}}}}

\newcommand{\webtwohori}[1][0.6]{\vcenter{\hbox{\!
 \begin{tikzpicture}[scale={#1},decoration={markings, mark=at
     position 0.5 with {\arrow{>}}},postaction={decorate}]
      \draw[postaction= {decorate}] (0.5,0.5) .. controls (1,1) and (2,1) .. (2.5, 0.5);
      \draw[postaction= {decorate}] (2.5,2.5) .. controls (2,2) and (1,2) .. (0.5, 2.5);
 \end{tikzpicture}}}}

\newcommand{\webvert}[1][0.6]{\vcenter{\hbox{\!
 \begin{tikzpicture}[scale={#1},decoration={markings, mark=at
     position 0.5 with {\arrow{>}}},postaction={decorate}]
      \draw[postaction= {decorate}] (0.5,0.5) -- (0.5, 2.5);
  \end{tikzpicture}}}}

\newcommand{\webbigon}[1][0.6]{\vcenter{\hbox{\!
 \begin{tikzpicture}[scale={#1},decoration={markings, mark=at
     position 0.5 with {\arrow{>}}},postaction={decorate}]
      \draw[postaction= {decorate}] (0.5,0.5) -- (0.5, 1);
      \draw[postaction= {decorate}] (0.5,2) -- (0.5, 2.5);
      \draw[postaction= {decorate}] (0.5,2) ..controls (0,1.5) and (0,1.5).. (0.5, 1);
      \draw[postaction= {decorate}] (0.5,2) ..controls (1,1.5) and (1,1.5).. (0.5, 1);
  \end{tikzpicture}}}}

\newcommand{\webcircle}[1][0.6]{\vcenter{\hbox{\!
 \begin{tikzpicture}[scale={#1},decoration={markings, mark=at
     position 0.5 with {\arrow{>}}},postaction={decorate}]
      \draw[postaction= {decorate}] (1.5,1.5) circle (1cm);
 \end{tikzpicture}}}}
\newcommand{\webcirclereverse}[1][0.6]{\vcenter{\hbox{\!
 \begin{tikzpicture}[scale={#1},decoration={markings, mark=at
     position 0.5 with {\arrow{<}}},postaction={decorate}]
      \draw[postaction= {decorate}] (1.5,1.5) circle (1cm);
 \end{tikzpicture}}}}

\newcommand{\foamcap}[1][0.6]{
\vcenter{\hbox{\!
 \begin{tikzpicture}[scale={#1},decoration={markings, mark=at
     position 0.5 with {\arrow{>}}},postaction={decorate}]
      \draw[dotted] (-0.8,0) arc (180:0:0.8 and 0.4 );
      \draw (-0.8,0) arc (-180:0:0.8 and 0.4);
      \draw (-0.8,0) arc (180:0:0.8);
 \end{tikzpicture}}}}

\newcommand{\foamgid}[1][0.6]{
\vcenter{\hbox{\!
 \begin{tikzpicture}[scale={#1},decoration={markings, mark=at
     position 0.5 with {\arrow{>}}},postaction={decorate}]
\draw (-2,0) -- +(1,0);
\draw (1,0) -- +(1,0);
\draw[dashed] (-1,0) ..controls +(1,0.5) and +(-1,0.5).. +(2,0);
\draw (-1,0) ..controls +(1,-0.5) and +(-1,-0.5).. +(2,0);
\draw (-2,3) -- +(4,0);
\draw (1,3) -- +(1,0);
\draw (-2,0) -- +(0,3);
\draw (-2,3) -- +(1,0);
\draw[postaction={decorate}] (-1,0) .. controls +(0,1.5) and +(0,1.5).. +(2,0);
\draw (2,0) -- +(0,3);
 \end{tikzpicture}}}}

\newcommand{\foamsquarec}[1][0.6]{
\vcenter{\hbox{\!
\begin{tikzpicture}[scale =#1,xshift=7cm, yshift= 14cm, decoration={markings, mark=at
     position 0.5 with {\arrow{>}}},postaction={decorate}]
\fill[gray,opacity=0.5] (-1,-3) -- (0,-4) arc (180:45:1) -- +(-1,1) arc (45:180:1)-- cycle;
\fill[gray,opacity=0.5] (1,-3) -- (2,-4) arc (0:45:1) -- +(-1,1) arc(45:0:1) --cycle;
\draw (0,-4) -- (2,-4);
\draw (2,-4) -- (1,-3);
\draw (1,-3) -- (-1,-3); 
\draw (-1,-3) -- (-0.5,-3.5);
\draw (-0.5,-3.5) -- (0,-4);
\draw (0,-4) -- ++(-0.5,-0.5) -- +(0,4);
\draw (2,-4) -- ++(0.7,-0.3)--+(0,4);
\draw (1,-3) -- ++(0.15,0.15);
\draw (1,-3) + (0.15,0.15) --++(0.5,0.5) -- +(0,4);
\draw (1.5,0.5) -- +(0,-1.3);
\draw (1.5,0.5) -- (1.5,1.5); 
\draw (-1,-3) -- ++(-0.7,0.3)--+(0,4);
\draw (0,-4) arc (180:0:1);
\draw (-1,-3) arc (180:120:1);
\draw (1,-3) arc (0:120:1);
\draw (0,-3)++(45:1) -- +(1,-1);
\draw (-1, 1) ++(-0.7, 0.3) .. controls (-1,1) and (1,1) .. (1.5,1.5);
\draw (0,0) ++(-0.5, -0.5) .. controls (0,0) and (2,0) .. (2.7,-0.3);
\end{tikzpicture}}}}

\definecolor{bleufonce}{rgb}{0,0,0.3}
\definecolor{vertfonce}{rgb}{0,0.3,0}
\definecolor{darkred}{rgb}{0.8,0,0}
\definecolor{darkblue}{rgb}{0,0,0.5}
\newcommand{\blue}{\mathrm{\color{darkblue}blue}}

\definecolor{darkgreen}{rgb}{0,0.6,0}
\newcommand{\green}{\mathrm{\color{darkgreen}green}}

\begin{document}
\title[A characterisation of indecomposable web-modules over $K^\epsilon$]{A characterisation of indecomposable web-modules over Khovanov-Kuperberg algebras}
\address{Louis-Hadrien Robert \\ IRMA \\ 7, rue René Descartes  \\ 67084 Strasbourg Cedex  \\ France}
\email{robert@math.unistra.fr}
\date{\today}
\begin{abstract}
  After shortly recalling the construction of the Khovanov-Kuperberg algebras, we give a characterisation of indecomposable web-modules. It says that a web-module is indecomposable if and only if one can deduce it directly from the Kuperberg bracket (via a Schur lemma argument). The proofs relies on the construction of idempotents given by explicit foams. These foams are encoded by combinatorial data called red graphs. The key point is to show that when, for a web $w$ the Schur lemma does not apply, one can find an appropriate red graph for $w$.
\end{abstract}
\maketitle
\tableofcontents
\section{Introduction}
The Khovanov-Kuperberg algebras $K^\epsilon$ were introduced in 2012 by Mackaay -- Pan -- Tubbenhauer \cite{2012arXiv1206.2118M} and the author \cite{LHR1} in order to give an algebraic definition of $\sll_3$-homology for tangles. 

In  \cite{2012arXiv1206.2118M}, the split Grothendieck group of the category $K^\epsilon$-$\mathsf{pmod}$ is computed. This shows, as in the $\sll_2$ case, that we have a categorification of $\hom_{U_q(\sll_3)}(\CC, V^{\otimes \epsilon})$ (here $\epsilon$ is an admissible  sequence of signs).\marginpar{sens: $\CC$ avant ou après ?} 
The proof of this result is far from being easy\footnote{In the $\sll_2$ case this is quite direct thanks to a Schur lemma argument.}, furthermore, the non-elliptic webs, the good candidates to correspond
 to the indecomposable modules, fails to play this role, the story starts to become dramatically more complicated than in the $\sll_2$ case.

This paper have a down-to-earth approach and gives a fully combinatorial characterisation of indecomposable web-modules.

\subsubsection*{Acknowledgments} The author wishes to thank Christian Blanchet for suggesting the subject. 
\label{sec:acknoledgment}
\section{The Khovanov--Kuperberg algebras}
\label{sec:khov-kuperb-algebr}







\subsection{The 2-category of web-tangles}
\label{sec:webs-foams}

\subsubsection{Webs}
In the following $\epsilon=(\epsilon^1,\dots,\epsilon^n)$ (or $\epsilon_0$, $\epsilon_1$ etc.) will always be a finite sequence of signs, its \emph{length} $n$ will be denoted by $l(\epsilon)$, such an $\epsilon$ will be \emph{admissible} if $\sum_{i=1}^{l(\epsilon)}\epsilon^i$ is divisible by 3.
\label{sec:webs}
\label{sec:two-category-web}
\begin{dfn}[Kuperberg, \cite{MR1403861}]\label{dfn:closed-web}
  A \emph{closed web} is a plane trivalent oriented finite graph (with possibly some vertexless loops and multiple edges) such that every vertex is either a sink or a source.
\end{dfn}
\begin{figure}[h]
  \centering
  \begin{tikzpicture}[yscale= 0.8, xscale= 0.8]
    \begin{scope}
   [yscale = {1}, xscale={1},decoration={markings, mark=at
     position 0.5 with {\arrow{>}}},postaction={decorate}]
\coordinate (A) at (0,0);
\coordinate (B) at (1,0);
\coordinate (C) at (2,-0.5);
\coordinate (D) at (3,0);
\coordinate (E) at (4,0);
\coordinate (F) at (5,0.5);
\coordinate (A1) at (0,1);
\coordinate (B1) at (1,1);
\coordinate (C1) at (2,1.5);
\coordinate (D1) at (3,1);
\coordinate (E1) at (4,1);
\coordinate (F1) at (4,2);
\coordinate (G) at (5,-0.5);
\coordinate (H) at (3.5,-1);

\draw[postaction=decorate] (A1) --(A);
\draw[postaction=decorate] (A1) .. controls +(-0.5,0)  and  +(-0.5,0).. (A);
\draw[postaction=decorate] (A1) -- (B1);
\draw[postaction=decorate] (B)-- (A);
\draw[postaction=decorate] (B)--(B1);
\draw[postaction=decorate] (B)--(C);
\draw[postaction=decorate] (C1)--(B1);
\draw[postaction=decorate] (C1)--(D1);
\draw[postaction=decorate] (C1)--(F1);
\draw[postaction=decorate] (D)--(C);
\draw[postaction=decorate] (D)--(D1);
\draw[postaction=decorate] (D)--(E);
\draw[postaction=decorate] (F) -- (G);
\draw[postaction=decorate] (F)-- (E);
\draw[postaction=decorate] (F) .. controls +(0,0.5) and  +(0.4,0.4) .. (F1);
\draw[postaction=decorate] (E1)--(F1);
\draw[postaction=decorate] (E1)--(D1);
\draw[postaction=decorate] (E1)--(E);
\draw[postaction=decorate] (H)--(C);
\draw[postaction=decorate] (H) .. controls +(0.5,0) and  +(0,-0.5) .. (G);
\draw[postaction=decorate] (H) .. controls +(0.2,0.3) and  +(-0.4,-0.1) .. (G);
\draw[postaction=decorate] (7,1) circle (0.5cm);
\end{scope}
  \end{tikzpicture}  
  \caption{Example of a closed web.}
  \label{fig:example-closed-web}
\end{figure}
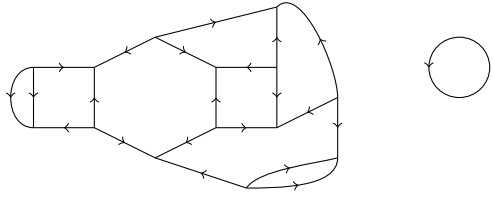
\begin{req}\label{req:basic-on-web}
The orientation condition is equivalent to say that the graph is bipartite (by sinks and sources).  
\begin{prop}
  A closed web contains at least a square, a digon or a vertexless circle.
\end{prop}\label{prop:closed2elliptic}
\begin{proof}
  It is enough to consider $w$ a connected web. A connected web is always 2-connected (because of the flow), hence it makes sense to use the Euler characteristic. Suppose that $w$ is not a circle. We have:
\[\#F - \#E + \#V = 2, \] but we have $3\#V=2\#E$. And if we denote by $F_i$ the set of faces with $i$ sides we have:
\[\sum_{i>0}i\#i\F_i = 2\#E.\]
All together, this gives:
\[
\sum_{i>0}F_i - \frac{i}6 F_i =2,
\]
and this proves that some faces have strictly less than 6 sides.
\end{proof}

\end{req}
\begin{prop}[Kuperberg\cite{MR1403861}] \label{prop:Kup}
  There exists one and only one map $\kup{\cdot}$ from closed webs to Laurent polynomials in $q$ which is invariant by isotopy, multiplicative with respect to disjoint union and which satisfies the following local relations~:
  \begin{align*}
   \kup{\websquare[0.4]} &= \kup{\webtwovert[0.4]} + \kup{\webtwohori[0.4]}, \\
   \kup{\webbigon[0.4]\,}  &= [2] \cdot \kup{\webvert[0.4]\,},\\
   \kup{\webcircle[0.4]} &= \kup{\webcirclereverse[0.4]} = [3],
  \end{align*}
where $[n]\eqdef\frac{q^n-q^{-n}}{q-q^{-1}}.$ We call this polynomial the \emph{Kuperberg bracket}. It's easy to check that the Kuperberg bracket of a web is symmetric in $q$ and $q^{-1}$.
\end{prop}
\begin{proof}
  Uniqueness comes from remark \ref{req:basic-on-web}. The existence follows from the representation theoretic point of view developed in \cite{MR1403861}. Note that a non-quantified version of this result is in \cite{MR1172374}.
\end{proof}
\begin{dfn}
  The \emph{degree} of a symmetric Laurent polynomial $P(q)=\sum_{i\in \ZZ}a_iq^{i}$ is  $\max_{i\in \ZZ}\{i \textrm{ such that }a_i\neq 0\}$.
\end{dfn}
\begin{dfn}\label{dfn:webtangle}
  A \emph{$(\epsilon_0,\epsilon_1)$-web-tangle} $w$ is an intersection of a closed web $w'$ with $[0,1]\times [0,1]$ such that :
  \begin{itemize}
  \item there exists $\eta_0 \in ]0,1]$ such that: \\ \(w\cap [0,1] \times [0,\eta_0] = \{\frac{1}{2l(\epsilon_0)}, \frac{1}{2l(\epsilon_0)} + \frac{1}{l(\epsilon_0)},  \frac{1}{2l(\epsilon_0)} + \frac{2}{l(\epsilon_0)}, \dots, \frac{1}{2l(\epsilon_0)} + \frac{l(\epsilon_0)-1}{l(\epsilon_0)} \}\times [0,\eta_0],\)
  \item there exists $\eta_1 \in [0,1[$ such that: \\ \(w\cap [0,1] \times [\eta_1,1] = \{\frac{1}{2l(\epsilon_1)}, \frac{1}{2l(\epsilon_1)} + \frac{1}{l(\epsilon_1)},  \frac{1}{2l(\epsilon_1)} + \frac{2}{l(\epsilon_1)}, \dots, \frac{1}{2l(\epsilon_1)} + \frac{l(\epsilon_1)-1}{l(\epsilon_1)} \}\times [\eta_1,1],\)
  \item the orientations of the edges of $w$, match $-\epsilon_0$ and $+\epsilon_1$ (see figure \ref{fig:exampl_webtangle} for conventions).
  \end{itemize}
When $\epsilon_1$ is the empty sequence, then we'll speak of \emph{$\epsilon_0$-webs}. And if $w$ is an $\epsilon$-web we will say that $\epsilon$ is the \emph{boundary} of $w$.
\end{dfn}
If $w_1$ is a $(\epsilon_0,\epsilon_1)$-web-tangle and $w_2$ is a $(\epsilon_1, \epsilon_2)$-web-tangle we define $w_1w_2$ to be the $(\epsilon_0,\epsilon_2)$-web-tangle obtained by gluing $w_1$ and $w_2$ along $\epsilon_1$ and resizing. Note that this can be thought as a composition if we think about a $(\epsilon,\epsilon')$-web-tangle as a morphism from $\epsilon'$ to $\epsilon$ (\ie the web-tangles should be read as morphisms from top to bottom). The \emph{mirror image} of a $(\epsilon_0,\epsilon_1)$-web-tangle $w$ is mirror image of $w$ with respect to $\RR\times \{\frac12\}$ with all orientations reversed. This is a $(\epsilon_1,\epsilon_0)$-web-tangle and we denote it by $\bar{w}$. If $w$ is a $(\epsilon,\epsilon)$-web-tangle the \emph{closure} of $w$ is the closed web obtained by connecting the top and the bottom by simple arcs (this is like a braid closure). We denote it by $\mathrm{tr}(w)$.   
\begin{figure}[h]
  \centering
  \begin{tikzpicture}[yscale= 0.5, xscale= 0.5]
    \input{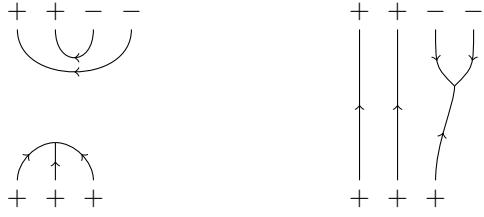}
  \end{tikzpicture}
  \caption{Two examples of $(\epsilon_0,\epsilon_1)$-web-tangles with $\epsilon_0 = (-,-,-)$ and $\epsilon_1=(-,-,+,+)$ and the mirror image of the second one.}
  \label{fig:exampl_webtangle}
\end{figure}
\begin{dfn}
  A web-tangle with no circle, no digon and no square is said to be \emph{non-elliptic}. The non-elliptic web-tangles are the minimal ones in the sense that they cannot be reduced by the relations of proposition \ref{prop:Kup}.
\end{dfn}
\begin{prop}[Kuperberg, \cite{MR1403861}] \label{prop:NEFinite}
  For any given couple $(\epsilon_0,\epsilon_1)$ of sequences of signs there are finitely many non-elliptic $(\epsilon_0,\epsilon_1)$-web-tangles.
\end{prop}
\begin{req}
  From the combinatorial flow modulo 3, we obtain that there exist some $(\epsilon_0,\epsilon_1)$-webs if and only if the sequence $-\epsilon_0$ concatenated with $\epsilon_1$ is admissible.
\end{req}
\subsubsection{Foams} All material here comes from \cite{MR2100691}.
\label{sec:foams}
\begin{dfn}
  A \emph{pre-foam} is a smooth oriented compact surface $\Sigma$ (its connected component are called \emph{facets}) together with the following data~:
\begin{itemize}
\item A partition of the connected components of the boundary into cyclically ordered 3-sets and for each 3-set $(C_1,C_2,C_3)$, three orientation preserving diffeomorphisms $\phi_1:C_2\to C_3$, $\phi_2:C_3\to C_1$ and $\phi_3:C_1\to C_2$ such that $\phi_3 \circ \phi_2 \circ \phi_1 = \mathrm{id}_{C_2}$.
\item A function from the set of facets to the set of non-negative integers (this gives the number of \emph{dots} on each facet).
\end{itemize}
The \emph{CW-complex associated with a pre-foam} is the 2-dimensional CW-complex $\Sigma$ quotiented by the diffeomorphisms so that the three circles of one 3-set are identified and become just one called a \emph{singular circle}.
The \emph{degree} of a pre-foam $f$ is equal to $-2\chi(\Sigma')$ where $\chi$ is the Euler characteristic, $\Sigma'$ is the CW-complex associated with $f$ with the dots punctured out (\ie a dot increases the degree by 2).
\end{dfn}
\begin{req}
  The CW-complex has two local models depending on whether we are on a singular circle or not. If a point $x$ is not on a singular circle, then it has a neighborhood diffeomorphic to a 2-dimensional disk, else it has a neighborhood diffeomorphic to a Y shape times an interval (see figure \ref{fig:yshape}).
  \begin{figure}[h]
    \centering
    \begin{tikzpicture}[scale=1]
      \input{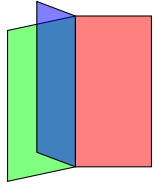}
    \end{tikzpicture}
    \caption{Singularities of a pre-foam}
    \label{fig:yshape}
  \end{figure}
\end{req}
\begin{dfn}
  A \emph{closed foam} is the image of an embedding of the CW-complex associated with a pre-foam such that the cyclic orders of the pre-foam are compatible with the left-hand rule in $\RR^3$ with respect to the orientation of the singular circles\footnote{We mean here that if, next to a singular circle, with the forefinger of the left hand we go from face 1 to face 2 to face 3 the thumb points to indicate the orientation of the singular circle (induced by orientations of facets). This is not quite canonical, physicists use more the right-hand rule, however this is the convention used in \cite{MR2100691}.}. The \emph{degree} of a closed foam is the degree of the underlying pre-foam. 
\end{dfn}
\begin{dfn}\label{dfn:wwfoam}
  If $w_b$ and $w_t$ are $(\epsilon_0,\epsilon_1)$-web-tangles, a \emph{$(w_b,w_t)$-foam} $f$ is the intersection of a foam $f'$ with $\RR\times [0,1]\times[0,1]$ such that
  \begin{itemize}
  \item there exists $\eta_0 \in ]0,1]$ such that $f\cap \RR \times [0,\eta_0]\times [0,1] = \{\frac{1}{2l(\epsilon_0)}, \frac{1}{2l(\epsilon_0)} + \frac{1}{l(\epsilon_0)},  \frac{1}{2l(\epsilon_0)} + \frac{2}{l(\epsilon_0)}, \dots, \frac{1}{2l(\epsilon_0)} + \frac{l(\epsilon_0)-1}{l(\epsilon_0)} \}\times [0,\eta_0]\times [0,1]$,
  \item there exists $\eta_1 \in [0,1[$ such that $f\cap \RR \times [\eta_1,1]\times [0,1] = \{\frac{1}{2l(\epsilon_1)}, \frac{1}{2l(\epsilon_1)} + \frac{1}{l(\epsilon_1)},  \frac{1}{2l(\epsilon_1)} + \frac{2}{l(\epsilon_1)}, \dots, \frac{1}{2l(\epsilon_1)} + \frac{l(\epsilon_1)-1}{l(\epsilon_1)}\}\times [\eta_1,1]\times [0,1]$,
  \item there exists $\eta_b \in ]0,1]$ such that $f\cap \RR \times [0,1 ]\times  [0, \eta_b] = w_b \times [0, \eta_b]$,
  \item there exists $\eta_t \in [0,1[$ such that $f\cap \RR \times [0,1 ]\times  [\eta_t, 1] = w_t \times [\eta_t,1]$,
  \end{itemize}
with compatibility of orientations of the facets of $f$ with the orientation of $w_t$ and the reversed orientation of $w_b$.
The \emph{degree} of a $(w_b,w_t)$-foam $f$ is equal to $\chi(w_b)+\chi(w_t)-2\chi(\Sigma)$ where $\Sigma$ is the underlying CW-complex associated with $f$ with the dots punctured out. 
\end{dfn}
If $f_b$ is a $(w_b,w_m)$-foam and $f_t$ is a $(w_m, w_t)$-foam we define $f_bf_t$ to be the $(w_b,w_t)$ foam obtained by gluing $f_b$ and $f_t$ along $w_m$ and resizing. This operation may be thought as a composition if we think of a $(w_1,w_2)$-foam as a morphism from $w_2$ to $w_1$ \ie from the top to the bottom. This composition map is degree preserving. Like for the webs, we define the \emph{mirror image} of a $(w_1,w_2)$-foam $f$ to be the $(w_2,w_1)$-foam which is the mirror image of $f$ with respect to $\RR\times\RR\times \{\frac12\}$ with all orientations reversed. We denote it by $\bar{f}$.
\begin{dfn}
  If $\epsilon_0=\epsilon_1=\emptyset$ and $w$ is a closed web, then a $(\emptyset,w)$-foam is simply called \emph{foam} or \emph{$w$-foam} when one wants to focus on the boundary of the foam.
\end{dfn}
All these data together lead  to the definition of a monoidal 2-category. 
\begin{dfn}
  The 2-category $\mathcal{WT}$ is the monoidal\footnote{Here we choose a rather strict point of view and hence the monoidal structure is strict (we consider everything up to isotopy), but it is possible to define the notion in a non-strict context, and the same data gives us a monoidal bicategory.} 2-category given by the following data~:
  \begin{itemize}
  \item The objects are finite sequences of signs,
  \item The 1-morphism from $\epsilon_1$ to $\epsilon_0$ are isotopy classes (with fixed boundary) of $(\epsilon_0,\epsilon_1)$-web-tangles,
  \item The 2-morphism from $\widehat{w_t}$ to $\widehat{w_b}$ are $\QQ$-linear combinations of isotopy classes of $(w_b,w_t)$-foams, where\ \ $\widehat{\cdot}$\ \ stands for the ``isotopy class of''. The 2-morphisms come with a grading, the composition respects the degree.
  \end{itemize}
The monoidal structure is given by concatenation of sequences at the $0$-level, and disjoint union of vertical strands or disks (with corners) at the $1$ and $2$ levels. 
\end{dfn}
\subsection{Khovanov's TQFT for web-tangles}
\label{sec:khovanov-tqft-web}
In \cite{MR2100691}, Khovanov defines a numerical invariant for pre-foams and this allows him to construct a TQFT $\mathcal{F}$ from the category $\hom_{\mathcal{WT}}(\emptyset,\emptyset)$ to the category of graded $\QQ$-modules, (via a universal construction à la BHMV 
\cite{MR1362791}). This TQFT is graded (this comes from the fact that pre-foams with non-zero degree are evaluated to zero), and satisfies the following local relations (brackets indicate grading shifts)~:
\begin{align*}
\mathcal{F}\left(\websquare[0.4]\,\right) &= 
\mathcal{F}\left({\webtwovert[0.4]}\,\right) \oplus
\mathcal{F}\left({\webtwohori[0.4]}\,\right), \\
\mathcal{F}\left({\webbigon[0.4]}\,\right)  &= 
\mathcal{F}\left({\webvert[0.4]}\,\right)\{-1\}\oplus \mathcal{F}\left({\webvert[0.4]}\right)\{1\},\\
\mathcal{F}\left({\webcircle[0.4]\,}\right) &= 
\mathcal{F}\left({\webcirclereverse[0.4]}\,\right) = \QQ\{-2\} \oplus \QQ \oplus\QQ\{2\}.
  \end{align*}
These relations show that $\mathcal{F}$ is a categorified counterpart of the Kuperberg bracket. We sketch the construction below.
\begin{dfn}
We denote by $\mathcal{A}$ the Frobenius algebra $\ZZ[X]/(X^3)$ with trace $\tau$ given by:
\[\tau(X^2)=-1, \quad \tau(X)=0, \quad \tau(1)=0.\] 
We equip $\mathcal{A}$ with a graduation by setting $\deg(1)=-2$, $\deg(X)=0$ and $\deg(X^2)=2$. With these settings, the multiplication has degree 2 and the trace has degree -2. The co-multiplication is determined by the multiplication and the trace and we have :
\begin{align*}
  &\Delta(1) = -1\otimes X^2 - X\otimes X - X^2\otimes 1 \\
  &\Delta(X) = -X\otimes X^2 - X^2\otimes X \\
  &\Delta(X^2) = -X^2\otimes X^2
\end{align*}
\end{dfn}
This Frobenius algebra gives us a 1+1 TQFT (this is well-known, see~\cite{MR2037238} for details), we denote it by $\mathcal{F}$~: the circle is sent to $\mathcal{A}$, a cup to the unity, a cap to the trace, and a pair of pants either to multiplication or co-multiplication. A dot on a surface will represent multiplication by $X$ so that $\mathcal{F}$ extends to the category of oriented dotted (1+1)-cobordisms.
We then have a surgery formula given by figure~\ref{fig:surg}. 
\begin{figure}[h!]
  \centering
  \begin{tikzpicture}[scale=0.5]
    \begin{scope}[xshift=-0.5cm]
       \draw (0,-1) arc (270:90:0.5 and 1);
      \draw[dotted] (0,-1) arc (-90:90:0.5 and 1);
      \draw (3,0) ellipse (0.5 and 1);
      \draw (0,-1) -- (3,-1);
      \draw (0,1) -- (3,1);
\end{scope}
\node at (4.5,0) {$=-$};
\node at (10,0) {$-$};
\node at (15,0) {$-$};
\begin{scope}[xshift=6cm]
       \draw (0,-1) arc (270:90:0.5 and 1);
      \draw[dotted] (0,-1) arc (-90:90:0.5 and 1);
      \draw (3,0) ellipse (0.5 and 1);
      \draw (0,-1) .. controls +(1.5,0) and +(1.5,0) .. (0,1);
      \draw (3,-1) .. controls +(-1.5,0) and +(-1.5,0) .. (3,1);
      \filldraw (0.7,0.2) ellipse (1pt and 2pt);
      \filldraw (0.7,-0.2) ellipse (1pt and 2pt);
\end{scope}
\begin{scope}[xshift=11cm]
       \draw (0,-1) arc (270:90:0.5 and 1);
      \draw[dotted] (0,-1) arc (-90:90:0.5 and 1);
      \draw (3,0) ellipse (0.5 and 1);
      \draw (0,-1) .. controls +(1.5,0) and +(1.5,0) .. (0,1);
      \draw (3,-1) .. controls +(-1.5,0) and +(-1.5,0) .. (3,1);
      \filldraw (0.7,0) ellipse (1pt and 2pt);
      \filldraw (2.2,0) ellipse (1pt and 2pt);
\end{scope}
\begin{scope}[xshift=16cm]
       \draw (0,-1) arc (270:90:0.5 and 1);
      \draw[dotted] (0,-1) arc (-90:90:0.5 and 1);
      \draw (3,0) ellipse (0.5 and 1);
      \draw (0,-1) .. controls +(1.5,0) and +(1.5,0) .. (0,1);
      \draw (3,-1) .. controls +(-1.5,0) and +(-1.5,0) .. (3,1);
      \filldraw (2.2,-0.2) ellipse (1pt and 2pt);
      \filldraw (2.2,0.2) ellipse (1pt and 2pt);
\end{scope}
  \end{tikzpicture}
  \caption{The surgery formula for the TQFT $\mathcal{F}$.}
  \label{fig:surg}
\end{figure}
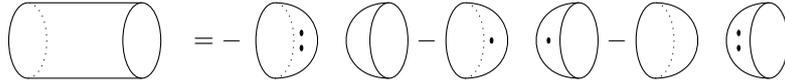

 This TQFT gives of course a numercial invariant for closed dotted oriented surfaces. If one defines numerical values for the differently dotted theta pre-foams (the theta pre-foam consists of 3 disks with trivial diffeomorphisms between their boundary see figure \ref{fig:thetapre}) then by applying the surgery formula, one is able to compute a numerical value for pre-foams. 

\begin{figure}[h!]
  \centering
  \begin{tikzpicture}
    \input{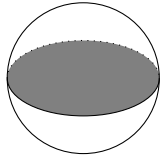}
  \end{tikzpicture}
  \caption{The dotless theta pre-foam.}
  \label{fig:thetapre}
\end{figure}
\begin{figure}[h!]
  \centering
\begin{tikzpicture}[scale = 0.7]
\input{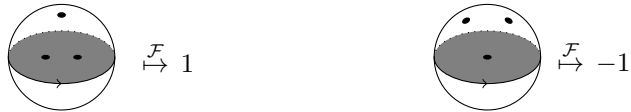}
\end{tikzpicture}
  \caption{The evaluations of dotted theta foams, the evaluation is unchanged when one cyclically permutes the faces. All the configurations which cannot be obtained from these by cyclic permutation are sent to $0$ by $\mathcal{F}$.}
  \label{fig:thetaeval}
\end{figure}

In \cite{MR2100691}, Khovanov shows that setting the evaluations of the dotted theta foams as shown on figure \ref{fig:thetaeval}, leads to a well defined numerical invariant $\mathcal{F}$ for pre-foams. This numerical invariant gives the opportunity to build a (closed web, $(\cdot,\cdot)$-foams)-TQFT~: for a given web $w$, consider the $\QQ$-module generated by all the $(w,\emptyset)$-foams, and mod this space out by the kernel of the bilinear map $(f,g)\mapsto \mathcal{F}(\bar{f}g)$. Note that $\bar{f}g$ is a closed foam. Khovanov showed that the obtained graded vector spaces are finite dimensional with graded dimensions given by the Kuperberg formulae, and he showed that we have the local relations described on figure~\ref{fig:localrel}.

\begin{figure}[h]
  \centering 
  \begin{tikzpicture}[scale=0.25]
\begin{scope}[xshift=0cm, yshift= 0cm, decoration={markings, mark=at
     position 0.5 with {\arrow{>}}},postaction={decorate}]
\draw[postaction= {decorate}] (-2,0) arc (-180:0:2 and 1);
\draw (-2,0) arc (180:0:2);
\draw[dashed] (-2,0) arc (180:0:2 and 1);
\draw[dashed] (-2,0) arc (-180:-135:2);
\draw[dashed] (2,0) arc (0:-45:2);
\draw (0,0)++(-135:2) arc (-135:-45:2);
\draw[dashed] (0,0) +(135:2) -- +(45:2);
\draw (0,0)++(135:2) --  ++(-0.7,0) -- ++(-135:4) -- ++(7,0) -- ++(45:4) -- (45:2);
\filldraw (0,1.8) ellipse (3pt and 1.5pt);
\end{scope}
\node at (5,0) {$=-$};
\begin{scope}[xshift=10cm, yshift= 0cm, decoration={markings, mark=at
     position 0.5 with {\arrow{>}}},postaction={decorate}]
\draw[postaction= {decorate}] (-2,0) arc (-180:0:2 and 1);
\draw (-2,0) arc (180:0:2);
\draw[dashed] (-2,0) arc (180:0:2 and 1);
\draw[dashed] (-2,0) arc (-180:-135:2);
\draw[dashed] (2,0) arc (0:-45:2);
\draw (0,0)++(-135:2) arc (-135:-45:2);
\draw[dashed] (0,0) +(135:2) -- +(45:2);
\draw (0,0)++(135:2) --  ++(-0.7,0) -- ++(-135:4) -- ++(7,0) -- ++(45:4) -- (45:2);
\filldraw (0,-1.8) ellipse (3pt and 1.5pt);
\end{scope}
\node at (15,0) {$=$};
\begin{scope}[xshift=20cm, yshift= 0cm, decoration={markings, mark=at
     position 0.5 with {\arrow{>}}},postaction={decorate}]
\draw (0,0)++(135:2) --  ++(-0.7,0) -- ++(-135:4) -- ++(7,0) -- ++(45:4) -- (135:2);
\end{scope}
\begin{scope} [yshift= -6cm]
  \begin{scope}[xshift=0cm, yshift= 0cm, decoration={markings, mark=at
     position 0.5 with {\arrow{>}}},postaction={decorate}]
\draw[postaction= {decorate}] (-2,0) arc (-180:0:2 and 1);
\draw (-2,0) arc (180:0:2);
\draw[dashed] (-2,0) arc (180:0:2 and 1);
\draw[dashed] (-2,0) arc (-180:-135:2);
\draw[dashed] (2,0) arc (0:-45:2);
\draw (0,0)++(-135:2) arc (-135:-45:2);
\draw[dashed] (0,0) +(135:2) -- +(45:2);
\draw (0,0)++(135:2) --  ++(-0.7,0) -- ++(-135:4) -- ++(7,0) -- ++(45:4) -- (45:2);
\end{scope}
\node at (5,0) {$=$};
\begin{scope}[xshift=10cm, yshift= 0cm, decoration={markings, mark=at
     position 0.5 with {\arrow{>}}},postaction={decorate}]
\draw[postaction= {decorate}] (-2,0) arc (-180:0:2 and 1);
\draw (-2,0) arc (180:0:2);
\draw[dashed] (-2,0) arc (180:0:2 and 1);
\draw[dashed] (-2,0) arc (-180:-135:2);
\draw[dashed] (2,0) arc (0:-45:2);
\draw (0,0)++(-135:2) arc (-135:-45:2);
\draw[dashed] (0,0) +(135:2) -- +(45:2);
\draw (0,0)++(135:2) --  ++(-0.7,0) -- ++(-135:4) -- ++(7,0) -- ++(45:4) -- (45:2);
\filldraw (0,1.8) ellipse (3pt and 1.5pt);
\filldraw (0,-1.8) ellipse (3pt and 1.5pt);
\end{scope}
\node at (15,0) {$=$};
\begin{scope}[xshift=20cm, yshift= 0cm, decoration={markings, mark=at
     position 0.5 with {\arrow{>}}},postaction={decorate}]
\draw (0,0)++(135:2) --  ++(-0.7,0) -- ++(-135:4) -- ++(7,0) -- ++(45:4) -- (135:2);
\filldraw (0.5,0) ellipse (3pt and 1.5pt);
\filldraw (0,0) ellipse (3pt and 1.5pt);
\filldraw (-0.5,0) ellipse (3pt and 1.5pt);
\end{scope}
\node at (25,0) {$=0$};
\end{scope}
\begin{scope}[yshift=-12cm]
  \begin{scope}[xshift=0cm, yshift= 0cm, decoration={markings, mark=at
     position 0.5 with {\arrow{>}}},postaction={decorate}]
\draw[postaction= {decorate}] (-2,0) arc (-180:0:2 and 1);
\draw (-2,0) arc (180:0:2);
\draw[dashed] (-2,0) arc (180:0:2 and 1);
\draw[dashed] (-2,0) arc (-180:-135:2);
\draw[dashed] (2,0) arc (0:-45:2);
\draw (0,0)++(-135:2) arc (-135:-45:2);
\draw[dashed] (0,0) +(135:2) -- +(45:2);
\draw (0,0)++(135:2) --  ++(-0.7,0) -- ++(-135:4) -- ++(7,0) -- ++(45:4) -- (45:2);
\filldraw (-0.2,-1.8) ellipse (3pt and 1.5pt);
\filldraw (0.2,-1.8) ellipse (3pt and 1.5pt);
\end{scope}
\node at (5,0) {$=-$};
\begin{scope}[xshift=10cm, yshift= 0cm, decoration={markings, mark=at
     position 0.5 with {\arrow{>}}},postaction={decorate}]
\draw[postaction= {decorate}] (-2,0) arc (-180:0:2 and 1);
\draw (-2,0) arc (180:0:2);
\draw[dashed] (-2,0) arc (180:0:2 and 1);
\draw[dashed] (-2,0) arc (-180:-135:2);
\draw[dashed] (2,0) arc (0:-45:2);
\draw (0,0)++(-135:2) arc (-135:-45:2);
\draw[dashed] (0,0) +(135:2) -- +(45:2);
\draw (0,0)++(135:2) --  ++(-0.7,0) -- ++(-135:4) -- ++(7,0) -- ++(45:4) -- (45:2);
\filldraw (0.2,1.8) ellipse (3pt and 1.5pt);
\filldraw (-0.2,1.8) ellipse (3pt and 1.5pt);
\end{scope}
\node at (15,0) {$=$};
\begin{scope}[xshift=20cm, yshift= 0cm, decoration={markings, mark=at
     position 0.5 with {\arrow{>}}},postaction={decorate}]
\draw (0,0)++(135:2) --  ++(-0.7,0) -- ++(-135:4) -- ++(7,0) -- ++(45:4) -- (135:2);
\filldraw (0,0) ellipse (3pt and 1.5pt);
\end{scope}
\end{scope}
\end{tikzpicture}

\vspace{0.7cm}
\begin{tikzpicture}[scale=0.24]

\begin{scope}[xshift=0cm, yshift= 0cm, decoration={markings, mark=at
     position 0.5 with {\arrow{>}}},postaction={decorate}]
\draw (0,-2) arc (270:90:1 and 2);
\draw[dashed] (0,-2) arc (-90:90:1 and 2);
\fill[gray] (3,0) ellipse (1 and 2);
\draw (3,-2) arc (270:90:1 and 2);
\draw[dashed] (3,-2) arc (-90:90:1 and 2);
\fill[gray] (6,0) ellipse (1 and 2);
\draw (6,-2) arc (270:90:1 and 2);
\draw[dashed] (6,-2) arc (-90:90:1 and 2);
\draw (9,0) ellipse (1 and 2);
\draw (0,-2) -- (9,-2);
\draw (0,2) -- (9,2);
\node at (12,0) {$=-$};
\draw (15,-2) arc (270:90:1 and 2);
\draw[dashed] (15,-2) arc (-90:90:1 and 2);
\draw (15,-2) arc (-90:90:2 and 2);
\draw (20,-2) arc (-90:-270:2 and 2);
\draw(20,0) ellipse (1 and 2);
\end{scope}

\begin{scope}[xshift=-2cm, yshift= -6cm, decoration={markings, mark=at
     position 0.5 with {\arrow{>}}},postaction={decorate}]
\draw (0,-2) arc (270:90:1 and 2);
\draw[dashed] (0,-2) arc (-90:90:1 and 2);
\fill[gray] (3,0) ellipse (1 and 2);
\draw[postaction= {decorate}] (3,-2) arc (270:90:1 and 2);
\draw[dashed] (3,-2) arc (-90:90:1 and 2);
\draw (6,0) ellipse (1 and 2);
\draw (0,-2) -- (6,-2);
\draw (0,2) -- (6,2);
\node at (8,0) {$=$};
\draw (10,-2) arc (270:90:1 and 2);
\draw[dashed] (10,-2) arc (-90:90:1 and 2);
\draw (10,-2) arc (-90:90:2 and 2);
\draw (15,-2) arc (-90:-270:2 and 2);
\draw(15,0) ellipse (1 and 2);
\fill (15,0) ellipse (2pt and 4pt);
\node at (17,0) {$-$};
\fill (19,0) ellipse (2pt and 4pt);b
\draw (19,-2) arc (270:90:1 and 2);
\draw[dashed] (19,-2) arc (-90:90:1 and 2);
\draw (19,-2) arc (-90:90:2 and 2);
\draw (24,-2) arc (-90:-270:2 and 2);
\draw(24,0) ellipse (1 and 2);
\end{scope} 
\end{tikzpicture}

\vspace{0.7cm}
\begin{tikzpicture}[scale=0.4]
\begin{scope}[xshift=0cm, yshift= 0cm, decoration={markings, mark=at
     position 0.5 with {\arrow{>}}},postaction={decorate}]
\draw (-2,0) -- +(1,0);
\draw (1,0) -- +(1,0);
\draw[dashed] (-1,0) .. controls +(1,0.5) and +(-1,0.5).. +(2,0);
\draw (-1,0) .. controls +(1,-0.5) and +(-1,-0.5).. +(2,0);
\draw (-2,3) -- +(1,0);
\draw (1,3) -- +(1,0);
\draw (-1,3) .. controls +(1,0.5) and +(-1,0.5).. +(2,0);
\draw (-1,3) .. controls +(1,-0.5) and +(-1,-0.5).. +(2,0);
\draw (-2,0) -- +(0,3);
\draw[postaction={decorate}] (-1,0) -- +(0,3);
\draw[postaction={decorate}] (1,3) -- +(0,-3);
\draw (2,0) -- +(0,3);
\end{scope}
\node at (3,1.5) {$=$};
\begin{scope}[xshift=6cm, yshift= 0cm, decoration={markings, mark=at
     position 0.5 with {\arrow{>}}},postaction={decorate}]
\draw (-2,0) -- +(1,0);
\draw (1,0) -- +(1,0);
\draw[dashed] (-1,0) ..controls +(1,0.5) and +(-1,0.5).. +(2,0);
\draw (-1,0) ..controls +(1,-0.5) and +(-1,-0.5).. +(2,0);
\draw (-2,3) -- +(1,0);
\draw (1,3) -- +(1,0);
\draw (-1,3) ..controls +(1,0.5) and +(-1,0.5).. +(2,0);
\draw (-1,3) ..controls +(1,-0.5) and +(-1,-0.5).. +(2,0);
\draw (-2,0) -- +(0,3);
\draw[postaction={decorate}] (-1,0) .. controls +(0,1.5) and +(0,1.5).. +(2,0);
\draw[postaction={decorate}] (1,3) .. controls +(0,-1.5) and +(0,-1.5).. +(-2,0);
\draw (2,0) -- +(0,3);
\fill (0,3) circle (2pt and 2pt);
\end{scope}
\node at (9,1.5) {$-$}; 
\begin{scope}[xshift=12cm, yshift= 0cm, decoration={markings, mark=at
     position 0.5 with {\arrow{>}}},postaction={decorate}]
\draw (-2,0) -- +(1,0);
\draw (1,0) -- +(1,0);
\draw[dashed] (-1,0).. controls +(1,0.5) and +(-1,0.5).. +(2,0);
\draw (-1,0) ..controls +(1,-0.5) and +(-1,-0.5).. +(2,0);
\draw (-2,3) -- +(1,0);
\draw (1,3) -- +(1,0);
\draw (-1,3).. controls +(1,0.5) and +(-1,0.5).. +(2,0);
\draw (-1,3).. controls +(1,-0.5) and +(-1,-0.5).. +(2,0);
\draw (-2,0) -- +(0,3);
\draw[postaction={decorate}] (-1,0) .. controls +(0,1.5) and +(0,1.5).. +(2,0);
\draw[postaction={decorate}] (1,3) .. controls +(0,-1.5) and +(0,-1.5).. +(-2,0);
\draw (2,0) -- +(0,3);
\fill (0,0) circle (2pt and 2pt);
\end{scope} 
\end{tikzpicture}

\vspace{0.7cm}
\begin{tikzpicture}[scale=0.32]
\begin{scope}[xshift=0cm, yshift= 0cm]
\draw (0,0) -- (2,0);
\draw (2,0) -- (1,1);
\draw (1,1) -- (-1,1); 
\draw (-1,1) -- (0,0);
\draw (0,0) -- +(-0.5,-0.5);
\draw (2,0) -- +(0.7,-0.3);
\draw (1,1) -- +(0.5,0.5); 
\draw (-1,1) -- +(-0.7,0.3);
\draw (0,-4) -- (2,-4);
\draw[dashed] (2,-4) -- (1,-3);
\draw[dashed] (1,-3) -- (-1,-3); 
\draw (-1,-3) -- (-0.5,-3.5);
\draw[dashed] (-0.5,-3.5) -- (0,-4);
\draw (0,-4) -- ++(-0.5,-0.5) -- +(0,4);
\draw (2,-4) -- ++(0.7,-0.3)--+(0,4);
\draw[dashed] (1,-3) -- ++(0.5,0.5)--+(0,3);
\draw (1.5,0.5) -- (1.5,1.5); 
\draw (-1,-3) -- ++(-0.7,0.3)--+(0,4);
\draw (0,-4) -- +(0,4);
\draw (2,-4) -- +(0,4);
\draw[dashed] (1,-3) -- +(0,3);
\draw (1,0)-- +(0,1);
\draw (-1,-3) -- +(0,4);
\end{scope}
\node at (3.9,-2) {$=-$};
\begin{scope}[xshift=7cm, yshift= 0cm, decoration={markings, mark=at
     position 0.5 with {\arrow{>}}},postaction={decorate}]
\fill[gray, opacity =0.5] (-1,1) -- (0,0) arc (180:225:1) -- +(-1,1) arc (225:180:1)-- cycle;
\fill[gray,opacity =0.5] (1,1) -- (2,0) arc (0:-135:1) -- +(-1,1) arc(-135:0:1) --cycle;
\fill[gray,opacity=0.5] (-1,-3) -- (0,-4) arc (180:45:1) -- +(-1,1) arc (45:180:1)-- cycle;
\fill[gray,opacity=0.5] (1,-3) -- (2,-4) arc (0:45:1) -- +(-1,1) arc(45:0:1) --cycle;
\draw (0,0) -- (2,0);
\draw (2,0) -- (1,1);
\draw (1,1) -- (-1,1); 
\draw (-1,1) -- (0,0);
\draw (0,0) -- +(-0.5,-0.5);
\draw (2,0) -- +(0.7,-0.3);
\draw (1,1) -- +(0.5,0.5); 
\draw (-1,1) -- +(-0.7,0.3);
\draw (0,-4) -- (2,-4);
\draw[dashed] (2,-4) -- (1,-3);
\draw[dashed] (1,-3) -- (-1,-3); 
\draw (-1,-3) -- (-0.5,-3.5);
\draw[dashed] (-0.5,-3.5) -- (0,-4);
\draw (0,-4) -- ++(-0.5,-0.5) -- +(0,4);
\draw (2,-4) -- ++(0.7,-0.3)--+(0,4);
\draw[dashed] (1,-3) -- ++(0.15,0.15);
\draw (1,-3) + (0.15,0.15) --++(0.5,0.5) -- +(0,1.7);
\draw[dashed] (1.5,0.5) -- +(0,-1.3);
\draw (1.5,0.5) -- (1.5,1.5); 
\draw (-1,-3) -- ++(-0.7,0.3)--+(0,4);
\draw (0,-4) arc (180:0:1);
\draw (0,0) arc (-180:0:1);
\draw[dashed] (1,1) arc (0:-180:1);
\draw (-1,-3) arc (180:120:1);
\draw[dashed] (1,-3) arc (0:120:1);
\draw (0,-3)++(45:1) -- +(1,-1);
\draw (0,1)++(-135:1) -- +(1,-1);
\end{scope}
\node at (10.9,-2) {$-$};
\begin{scope}[xshift=14cm, yshift= 0cm, decoration={markings, mark=at
     position 0.5 with {\arrow{>}}},postaction={decorate}]
\fill[gray,opacity=0.5] (-1,1) ..controls +(0,-1.5) and +(-0.1,0).. (-0.5,-0.8) -- ++(2,0).. controls +(-0.1,0) and +(0,-1.5) .. (1,1);
\fill[gray,opacity=0.5] (0,0) ..controls +(0,-0.8) and +(0.1,0).. (-0.5,-0.8) -- ++(2,0).. controls +(0.1,0) and +(0,-0.8) .. (2,0);
\fill[gray,opacity=0.5] (-1,-3) ..controls +(0,0.8) and +(-0.1,0).. ++(0.5,0.8) -- ++(2,0).. controls +(-0.1,0) and +(0,0.8) .. (1,-3);
\fill[gray,opacity=0.5] (0,-4) ..controls +(0,1.5) and +(0.1,0).. (-0.5,-2.2) -- ++(2,0).. controls +(0.1,0) and +(0,1.5) .. (2,-4);
\draw (0,0) -- (2,0);
\draw (2,0) -- (1,1);
\draw (1,1) -- (-1,1); 
\draw (-1,1) -- (0,0);
\draw (0,0) -- +(-0.5,-0.5);
\draw (2,0) -- +(0.7,-0.3);
\draw (1,1) -- +(0.5,0.5); 
\draw (-1,1) -- +(-0.7,0.3);
\draw (0,-4) -- (2,-4);
\draw[dashed] (2,-4) -- (1,-3);
\draw[dashed] (1,-3) -- (-1,-3); 
\draw (-1,-3) -- (-0.5,-3.5);
\draw[dashed] (-0.5,-3.5) -- (0,-4);
\draw (0,-4) -- ++(-0.5,-0.5) -- +(0,4);
\draw (2,-4) -- ++(0.7,-0.3)--+(0,4);
\draw[dashed] (1,-3) -- ++(0.5,0.5)--+(0,0.3);
\draw (1.5,-2.2) -- (1.5,-0.8);
\draw[dashed] (1.5,-0.8) -- (1.5,0.5);
\draw (1.5,0.5)-- (1.5,1.5 );
\draw (1.5,0.5) -- (1.5,1.5); 
\draw (-1,-3) -- ++(-0.7,0.3)--+(0,4);
\draw[dashed] (1,-3) ..controls +(0,0.8) and +(-0.1,0) .. ++(0.5,0.8).. controls +(0.1,0) and +(0,1.5) .. (2,-4);
\draw (-1,-3) ..controls +(0,0.8) and +(-0.1,0) .. ++(0.5,0.8).. controls +(0.1,0) and +(0,1.5) .. (0,-4);
\draw (1,1) ..controls +(0,-1.5)  and +(-0.1,0) .. (1.5,-0.8).. controls +(0.1,0) and +(0,-0.8) .. (2,0);
\draw (-1,1) ..controls +(0,-1.5) and +(-0.1,0) .. (-0.5,-.8).. controls +(0.1,0) and +(0,-0.8) .. (0,0);
\draw (-0.5,-0.8)-- +(2,0);
\draw (-0.5,-2.2)-- +(2,0);
\end{scope} 
\end{tikzpicture}

\vspace{0.7cm}
\begin{tikzpicture}[scale=0.5]
\input{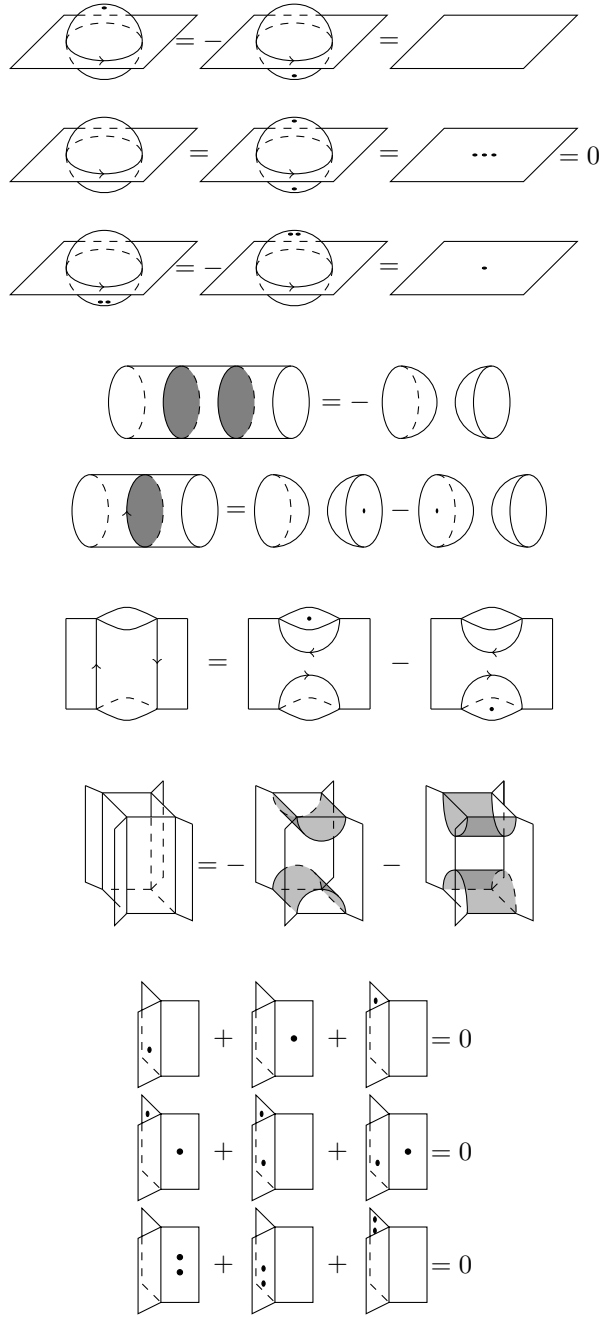} 
\end{tikzpicture}
  \caption{Local relations for 2-morphism in $\mathbb{WT}$. The first 3 lines are called bubbles relations, the 2 next are called bamboo relations, the one after digon relation, then we have the square relation and the 3 last ones are the dots migration relations.}
  \label{fig:localrel}
\end{figure}

This method allows us to define a new graded 2-category $\mathbb{WT}$. Its objects and its 1-morphisms are the ones of the 2-category $\mathcal{WT}$ while its 2-morphisms-spaces are the ones of $\mathcal{WT}$ mod out like in the last paragraph. One should notice that a $(w_b,w_t)$-foam can always be deformed into a $(\mathrm{tr}(\bar{w_b}w_t),\emptyset)$-foam and vice-versa. Khovanov's results restated in this language give that if $w_b$ and $w_t$ are $(\epsilon_0,\epsilon_1)$-web-tangles, the graded dimension of $\hom_{\mathbb{WT}}(w_t,w_b)$ is given by $\kup{\mathrm{tr}(\bar{w_b}w_t)}\cdot q^{l(\epsilon_0)+l(\epsilon_1)}$. Note that when $\epsilon_1=\emptyset$, there is no need to take the closure, because $w_b\bar{w_t}$ is already a closed web. The shift by $l(\epsilon_0)+l(\epsilon_1)$ comes from the fact that $\chi(\mathrm{tr}(\bar{w_b}w_t)) = \chi(w_t)+\chi(w_b) - (l(\epsilon_0)+l(\epsilon_1))$.

\begin{prop}\label{prop:relFR}
We consider the set \FR{} of local relations which consists of:
\begin{itemize}
\item the surgery relation,
\item the evaluations of the dotted spheres and of the dotted theta-foams,
\item the square relations and the digon relations (see figure~\ref{fig:localrel}).
\end{itemize}
We call them the \emph{foam relations} or relations \FR, then for any closed web $w$ $\F(w)$ is isomorphic to $\mathcal{G}(w)$ modded out by \FR.
\end{prop}

\subsection{The Kuperberg-Khovanov algebra $K^\epsilon$}
\label{sec:algebra-k_s}
We want to extend the Khovanov TQFT to the 0-dimensional objects \ie to build a 2-functor from the 2-category $\mathcal{WT}$ to the 2-category of algebras. We follow the methodology of \cite{MR1928174} and we start by defining the image of the $0$-objects~: they are the algebras $K^\epsilon$. This can be compared with \cite{2012arXiv1206.2118M}.
\begin{dfn}
  Let $\epsilon$ be an admissible finite sequence of signs. We define $\tilde{K}^\epsilon$ to be the full sub-category of $\hom_{\mathbb{WT}}(\emptyset,\epsilon)$ whose objects are non-elliptic $\epsilon$-webs. This is a graded $\QQ$-algebroid. We recall that a $k$-algebroid is a $k$-linear category. This can be seen as an algebra by setting~:
\[K^\epsilon = \bigoplus_{(w_b,w_t)\in (\mathrm{ob}(\tilde{K}^\epsilon))^2} \hom_{\mathbb{WT}}(w_b,w_t)\]
and the multiplication on $K^\epsilon$ is given by the composition of morphisms in $\tilde{K}_\epsilon$ whenever it's possible and by zero when it's not. We will denote $\tensor[_{w_1}]{K}{^{\epsilon}_{w_2}}\eqdef \hom_{\mathbb{WT}}(w_2,w_1)$. This is a unitary algebra because of proposition \ref{prop:NEFinite}. The unite element is $\sum_{w\in \mathrm{ob}(\tilde{K}_\epsilon)} 1_w$. Suppose $\epsilon$ is fixed, for $w$ a non-elliptic $\epsilon$-web, we define $P_w$ to be the left $K^\epsilon$-module~:
\[
P_w=\bigoplus_{w'\in\mathrm{ob}(\tilde{K}_\epsilon)}\hom_{\mathbb{WT}}(w,w') = \bigoplus_{w'\in\mathrm{ob}(\tilde{K}_\epsilon)} \tensor[_{w'}]{K}{^\epsilon_{w}}.
\]The structure of module is given by composition on the left.
\end{dfn}
For a given $\epsilon$, the modules $P_w$ are all projective and we have the following decomposition in the category of left $K^\epsilon$-modules~:
\[
K^\epsilon=\bigoplus_{w\in \mathrm{ob}(\tilde{K}_\epsilon)} P_w.
\]
\begin{prop}
Let $\epsilon$ be an admissible sequence of signs, and $w_1$ and $w_2$ two non-elliptic $\epsilon$-webs, then the graded dimension of $\hom_{K^\epsilon}(P_{w_1}, P_{w_2})$ is given by $\kup{(\bar{w_1}w_2)}\cdot q^{l(\epsilon)}$.
\end{prop}
\begin{proof}
  An element of  $\hom_{K^\epsilon}(P_{w_1}, P_{w_2})$ is completely determined by the image of $1_{w_1}$ and this element can be sent on any element of $\hom_{\mathbb{WT}}(P_{w_2}, P_{w_1})$, and $\dim_q(\hom_{\mathbb{WT}}(P_{w_1}, P_{w_2}))=\kup{(\bar{w_1}w_2)}\cdot q^{l(\epsilon)}$.
\end{proof}
In what follows we will prove that some modules are indecomposable, they all have finite dimension over $\QQ$ hence it's enough to show that their rings of endomorphisms contain no non-trivial idempotents. It appears that an idempotent must have degree zero, so we have the following lemma~:
\begin{lem}\label{lem:monic2indec}
  If $w$ is a non-elliptic $\epsilon$-web such that $\kup{\bar{w}w}$ is monic of degree $l(\epsilon)$, then the graded $K^\epsilon$-module $P_w$ is indecomposable.
\end{lem}
\begin{proof}
  This follows from previous discussion~: if $\hom_{K^\epsilon}(P_w,P_w)$ contained a-non trivial idempotent, there would be at least two linearly independent elements of degree 0, but $\dim((\hom_{\mathbb{WT}}(P_{w}, P_{w})_0) = a_{-l(\epsilon)}$ if we write $\kup{\bar{w}w}=\sum_{i\in \ZZ}a_iq^i$ but as $\kup{\bar{w}w}$ is symmetric (in $q$ and $q^{-1}$) of degree $l(\epsilon)$ and monic, $a_{-l(\epsilon)}$ is equal to 1 and this is a contradiction.
\end{proof}
We have a similar lemma to prove that two modules are not isomorphic.
\begin{lem}\label{lem:0monic2noiso}
  If $w_1$ and $w_2$ are two non-elliptic $\epsilon$-webs such that $\kup{\bar{w_1}w_2}$ has degree strictly smaller than $l(\epsilon)$, then the graded $K^\epsilon$-modules $P_{w_1}$  and $P_{w_2}$ are not isomorphic.
\end{lem}
\begin{proof}
  If they were isomorphic, there would exist two morphisms $f$ and $g$ such that $f\circ g=1_{P_{w_1}}$ and therefore $f\circ g$ would have degree zero. The hypothesis made implies that $f$ and $g$ (because $\kup{\bar{w_1}w_2} = \kup{\bar{w_2}w_1}$) have positive degree so that the degree of their composition is as well positive.
\end{proof}
\begin{req}
  The way we constructed the algebra $K^\epsilon$ is very similar to the construction of $H^n$ in \cite{MR1928174}. Using the same method we can finish the construction of a $0+1+1$ TQFT~:
  \begin{itemize}
  \item If $\epsilon$ is an admissible sequence of signs, then $\F(\epsilon) = K^\epsilon$.
  \item If $w$ is a $(\epsilon_1,\epsilon_2)$-web-tangle with $\epsilon_1$ and $\epsilon_2$ admissible, then 
\[\F(w) = \bigoplus_{\substack{u\in \mathrm{ob}(\tilde{K}^{\epsilon_1}) \\ v\in \mathrm{ob}(\tilde{K}^{\epsilon_2})}} \F(\bar{u}wv),
\] and it has a structure of graded $K^{\epsilon_1}$-module-$K^{\epsilon_2}$. Note that if $w$ is a non-elliptic $\epsilon$-web, then $\F(w)=P_w$.
\item If $w$ and $w'$ are two $(\epsilon_1,\epsilon_2)$-web-tangles, and $f$ is a $(w,w')$-foam, then we set 
\[
\F(f) = \sum_{\substack{u\in \mathrm{ob}(\tilde{K}^{\epsilon_1}) \\ v\in \mathrm{ob}(\tilde{K}^{\epsilon_2})}} \F(\tensor[_{\bar{u}}]{f}{_v}),
\] where $\tensor[_{\bar{u}}]{f}{_v}$ is the foam $f$ with glued on its sides $\bar{u}\times[0,1]$ and $v\times [0,1]$. This is a map of  graded $K^{\epsilon_1}$-modules-$K^{\epsilon_2}$.
  \end{itemize}
We encourage the reader to have a look at this beautiful construction for the $sl_2$ case in \cite{MR1928174}. 
\end{req}

In the $\sll_2$ case the classification of projective indecomposable modules is fairly easy, and a analogous result, would state in our context that the projective indecomposable modules are exactly the modules associated with non-elliptic webs. However we have:

\begin{prop}[\cite{MR2457839}, see\cite{LHRThese} for details]\label{prop:Pwdec}
  Let $\epsilon$ be the sequence of signs:  $(+,-,-,+,+,-,-,+,+,-,-,+)$, and let $w$ and $w_0$ be the two $\epsilon$-webs given by figure~\ref{fig:thewebw}. Then the web-module $P_w$ is decomposable and admits $P_{w_0}$ as a direct factor.
\end{prop}
\begin{figure}[ht]
  \centering
\begin{tikzpicture}[scale= 0.7]
  \input{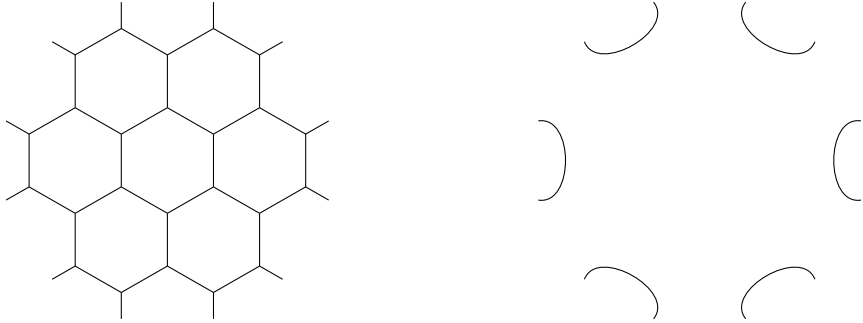}
\end{tikzpicture}
  \caption{The $\epsilon$-webs $w$ (on the left) and $w_0$ (on the right), to fit in formal context of the 2-category one should stretch the outside edges to horizontal line below the whole picture, we draw it this way to enjoy more symmetry. To simplify we didn't draw the arrows.}
  \label{fig:thewebw}
\end{figure}

\label{cha:red-graphs}
In \cite{LHR1}, one gives a sufficient condition for a web-module to be indecomposable. All the argumentation relies on the computation of the dimension of the space of the degree 0 endomorphisms of web-modules: in fact, when for a web $w$, this space has dimension 1, then the web-module $P_w$ is indecomposable. Translated in terms of Kuperberg bracket, it says (see as well lemma~\ref{lem:monic2indec}): 
\begin{quotation}
\noindent\emph{If $w$ is an $\epsilon$-web such that $\kup{\overline{w}w}$ is monic of degree $l(\epsilon)$, then the $K^\epsilon$-module $P_w$ is indecomposable.}
\end{quotation}
The aim of this chapter is to prove the converse. This will give the following characterisation of indecomposable web-modules:
\begin{thm}
  Let $w$ be an $\epsilon$-web. The $K^\epsilon$-module $P_w$ is indecomposable if and only if $\kup{\overline{w}w}$ is monic of degree $l(\epsilon)$. Furthermore if the $K^\epsilon$-module $P_w$ is decomposable it contains another web-module as a direct factor.
\end{thm}

The proof relies on some combinatorial tools called red graphs. In a first part we give an explicit construction (in terms of foams) of a non-trivial idempotent associated to a red graph. In a second part we show that when an $\epsilon$-web $w$ is such that $\kup{\overline{w}w}$ is not monic of degree $l(\epsilon)$, then it contains a red graph.

\section{Red graphs}
\label{sec:redgraph}
\subsection{Definitions}
\label{sec:defRG}
The red graphs are sub-graphs of the dual graphs webs, we recall here the definition of a dual graph. For an introduction to graph theory we refer to~\cite{MR0256911} and \cite{MR2368647}.\marginpar{trouver une bonne référence pour la theorie des graphe topologique} 
\begin{dfn}
  Let $G$ be a plane graph (with possibly some vertex-less loops), we define \emph{the dual graph $D(G)$ of $G$} to be the abstract graph given as follows:
  \begin{itemize}
  \item The set of vertices $V(D(G))$ of $D(G)$ is in one-one correspondence with the set of connected components of $\RR^2\setminus G$ (including the unbounded connected component). Such connected component are called \emph{faces}. 
  \item The set of edges of $D(G)$ is in one-one correspondence with the set of edges of $G$ (in this construction, vertex-less loops are not seen as edges). If an edge $e$ of $G$ is adjacent to the faces $f$ and $g$ (note that $f$ may be equal to $g$ if $e$ is a bridge), then the corresponding edge $e'$ in $D(G)$ joins $f'$ and $g'$, the vertices of $D(G)$ corresponding to $f$ and $g$.
  \end{itemize}
\end{dfn}
Note that in general the faces need not to be diffeomorphic to disks. It is easy to see that the dual graph of a plane graph is planar:  we place one vertex inside each face, and we draw an edge $e'$ corresponding to $e$ so that it crosses $e$ exactly once and it crosses no other edges of $G$. Such an embedding of $D(G)$ is a \emph{plane dual} of the graph $G$ (see figure~\ref{fig:exdualgrpah}).
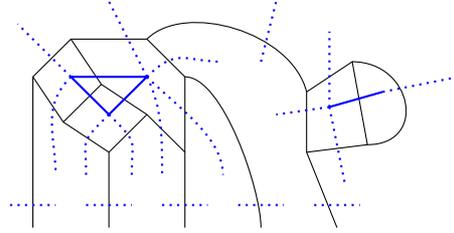
\begin{figure}[ht]
  \centering
  \begin{tikzpicture} 
    \begin{scope}[xshift = 6cm, yscale = {1}, xscale={1},decoration={markings, mark=at
     position 0.5 with {\arrow{>}}},postaction={decorate}]
\draw (0,0) -- (0,2) -- (0.5,2.5) -- (1.5, 2.5) -- (2,2) -- (2,1) -- (2,0);
\draw (1,0) --(1,1)-- (0.4,1.4) -- (0,2);
\draw (0.5,2.5) -- (0.9,1.9) -- (1.5, 1.5) -- (2,1);
\draw (0.4,1.4) -- (0.9,1.9);
\draw (1,1) -- (1.5, 1.5);
\draw (2,2) .. controls +(0.5,0) and +(0,0.5) .. (3,0);
\draw (1.5, 2.5) .. controls +(0.5,0.5) and +(-0.2,0.8) .. (3.6,1.8) -- (3.6,1) -- (4,0);
\draw (3.6,1.8) -- (4.2,2.2) -- (4.4, 1.1) -- (3.6,1);
\draw (4.2,2.2) .. controls +(0.8,0) and + (0.8,0) .. (4.4,1.1);  
\coordinate (A) at (0.5,2);
\coordinate (B) at (1,1.5);
\coordinate (C) at (1.5,2);
\coordinate (D) at (4.6, 1.8);
\coordinate (E) at (3.9, 1.6);
\draw[color = blue, thick] (A) -- (B) -- (C)-- (A);
\draw[color = blue, thick] (D) -- (E);
\draw[color = blue, dotted, thick] (A) -- +(-0.7,0.7);
\draw[color = blue, dotted, thick] (A) .. controls +(-0.4,-0.4) and +(0,0.5) .. (0.3,0.7);
\draw[color = blue, dotted, thick] (B) .. controls +(-0.4,-0.4) and +(0,0.5) .. (0.7,0.7);
\draw[color = blue, dotted, thick] (B) .. controls +(+0.4,-0.4) and +(0,0.5) ..
(1.3,0.7);
\draw[color = blue, dotted, thick] (C) .. controls +(+0.2,-0.4) and +(0,0.5)..  (1.7,0.7);
\draw[color = blue, dotted, thick] (C) .. controls +(+0.4,-0.4) and +(0,0.5)..  (2.5,0.7);
\draw[color = blue, dotted, thick] (C) .. controls +(+0.4,+0.4) and +(-0.5,0) .. +(1,0.2);
\draw[color = blue, dotted, thick] (C) -- +(-0.5,1);
\draw[color = blue, dotted, thick] (D) -- +(1,0.2);
\draw[color = blue, dotted, thick] (E) -- +(-0.7,-0.1);
\draw[color = blue, dotted, thick] (E) -- +(0,1);
\draw[color = blue, dotted, thick] (E) -- +(0.2,-1);
\draw[color = blue, dotted, thick] (-0.3,0.3) -- (0.3,0.3);
\draw[color = blue, dotted, thick] (0.7,0.3) -- (1.3,0.3);
\draw[color = blue, dotted, thick] (1.7,0.3) -- (2.3,0.3);
\draw[color = blue, dotted, thick] (2.7,0.3) -- (3.3,0.3);
\draw[color = blue, dotted, thick] (3.7,0.3) -- (4.3,0.3);
\draw[color = blue, dotted, thick] (3,2.2) -- (3.2,3);
\end{scope}
  \end{tikzpicture}
  \caption{In black an $\epsilon$-web $w$ and in blue the dual graph of $w$. The dotted edges are all meant to belong to $D(w)$ and to reach the vertex $u$ corresponding to the unbounded component of $\RR^2\setminus w$.}
\label{fig:exdualgrpah}
\end{figure}

\begin{dfn}\label{dfn:red-graph}
  Let $w$ be an $\epsilon$-web, a \emph{red graph} for $w$ is a non-empty subgraph $G$ of $D(w)$ such that:
  \begin{enumerate}[(i)]
  \item\label{item:dfnRG1} All faces belonging to $V(G)$ are diffeomorphic to disks. In particular, the unbounded face is not in $V(G)$.
  \item\label{item:dfnRG2} If $f_1$, $f_2$ and $f_3$ are three faces of $w$ which share together a vertex, then at least one of the three does not belong to $V(G)$.
  \item\label{item:dfnRG3} If $f_1$ and $f_2$ belongs to $V(G)$ then every edge of $D(w)$ between $f_1$ and $f_2$ belongs to $E(G)$, \ie $G$ is an induced subgraph of $D(w)$.
  \end{enumerate}
  If $f$ is a vertex of $G$ we define $\ed(f)$, \emph{the external degree of $f$}, by the formula: \[ \ed(f) = \deg_{D(w)}(f) - 2\deg_G(f).\]
\end{dfn}
\begin{figure}[ht]
  \centering
  \begin{tikzpicture}[scale=1]
     \begin{scope}
\foreach \n/\a in {a/30, b/60, c/120, d/150, e/210, f/240, g/300, h/330}
{
\draw (\a:0.5) -- (\a:1.5);
}
\draw (30:0.5) -- (60:0.5) -- (90:0.5) -- (120:0.5) --  (150:0.5) --  (210:0.5) --  (240:0.5) --  (270:0.5) --  (300:0.5) --  (330:0.5) -- cycle;
\draw (30:1.5) -- (60:1.5) -- (90:1.5) -- (120:1.5)  -- (127:1.5) .. controls (135:1.2) and (135:1.2) .. (143:1.5) --  (150:1.5) --  (170:1.5) --  (190:1.5)--  (210:1.5) --  (240:1.5) --  (270:1.5) --  (300:1.5) --  (310:1.5)--  (320:1.5) --  (330:1.5) --  (350:1.5)--  (10:1.5) -- cycle;
\draw (127:1.5) .. controls (135:1.8) and (135:1.8) .. (143:1.5);
\draw (0,0.5) -- (0,-0.5);
\draw (170:1.5) .. controls +(-1.6,0.1) and +(0, 0.5) .. (-1, -2);
\draw (190:1.5) .. controls +(-1,-0.1) and +(0, 0.5) .. (-0.5, -2); 
\draw (270:1.5) -- (0,-2);
\draw (350:1.5) .. controls +(0.4,-0.1) and +(0, 0.5) .. (1.5, -2);
\draw (10:1.5) .. controls +(0.4,+0.1) and +(0, 0.5) .. (2, -2);
\draw (310:1.5) .. controls +(0.15,-0.2) and +(0, 0.3) .. (0.5, -2);
\draw (320:1.5) .. controls +(0.2,-0.15) and +(0, 0.3) .. (1, -2); 
\draw (90:1.5) .. controls +(0.7,1) and +(0,3) .. (2.5,-2);
\foreach \n/\a in {a/0, b/45, d/90, e/315, f/270, c/180}
{
\filldraw[red] (\a:1) circle (2pt);
}
\filldraw[red] (180:0.3) circle (2pt);
\filldraw[red] (135:1.5) circle (2pt);
\draw[very thick, red] (180:1) -- (180:0.3) -- (90:1) -- (45:1) -- (0:1) -- (315:1) -- (270:1)--(180:0.3);
\end{scope}
  \end{tikzpicture}
  \caption{Example of a red graph.}
  \label{fig:exRG}
\end{figure}
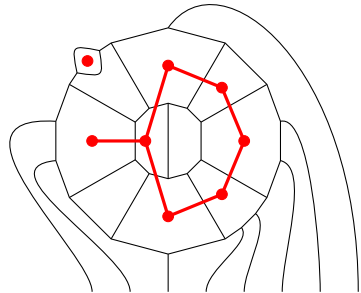

\begin{req}\label{rk:evenextdeg}
  Note that the external degree of a face $f$ is always an even number because $w$ being bipartite, all cycles are of even length and hence $\deg_{D(w)}$ is even.
\end{req}

Let $G$ be a red graph for $w$, then if on the web we colour the faces which belongs to $V(G)$, then the external degree of a face $f$ in $V(G)$ is the number of half-edges of $w$ which touch the face $f$ and lie in the uncoloured region. These half-edges are called the \emph{grey half-edges} of $f$ in $G$ or of $G$ when we consider the set of all grey half-edges of all vertices of $G$. See figure~\ref{fig:halfgrey}.
\begin{figure}[ht]
  \centering
  \begin{tikzpicture}[scale=0.6]
    \input{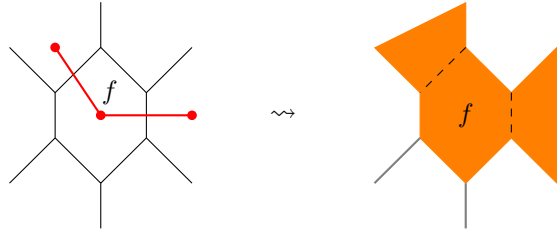}
  \end{tikzpicture}
  \caption{Interpetation of the external degree in terms of grey half-edges. On the left, a portion of a web $w$ with a red graph; on the right, the same portion of $w$ with the vertices of $G$ orange-coloured. The external degree of $f$ is the number of half edges touching $f$ which are not orange. In our case $\ed(f)=2$.}
  \label{fig:halfgrey}
\end{figure}

An oriented red graph is a red graph together with an orientation, \emph{a priori} there is no restriction on the orientations, but as we shall see just a few of them will be relevant to consider.
\begin{dfn}\label{dfn:index-red-graph}
  Let $w$ be an $\epsilon$-web, $G$ be a red graph for $w$ and $o$ an orientation for $G$, we define the level $i_o(f)$ (or $i(f)$ when this is not ambiguous) of a vertex $f$ of $G$ by the formula:
\begin{align*}
i_o(f) &\eqdef  2 - \frac{1}{2}\ed(f) - \#\{\textrm{edges of $G$  pointing to $f$}\} \\
&= 2- \frac{\deg_{D(w)}}2 + \#\{\textrm{edges of $G$  pointing away from $f$}\}
\end{align*}
and the level $I(G)$ of $G$ is the sum of levels of all vertices of $G$. 
\end{dfn}
\begin{req}\label{req:formindex}
The level is an integer because of remark~\ref{rk:evenextdeg}.
Note that the level of $G$ does not depend on the orientation of $G$ and we have the formula:\[ I(G) = 2\#V(G) - \#E(G) - \frac{1}{2}\sum _{f\in v(G)}\ed(f).\]
\end{req}

\begin{dfn}\label{dfn:RG-admissible}
A red graph is \emph{admissible} if one can choose an orientation such that for each vertex $f$ of $G$ we have: $i(f)\geq 0$. Such an orientation is called \emph{a fitting orientation}. An admissible red graph $G$ for $w$ is \emph{exact} if $I(G)=0$. 
\marginpar{For a web $w$ we define $M(w) = \max I(G)$ where the maximum is taken over all admissible red graph $G$ for $w$.???}
\end{dfn}

\begin{dfn}\label{dfn:paired-RG}\marginpar{should we put the orientation ?}
  Let $w$ be an $\epsilon$-web and $G$ be a red graph for $w$. A \emph{pairing} of $G$ is a partition of the grey half-edges of $G$ into subsets of 2 elements such that for any subset the two half-edges touch the same face $f$, and one points to $f$ and the other one points away from $f$. A red graph together with a pairing is called a \emph{paired red graph.}
\end{dfn}

\begin{dfn}
  A red graph $G$ in an $\epsilon$-web $w$ is \emph{fair} (\resp \emph{nice}) if for every vertex $f$ of $G$ we have $\ed(f)\leq 4$ (\resp $\ed(f)\leq 2$). 
\end{dfn}

\begin{lem}\label{lem:RGadm2fair}
  If $G$ is an admissible red graph in an $\epsilon$-web $w$, then $G$ is fair.
\end{lem}
\begin{proof}
 It follows directly from the definition of the level.
\end{proof}
\begin{cor}\label{cor:RGinNEhaveVs}
  Let $w$ be a non-elliptic $\epsilon$ web, then if $G$ is an admissible red graph for $w$ then it has at least two edges.
\end{cor}
\begin{proof}
  If G would contain just one vertex $f$, this would have external degree greater or equal to 6, contradicting lemma~\ref{lem:RGadm2fair}. We can actually show that such a red graph contains at least 6 vertices (see corollary~\ref{cor:no-tree} and proposition~\ref{prop:largecycle}).
\end{proof}
\begin{req}\label{req:number-of-pairing}
If a red graph $G$ is nice, there is only one possible pairing. If it is fair the number of pairing is $2^n$ where $n$ denote the number of vertices with external degree equal to 4. 

If on a picture one draws together a web $w$ and a red graph $G$ for $w$, one can encode a pairing of $G$ on the picture by joining\footnote{We impose that $w$ intersect the dashed lines only at their ends.} with dashed line the paired half-edges. Note that if $G$ is fair it's always possible to draw disjoint dashed lines (see figure~\ref{fig:pairings} for an example).

\end{req}
\begin{figure}[ht]
  \centering
  \begin{tikzpicture}[scale=0.7]
    \begin{scope}[yscale = {1}, xscale={1},decoration={markings, mark=at
     position 0.5 with {\arrow{>}}},postaction={decorate}]
\draw (-3.5,-1) .. controls +(0,2.5) and +(0,1) .. (-0.5, 0.5) -- (0.5, 0.5) .. controls +(0,1) and +(0,2.5).. (3.5, -1);
\draw (-3, -1) .. controls +(0,1) and +(-1, 0) .. (-2, 0.5) -- (2, 0.5).. controls +(+1,0) and +(0,1) .. (3, -1);
\draw (-2.5, -1) .. controls +(0, 0.5) and +(-0.5, 0) ..  (-2, -0.5) -- (2, -0.5) ..controls +(+0.5,0) and +(0,0.5) .. (2.5, -1);
\draw (-2, 0.5) -- (-2, -0.5);
\draw (-1, 0.5) -- (-1, -0.5);
\draw (1, 0.5) -- (1, -0.5);
\draw (2, 0.5) -- (2, -0.5);
\draw (-1, 0.5) -- (-1, -0.5);
\draw (-0.5,-1) -- +(0,0.5);
\draw (0.5,-1) -- +(0,0.5);
\draw[dashed] (-0.5, -0.5) -- +(0, 1);
\draw[dashed] (+0.5, -0.5) -- +(0, 1);
\draw[dashed] (-2, -0.5) .. controls +(0.5, 0) and +(0.5,0) .. +(0,1);
\draw[dashed] (2, -0.5) .. controls +(-0.5, 0) and +(-0.5,0) .. +(0,1);
\draw[fill, color =red] (0,0) circle (3pt);
\draw[fill, color =red] (-1.5,0) circle (3pt);
\draw[fill, color =red] (1.5,0) circle (3pt);
\draw[very thick, color =red] (-1.5,0) -- (1.5,0);
\end{scope}

\begin{scope}[xshift = 10cm, yscale = {1}, xscale={1},decoration={markings, mark=at
     position 0.5 with {\arrow{>}}},postaction={decorate}]
\draw (-3.5,-1) .. controls +(0,2.5) and +(0,1) .. (-0.5, 0.5) -- (0.5, 0.5) .. controls +(0,1) and +(0,2.5).. (3.5, -1);
\draw (-3, -1) .. controls +(0,1) and +(-1, 0) .. (-2, 0.5) -- (2, 0.5).. controls +(+1,0) and +(0,1) .. (3, -1);
\draw (-2.5, -1) .. controls +(0, 0.5) and +(-0.5, 0) ..  (-2, -0.5) -- (2, -0.5) ..controls +(+0.5,0) and +(0,0.5) .. (2.5, -1);
\draw (-2, 0.5) -- (-2, -0.5);
\draw (-1, 0.5) -- (-1, -0.5);
\draw (1, 0.5) -- (1, -0.5);
\draw (2, 0.5) -- (2, -0.5);
\draw (-1, 0.5) -- (-1, -0.5);
\draw (-0.5,-1) -- +(0,0.5);
\draw (0.5,-1) -- +(0,0.5);
\draw[dashed] (-0.5, -0.5) .. controls +(0,0.3) and +(0,0.3) .. (+0.5,-0.5);
\draw[dashed] (-0.5, +0.5) .. controls +(0,-0.3) and +(0,-0.3) .. (+0.5,+0.5);
\draw[dashed] (-2, -0.5) .. controls +(0.5, 0) and +(0.5,0) .. +(0,1);
\draw[dashed] (2, -0.5) .. controls +(-0.5, 0) and +(-0.5,0) .. +(0,1);
\draw[fill, color =red] (0,0) circle (3pt);
\draw[fill, color =red] (-1.5,0) circle (3pt);
\draw[fill, color =red] (1.5,0) circle (3pt);
\draw[very thick, color =red] (-1.5,0) -- (1.5,0);
\end{scope}
  \end{tikzpicture}
  \caption{A web $w$, a red graph $G$ and the two possible pairings for $G$.}
  \label{fig:pairings}
\end{figure}
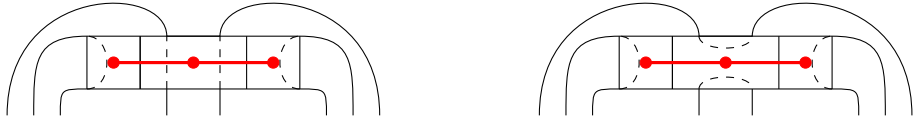
The rest of the chapter (respectively in section~\ref{sec:RG2idem} and \ref{sec:kup2RG}) will be devoted to show the following two theorems:
\begin{thm}\label{thm:RG2idempotent}
To every exact paired red graph of $w$ we can associate a non trivial idempotent of $Hom_{K^\epsilon}(P_w,P_w)$. Further more the direct factor associated with the idempotent is a web-module.
\end{thm}
\begin{thm}\label{thm:on-monic-2-RG}
  Let $w$ be a non-elliptic $\epsilon$-web, then if  $\kup{\overline{w}w}$ is non-monic or have degree bigger than $l(\epsilon)$, then there exists an exact red graph for $w$, therefore the $K^\epsilon$-module $P_w$ is decomposable.
\end{thm}

\subsection{Combinatorics on red graphs}
\label{sec:comb-red-graphs}
On the one hand, the admissibility of a red graph relies on the local non-negativity of the level for some orientation, on the other hand the global level $I$ does not depend on the orientation. However, it turns out that the existence of admissible red graph $G$ for an $\epsilon$-web $w$ can be understood thanks to $I$ in some sense: 
\begin{prop}\label{prop:2admissible}
  Let $w$ be an $\epsilon$-web, suppose that there exists  $G$ a red graph for $w$ such that $I(G)\geq 0$, then there exists an admissible red graph $\widetilde{G}$ for $w$ such that $I(\widetilde{G})\geq I(G)$.\marginpar{mettre non-empty dans la defnition de red graph}
\end{prop}
\begin{proof} If $G$ is already admissible, there is nothing to show, hence we suppose that $G$ is not admissible.  Among all the orientations for $G$, we choose one such that $\sum_{f\in V(G)} |i(f)|$ is minimal, we denote is by $o$. From now on $G$ is endowed with this orientation. As $G$ is not admissible there exists some vertices with negative level and some with positive level. 

We first show that there is no oriented path from a vertex $f_p$ with $i_o(f_p)> 0$ to a vertex $f_n$ with $i_o(f_n)<0$. Suppose there exists $\gamma$ such a path. Let us inspect $o'$ the orientation which is the same as $o$ expect along the path $\gamma$ where it is reversed. For all vertices $f$ of $G$ but $f_p$ and $f_n$, we have $i_o(f)=i_{o'}(f)$ (for all the vertices inside the path, the position of the edges pointing to them is changed, but not their number), and we have:
\[
i_{o'}(f_p) = i_o(f_p)-1 \quad i_{o'}(f_n) = i_o(f_n)+1.
\] But then $\sum_{f\in V(G)} |i_{o'}(f)|$ would be strictly smaller than $\sum_{f\in V(G)} |i_o(f)|$ and this contradicts that $o$ is minimal.

We consider $(\widetilde{G}, \tilde{o})$ the induced oriented sub-graph of $(G,o)$ with set of vertices $V(\widetilde{G})$ equal to the vertices of $G$ which can be reach from a vertex with positive level by an oriented path. This set is not empty since it contains the vertices with positive degree. It contains no vertex with negative degree. For all vertices of $\widetilde{G}$, we have:
\begin{align*} i_{\tilde{o}}(f) &=  2 - \frac{\deg_{D(w)}(f)}2 + \#\{\textrm{edges of $G$  pointing away from  $f$ in $\widetilde{G}$}\} \\ &=
2 - \frac{\deg_{D(w)}(f)}2 + \#\{\textrm{edges of $G$  pointing away from $f$ in $G$}\} \\
&= i_o(f).\end{align*} 
The second equality holds because if $f$ is in $V(\widetilde{G})$ all the edges in $E(G)\setminus E(G')$. $G$ which are not in $\widetilde{G}$ point to $f$ by definition of $\widetilde{G}$. This shows that $\widetilde{G}$ is admissible and $I(\widetilde{G})>I(G)$.
\end{proof}

\begin{lem}\label{lem:RGinNE2niceRG}
  Let $w$ be a non-elliptic web, suppose that it contains a red graph of level $k$, then it contains an admissible nice red graph of level at least $k$.
\end{lem}
\begin{proof}
  We consider $G$ a red graph of $w$ of level $k$. Thanks to lemma \ref{prop:2admissible} we can suppose that it is admissible. We can take a minimal red graph $G$ for the property of being of level at least $k$ and admissible. The graph $G$ is endowed with a fitting orientation. Now suppose that it is not nice, it means that there exists a vertex $v$ of $G$ which have exterior degree equal to 4. But $G$ being admissible all the edges of $G$ adjacent to $v$ point out of $v$, so that we can remove $v$ \ie we can consider then induced sub-graph $G'$ with all the vertex of $G$ but $v$ with the induced orientation. Then it is admissible, with the same level, hence $G$ is not minimal, contradiction.
\end{proof}

For a non-elliptic $\epsilon$-web, the existence of an exact red graph may appear as an exceptional situation between the case where there is no admissible red graph and the case where all admissible red graphs are non-exact. The aim of the rest of this section is to show the proposition~\ref{prop:exist-exact} which indicates that this is not the case. On the way we state some small results which are not directly useful for the proof but may alight what red-graphs look like.
\begin{prop}\label{prop:exist-exact}
  Let $w$ be a non-elliptic $\epsilon$-web. If there exists an admissible red graph for $w$ then there exists an exact red graph for $w$.
\end{prop}
\begin{dfn}\label{dfn:subred}
  Let $w$ be an $\epsilon$-web, and $G$ and $G'$ two admissible red graphs for $w$. We say that $G'$ is a \emph{red sub-graph} of $G$ if $V(G')\subset V(G')$. We denote by $\G(G)$ the set of all admissible red sub-graphs. It is endowed with the order given by the inclusion of sets of vertices. We say that $G$ is \emph{minimal} if $\G(G) = \{G\}$.
\end{dfn}
Note that a red sub-graph is an induced sub-graph and that a minimal red-graph is connected.

\begin{lem}\label{lem:no_cut}
   Let $w$ be an $\epsilon$-web and $G$ a minimal admissible red graph endowed with a fitting orientation. There is no non-trivial partition of $V(G)$ into two sets $V_1$ and $V_2$ such that for each vertex $v_1$ in $V_1$ and each vertex $v_2$ in $V_2$ every edge between $v_1$ and $v_2$ is oriented from $v_1$ to $v_2$.
\end{lem}
\begin{proof}
  If there were a such a partition, we could consider the red sub-graph $G'$ with $V(G') =V_2$. For every vertex in $V_2$ the level is the same in $G$ and in $G'$ and hence, $G'$ would be admissible and $G$ would not be minimal.
\end{proof}
\begin{cor}
  \label{cor:noleaf4minimal} Let $w$ be a $\epsilon$-web and $G$ a minimal admissible red graph for $w$, then the graph $G$ has no leaf\footnote{We mean vertex of degree 1.}. Therefore if it has 2 or more vertices, then it is not a tree.
\end{cor}
\begin{proof}
Indeed, if $v$ were a leaf of $G$, the vertex $v$ would be either a sink or a source, hence $V(G)\setminus\{v\}$ and $\{v\}$ would partitioned $V(G)$ in a way forbidden by lemma~\ref{lem:no_cut}.
\end{proof}
\begin{cor}\label{cor:no-tree}
  If $G$ is an admissible red graph for a non-elliptic $\epsilon$-web $w$, then $G$ is not a tree.
\end{cor}
\begin{proof}
  Consider a minimal red sub-graph of $G$. Thanks to corollaries~\ref{cor:RGinNEhaveVs} and \ref{cor:noleaf4minimal}, it is not a tree, hence $G$ is not a tree.
\end{proof}
\begin{lem}
Let $w$ be an $\epsilon$-web and $G$ a minimal red graph for $w$. If $G$ has more than 2 vertices, then it is nice.
\end{lem}
\begin{proof}
  Suppose that we have a vertex $v$ of $G$ with external degree equal to 4. Consider a fitting orientation for $G$. All edges of $G$ adjacent to $v$ must point out, otherwise the degree of $v$ would be negative. So $v$ would be a sink and, thanks to lemma~\ref{lem:no_cut}, this is not possible.
\end{proof}
\begin{lem}\label{lem:strong-connected}
  Let $w$ be a non-elliptic $\epsilon$-web and $G$ a minimal admissible red graph. If the red graph $G$ is endowed with a fitting orientation, then it is strongly connected. 
\end{lem}
The terms \emph{weakly connected} and \emph{strongly connected} are classical in graph theory the first means that the underlying unoriented graph is connected in the usual sense. The second that for any pair of vertices $v_1$ and $v_2$, there exists an oriented path from $v_1$ to $v_2$ and an oriented path from $v_2$ to $v_1$.
\begin{proof}
Let $v$ be a vertex of $G$, consider the subset $V_v$ of $V(G)$ which contains the vertices of $G$ reachable from $v$ by an oriented path. The sets $V_v$ and $V(G)\setminus V_v$ form a partition of $V(G)$ which must be trivial because of lemma~\ref{lem:no_cut}, but $v$ is in $V_v$ therefore $V_v=V(G)$, this is true for any vertex $v$, and this shows that $G$ is strongly connected.
\end{proof}
\begin{prop}\label{prop:largecycle} 
  If $G$ is a red graph for a non-elliptic $\epsilon$-web $w$, then any (not-oriented)  simple cycle has at least 6 vertices. 
\end{prop}
\begin{proof}Take $C$ a non-trivial simple cycle in $G$. \marginpar{We can suppose that $C$ is minimal. A red graph contains no triangle, this shows that some faces of $w$ must be nested in the cycle $C$.} We consider the collection of faces of $w$ nested by $C$ (this is non empty thanks to condition~(\ref{item:dfnRG3}) of the definition of red graphs). This defines a plane graph $H$. We define $H'$ to be the graph $H$ with the bivalent vertices smoothed (we mean here that if locally $H$ looks like $\vcenter{\hbox{\tikz{\draw (0,0)-- (1,0); \filldraw (0.5,0) circle (1pt);}}}$, then $H'$ looks like $\vcenter{\hbox{\tikz{\draw (0,0) -- (1,0);}}}$). An example of this construction is depicted on figure~\ref{fig:exlargecycle}.

  \begin{figure}[ht]
    \centering
    \begin{tikzpicture}[scale=0.95]
      \input{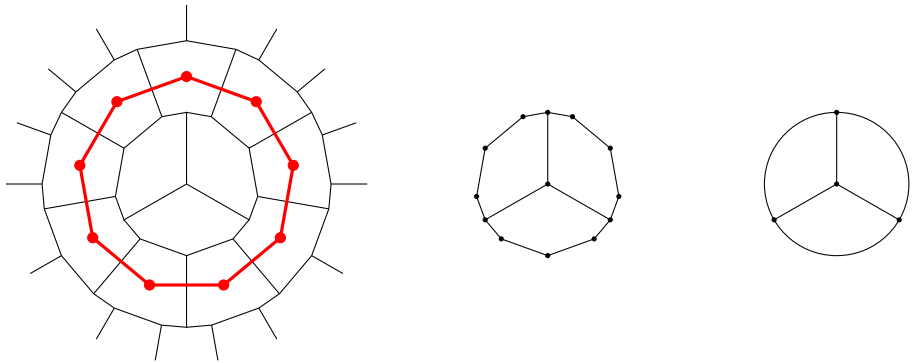}
    \end{tikzpicture}
    \caption{On the left the $\epsilon$-web $w$ and the red graph $G$, in the middle the graph $H$, and on the right, the graph $H'$.}
 \label{fig:exlargecycle}
  \end{figure}

The $\epsilon$-web $w$ being non-elliptic, each face of $H$ has at least $6$ sides. We compute the Euler characteristic of $H'$:
\[\chi(H')= \#F(H')-\#E(H')+\#V(H') =2. \]
As in proposition~\ref{prop:closed2elliptic}, this gives us $\sum_{i\in \NN} F_i(H')(1-\frac i6) =2$  where $F_i(H)$ is the number of faces of $H'$ with $i$ sides. Restricting the sum to $i\leq 5$ and considering $F'_i$ the number of bounded faces, we have:
\[\sum_{i=0}^5 F'_i(H')(6 - i) \geq 6.\]
But the bounded faces of $H'$ with less that 6 sides come from bounded faces of $H$ which have at least 6 sides. The number $n$ of bivalent vertices in $H$ is therefore greater than or equal to $\sum_{i=0}^5 F'_i(H')(6 - i)$ \ie greater than or equal to 6. But $n$ is as well the length of the cycle $C$. \marginpar{faire un exemple ?}
\end{proof}
Note that a cycle in a red graph can have an odd length (as in the example of figure~\ref{fig:exlargecycle}).
\begin{lem}
  Let $G$ be a minimal admissible red graph for a non-elliptic $\epsilon$-web $w$. Then $G$ has at least one vertex with degree 2.
\end{lem}
\begin{proof}
  Suppose  that all vertices of $G$ have degree greater or equal to 3, then the graph $G$  would contain a face with less than 5 sides (this is the same argument than in proposition~\ref{prop:closed2elliptic} which tells that a closed web contains a circle, a digon or a square). But this contradicts lemma~\ref{prop:largecycle}. \marginpar{terminology: oriented graph or digraph ?}
\end{proof}

The proposition~\ref{prop:exist-exact} is a direct consequence of the following lemma:
\begin{lem}\label{lem:minimal2exact}
Let $w$ be a non-elliptic $\epsilon$-web and $G$ a minimal admissible red graph for $w$. Then $G$ is exact.
\end{lem}
\begin{proof}
  We endow $G$ with $o$ a fitting orientation. Suppose $G$ is not exact, then we can find a vertex $f$ with $i_o(f)>0$. 

We first consider the case where $\deg(f)=2$. The $\epsilon$-web $w$ being non-elliptic, $\ed(f)\geq 2$. This shows that the two edges adjacent to $f$ point away from $f$, hence, $f$ is a sink and this contradicts lemma~\ref{lem:no_cut}. 

Now, let us consider the general case. Let $f'$ be a vertex with degree 2. The lemma \ref{lem:strong-connected} implies that there exists $\gamma$ an oriented path from $f$ to $f'$. Let us reverse the orientations of the edges of $\gamma$. We denote by $o'$ this new orientation. Then we have $i_{o'}(f)=i_{o}(f)-1\geq 0$ and $i_{o'}(f')=i_G(f')+1\geq 1$. The levels of all other edges are not changed, hence $o'$ is a fitting orientation, and we are back in the first situation (where $f'$ plays the role of $f$).
\end{proof}

\section{Idempotents from red graphs}
\label{sec:RG2idem}


\begin{dfn}
  Let $w$ be an $\epsilon$-web and $G$ a paired red graph for $w$. We define the $G$-reduction of $w$ to be the $\epsilon$-web denoted by $w_G$ and constructed as follows (see figure~\ref{fig:Greduction} for an example):
  \begin{enumerate}
  \item for every face of $w$ which belongs (as a vertex) to $G$, remove all edges adjacent to this face.
  \item or every face of $w$ connect the grey half-edges of $G$ according to the pairing. 
  \end{enumerate}
\end{dfn}
Note that if $w$ is non-elliptic, $w_G$ needs not to be non-elliptic.
\begin{figure}[ht]
  \centering
  \begin{tikzpicture}[scale =0.72]
    \begin{scope}[yscale = {1}, xscale={1},decoration={markings, mark=at
     position 0.5 with {\arrow{>}}},postaction={decorate}]
\draw (-3.5,-1) .. controls +(0,2.5) and +(0,1) .. (-0.5, 0.5) -- (0.5, 0.5) .. controls +(0,1) and +(0,2.5).. (3.5, -1);
\draw (-3, -1) .. controls +(0,1) and +(-1, 0) .. (-2, 0.5) -- (2, 0.5).. controls +(+1,0) and +(0,1) .. (3, -1);
\draw (-2.5, -1) .. controls +(0, 0.5) and +(-0.5, 0) ..  (-2, -0.5) -- (2, -0.5) ..controls +(+0.5,0) and +(0,0.5) .. (2.5, -1);
\draw (-2, 0.5) -- (-2, -0.5);
\draw (-1, 0.5) -- (-1, -0.5);
\draw (1, 0.5) -- (1, -0.5);
\draw (2, 0.5) -- (2, -0.5);
\draw (-1, 0.5) -- (-1, -0.5);
\draw (-0.5,-1) -- +(0,0.5);
\draw (0.5,-1) -- +(0,0.5);
\draw[dashed] (-0.5, -0.5) -- +(0, 1);
\draw[dashed] (+0.5, -0.5) -- +(0, 1);
\draw[dashed] (-2, -0.5) .. controls +(0.3, 0) and +(0.3,0) .. +(0,1);
\draw[dashed] (2, -0.5) .. controls +(-0.3, 0) and +(-0.3,0) .. +(0,1);
\draw[fill, color =red] (0,0) circle (3pt);
\draw[fill, color =red] (-1.5,0) circle (3pt);
\draw[fill, color =red] (1.5,0) circle (3pt);
\draw[very thick, color =red] (-1.5,0) -- (1.5,0);
\end{scope}

\begin{scope}[yscale = {1}, xscale={1},decoration={markings, mark=at
     position 0.5 with {\arrow{>}}},postaction={decorate}, xshift = 10cm]
\draw (-3.5,-1) .. controls +(0,2.5) and +(0,1) .. (-0.5, 0.5) -- +(0,-1.5);
\draw (0.5, -1) -- (0.5, 0.5) .. controls +(0,1) and +(0,2.5).. (3.5, -1);
\draw (-3, -1) .. controls +(0,1) and +(-1, 0) .. (-2, 0.5) .. controls +(0.5,0) and +(0.5,0).. (-2, -0.5) ..controls +(-0.5,0) and +(0,0.5) .. (-2.5, -1);
\draw (3, -1) .. controls +(0,1) and +(1, 0) .. (2, 0.5) .. controls +(-0.5,0) and +(-0.5,0).. (2, -0.5) ..controls 
+(0.5,0) and +(0,0.5) .. (2.5, -1);
\end{scope}
  \end{tikzpicture}
  \caption{Example of a $G$-reduction of an $\epsilon$-web $w$. The dotted lines represent the pairing.}
  \label{fig:Greduction}
\end{figure}
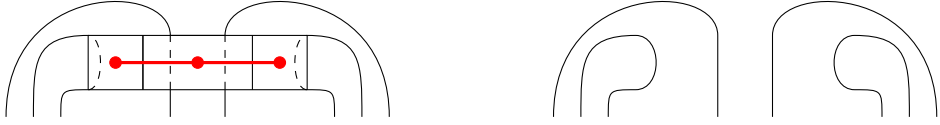

\begin{dfn}\label{dfn:wGfoams}
  Let $w$ be an $\epsilon$-web, and $G$ a fair paired red graph for $w$. We define \emph{the projection associated} with $G$ to be the $(w,w_G)$-foam denoted by $p_G$ and constructed as follows (from bottom to top):
  \begin{enumerate}
  \item For every edge $e'$ of $G$, perform an unzip (see figure~\ref{fig:unzip}) on the edge $e$ corresponding to $e$ in $w$. Note that the condition~(\ref{item:dfnRG2}) in the definition of red graph implies that all these unzip moves are all far from each other, therefore we can perform all the unzips simultaneously. Let us denote by $w'$ the $\epsilon$-web at the top of the foam after this step. Each vertex of $G$ corresponds canonically to some a face of $w'$, this faces are circles, digon or square (with an extra information given by the pairing) because $G$ being fair, every vertex of $G$ have an external degree smaller or equal to 4.
    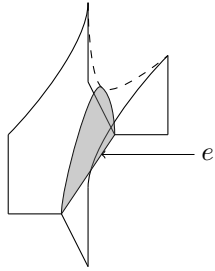
\begin{figure}[ht]
      \centering
      \begin{tikzpicture}[scale=0.35]
        \begin{scope}[xshift=10cm, yshift= 0cm, decoration={markings, mark=at
     position 0.5 with {\arrow{>}}},postaction={decorate}]
\coordinate (A2) at (1,5);
\coordinate (B2) at (4,3);
\coordinate (C2) at (-2,0);
\coordinate (D2) at (1,-2);
\coordinate (E2) at (0,-3);
\coordinate (F2) at (2,0);
\coordinate (A1) at (1,2);
\coordinate (B1) at (4,0);
\coordinate (C1) at (-2,-3);
\coordinate (D1) at (1,-5);
\coordinate (G2) at (1.5,1.8);
\draw (A1) -- (F2) -- (B1);
\draw (C1) -- (E2) -- (D1);
\draw (E2) -- (F2);
\draw[very thin, ->]  (5, -0.75) -- (1.5, -0.75);
\node at (5.5, -0.75) {$e$};
\draw (C2) .. controls +(1,1) and +(0,-1.5) .. (A2);
\draw (D2) .. controls +(0,1.5) and +(-1,-1) .. (B2);
\draw (A1) -- (A2);
\draw (B1) -- (B2);
\draw (C1) -- (C2);
\draw (D1) -- (D2);
\draw (E2) .. controls +(0,1) and  +(-0.4,0.3).. (G2).. controls +(0.4,-0.3) and +(0,0.5) .. (F2);
 \draw[dashed] (A2) .. controls +(0,0) and +( -0.5, +0.5) ..  (G2) .. controls +(0.5,-0.5) and +(0,0) .. (B2);
+(0,0).. (D2);
\fill[opacity=0.4, color= gray] (E2) .. controls +(0,1) and  +(-0.4,0.3).. (G2).. controls +(0.4,-0.3) and +(0,+0.5) .. (F2) --(E2);

\end{scope}
      \end{tikzpicture}
      \caption{Unzip on the edge $e$.}
      \label{fig:unzip}
    \end{figure}
  \item
    \begin{itemize}
    \item For each square of $w'$ which corresponds to a vertex of $G$, perform a square move on it, following the pairing information,  (see figure~\ref{fig:smdmc}).
    \item For each digon of $w'$ which corresponds to a vertex of $G$, perform a digon move  on it (see figure~\ref{fig:smdmc}).
    \item For each circle of $w'$ which corresponds to a vertex of $G$, glue a cap on it (see figure~\ref{fig:smdmc}).
    \end{itemize}
  \end{enumerate}
We define as well $i_G$, \emph{the injection associated with $G$} to be the $(w_G, w)$-foam which the mirror image of $p_G$ with respect to the horizontal plane $\RR^2\times \{\frac12\}$and $\tilde{e}_G$ to be the $(w,w)$-foam equal to $i_G \circ p_G$. An example can be seen figure~\ref{fig:exeG}.
\begin{figure}[ht]
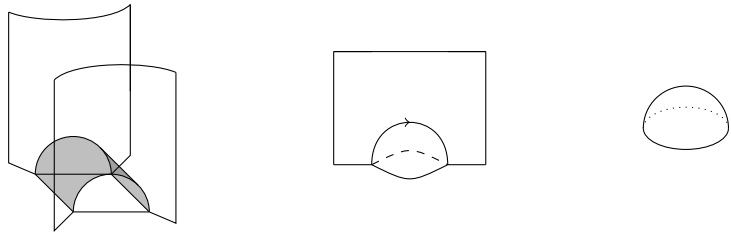

  \centering
  $\foamsquarec[0.5] \hspace{2cm} \foamgid[0.5] \hspace{2cm} \foamcap[0.7]$
  \caption{A square move, a bigon move and a cap.}
  \label{fig:smdmc}
\end{figure} \marginpar{Faire un exemple ?}
\end{dfn}

\begin{figure}[ht]
  \centering
  \begin{tikzpicture}[scale = 0.8]
    \begin{scope}[xshift= 0cm, scale = 2]
\draw (2, 0.2) -- (2,0.7);
\draw (2, 1.3) -- (2,1.7);
\draw (2,2.8) -- (2,2.3);
\draw (0.2, 1.8)-- (0.6, 1.8);
\draw (0.2, 1.2)-- (0.6, 1.2);
\draw[very thin, dashed] (0.6,1.8) .. controls +(0.2,0) and +(0.2,0) .. (0.6,1.2);
\draw[very thin, dashed] (2, 0.7) -- (2, 1.3); 
\draw[very thin, dashed] (2, 1.7) -- (2, 2.3); 
\draw[very thin, dashed] (2.55,0.75) .. controls +(0,0.2) and +(0,0.2) .. (2.85,0.85);
\draw (2.5, 0.2)-- (2.55,0.75);
\draw (3.1, 0.2)-- (2.85,0.85);
\draw (0.6,1.8)--(1, 2.1) -- (1.6, 2.3)-- (2.4, 2.3) --(3,2.1) .. controls (3.4,1.8) and (3.4,1.2) .. (3, 0.9)-- (2.4,0.7)-- (1.6,0.7) --  (1, 0.9)-- (0.6, 1.2)-- cycle;
\draw (1.3,1.7)--(2.7, 1.7) -- (2.7, 1.3)-- (1.3,1.3)--cycle;
\draw (0.6,1.7) -- (0.6,1.3);
\draw (1, 0.9)-- (1.3,1.3);
\draw (1.6,0.7) -- (1.7,1.3);
\draw (2.4,0.7)-- (2.3,1.3);
\draw (3, 0.9)-- (2.7,1.3);
\draw (3,2.1) -- (2.7,1.7);
\draw (2.4, 2.3) --(2.3,1.7);
\draw (1.6, 2.3) --(1.7,1.7);
\draw (1, 2.1) -- (1.3,1.7);
\draw (2.55,0.75)-- (2.85,0.85);
\draw[red, very thick] (2.2, 1) -- (2.6, 1.1) -- (3,1.5) -- (2.6, 1.9) --  (1.8,2) --
(1.4,1.9) -- (1,1.5) -- (1.4, 1.1) -- cycle;
\fill[red] (2.2, 1) circle (1pt);
\fill[red] (2.6, 1.1) circle (1pt);
\fill[red] (3,1.5) circle (1pt);
\fill[red] (2.6, 1.9) circle (1pt);
\fill[red] (1.8,2) circle (1pt);
\fill[red] (1.4,1.9) circle (1pt);
\fill[red] (1,1.5) circle (1pt);
\fill[red] (1.4, 1.1) circle (1pt);
\end{scope}
  \end{tikzpicture}

\vspace{0.5cm}
  \begin{tikzpicture}[scale = 0.6]
\begin{scope}
\draw[draw=  gray, line width = 2mm] (-4,0) --  (16,0); 
\draw [draw = gray, line width = 2mm] (16,3) -- (-4,3); 
\draw[draw=  gray, very thick] (0,0) -- (0,3);
\draw[draw=  gray, very thick] (4,0) -- (4,3);
\draw[draw=  gray, very thick] (8,0) -- (8,3);
\draw[draw=  gray, very thick] (12,0) -- (12,3);
\draw[draw=  gray, very thick] (-4,0) -- (-4,3);
\draw[draw=  gray, very thick] (16,0) -- (16,3);
\draw[draw= white, dotted, line width =1.2mm] (-4,0) -- (16,0);
\draw[draw= white, dotted, line width = 1.2mm] (-4,3) -- (16,3);
\end{scope}

\begin{scope}[xshift= 4cm]
 \draw (2, 0.2) -- (2,2.8);
 \draw (0.2,1.2) arc (-90:90:0.3);
 \draw (2.5,0.2) arc (180:0:0.3);
\end{scope}
\begin{scope}
 \draw (2, 0.2) -- (2,0.7);
\draw (2, 1.3) -- (2,1.7);
\draw (2,2.8) -- (2,2.3);
\draw[green] (2,1) ellipse (0.12 and 0.3);
\draw[green] (2,2) ellipse (0.12 and 0.3);
 \draw (0.2,1.2) -- +(0.5,0.2);
 \draw (0.2,1.8) -- +(0.5,-0.2);
 \draw (2.5, 0.2) -- +(0, 0.7);
 \draw (3.1, 0.2) -- +(-0.2, 0.7);
\filldraw[draw = green, fill= white] (0.9,1.5) circle (0.3);
\draw (3.1,1.5) circle (0.3);
\draw (1.4,1) circle (0.3);
\draw (1.4,2) circle (0.3);
\filldraw[fill=white, draw= green] (2.6,1) circle (0.3);
\draw (2.6,2) circle (0.3);
\end{scope}
\begin{scope}[xshift= -4cm]
\draw (2, 0.2) -- (2,0.7);
\draw (2, 1.3) -- (2,1.7);
\draw (2,2.8) -- (2,2.3);
\draw (0.2, 1.8)-- (0.6, 1.8);
\draw (0.2, 1.2)-- (0.6, 1.2);
\draw (2.5, 0.2)-- (2.55,0.75);
\draw (3.1, 0.2)-- (2.85,0.85);
\draw (0.6,1.8)--(1, 2.1) -- (1.6, 2.3)-- (2.4, 2.3) --(3,2.1) .. controls (3.4,1.8) and (3.4,1.2) .. (3, 0.9)-- (2.4,0.7)-- (1.6,0.7) --  (1, 0.9)-- (0.6, 1.2)-- cycle;
\draw (1.3,1.7)--(2.7, 1.7) -- (2.7, 1.3)-- (1.3,1.3)--cycle;
\draw[blue, thick] (0.6,1.7) -- (0.6,1.3);
\draw[blue, thick] (1, 0.9)-- (1.3,1.3);
\draw[blue, thick] (1.6,0.7) -- (1.7,1.3);
\draw[blue, thick] (2.4,0.7)-- (2.3,1.3);
\draw[blue, thick] (3, 0.9)-- (2.7,1.3);
\draw[blue, thick] (3,2.1) -- (2.7,1.7);
\draw[blue, thick] (2.4, 2.3) --(2.3,1.7);
\draw[blue, thick] (1.6, 2.3) --(1.7,1.7);
\draw[blue, thick] (1, 2.1) -- (1.3,1.7);
\draw[blue, thick] (2.55,0.75)-- (2.85,0.85);
\end{scope}
\begin{scope}[xshift= 12cm]
\draw (2, 0.2) -- (2,0.7);
\draw (2, 1.3) -- (2,1.7);
\draw (2,2.8) -- (2,2.3);
\draw (0.2, 1.8)-- (0.6, 1.8);
\draw (0.2, 1.2)-- (0.6, 1.2);
\draw (2.5, 0.2)-- (2.55,0.75);
\draw (3.1, 0.2)-- (2.85,0.85);
\draw (0.6,1.8)--(1, 2.1) -- (1.6, 2.3)-- (2.4, 2.3) --(3,2.1) .. controls (3.4,1.8) and (3.4,1.2) .. (3, 0.9)-- (2.4,0.7)-- (1.6,0.7) --  (1, 0.9)-- (0.6, 1.2)-- cycle;
\draw (1.3,1.7)--(2.7, 1.7) -- (2.7, 1.3)-- (1.3,1.3)--cycle;
\draw[blue, thick] (0.6,1.7) -- (0.6,1.3);
\draw[blue, thick] (1, 0.9)-- (1.3,1.3);
\draw[blue, thick] (1.6,0.7) -- (1.7,1.3);
\draw[blue, thick] (2.4,0.7)-- (2.3,1.3);
\draw[blue, thick] (3, 0.9)-- (2.7,1.3);
\draw[blue, thick] (3,2.1) -- (2.7,1.7);
\draw[blue, thick] (2.4, 2.3) --(2.3,1.7);
\draw[blue, thick] (1.6, 2.3) --(1.7,1.7);
\draw[blue, thick] (1, 2.1) -- (1.3,1.7);
\draw[blue, thick] (2.55,0.75)-- (2.85,0.85);
\end{scope}
\begin{scope}[xshift= 8cm]
 \draw (2, 0.2) -- (2,0.7);
\draw (2, 1.3) -- (2,1.7);
\draw (2,2.8) -- (2,2.3);
\draw[red] (2,1) ellipse (0.12 and 0.3);
\draw[red] (2,2) ellipse (0.12 and 0.3);
 \draw (0.2,1.2) -- +(0.5,0.2);
 \draw (0.2,1.8) -- +(0.5,-0.2);
 \draw (2.5, 0.2) -- +(0, 0.7);
 \draw (3.1, 0.2) -- +(-0.2, 0.7);
\filldraw[draw = red, fill= white] (0.9,1.5) circle (0.3);
\draw (3.1,1.5) circle (0.3);
\draw (1.4,1) circle (0.3);
\draw (1.4,2) circle (0.3);
\filldraw[fill=white, draw= red] (2.6,1) circle (0.3);
\draw (2.6,2) circle (0.3);
\end{scope}
  \end{tikzpicture}
  \caption{On the top a web together with a fair (actually nice) paired red graph $G$. On the bottom a movie representing $e_G$.}
  \label{fig:exeG}
\end{figure}
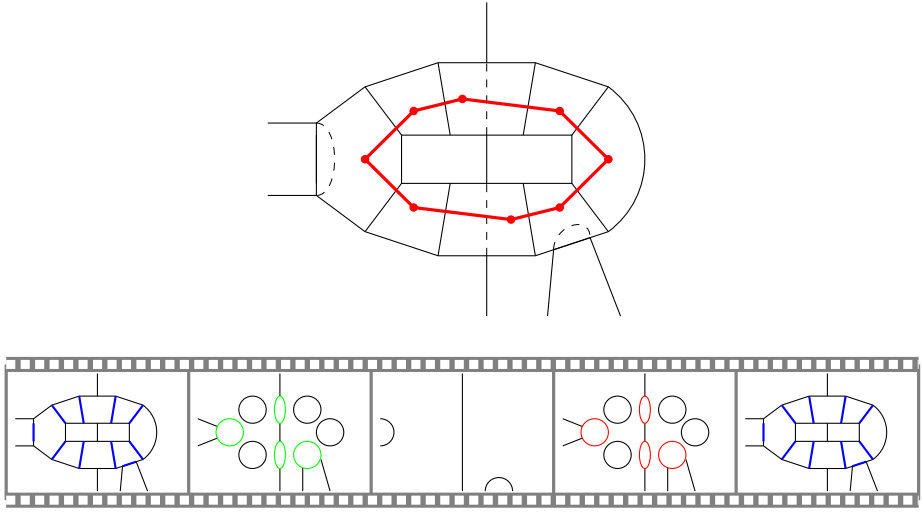

\begin{req}\label{req:ssandcap}
  It's worthwhile to note that a digon move can be seen as a unzip followed by a cap, and that a square move can be seen as two unzips followed by a cap. With this point of view, we see in that in $i_G$ (and in $p_G$), every edge of $G$ and every pair of grey half-edges corresponds a zip (or an unzip) and every vertex of $G$ corresponds a cup (or a cup).
\end{req}
The theorem \ref{thm:RG2idempotent} is an easy consequence of the  following proposition:
\begin{prop}\label{prop:pciisid}
  If $w$ is a non-elliptic web and $G$ is an exact paired red graph for $w$ then the $(w_G,w_G)$-foam $p_G \circ i_G$ is equivalent under the relations \FR{} (see proposition~\ref{prop:relFR}) to a non-zero multiple of the identity $(w_G, w_G)$-foam $w_G\times [0,1]$. 
\end{prop}
To prove this proposition we need to develop a framework to make some calculus with the explicit foams we gave in definition \ref{dfn:wGfoams}.

\subsection{Foam diagrams}
\label{sec:foam-diagrams}
\begin{dfn}
  Let $w$ be an $\epsilon$-web, a \emph{foam diagram} $\kappa$ for $w$ consists of the following data:
  \begin{itemize}
  \item the $\epsilon$-web $w$,
  \item a fair paired red graph $G$,
  \item a function $\delta$ (called a \emph{dot function for $w$}) from $E(w)$ the set of edges of $w$ to $\NN$ the set of non-negative integers. This function will be represented by the appropriate number of dots on each edge of $w$.
  \end{itemize}
With a foam diagram $\kappa$ we associate $f(\kappa)$ the $(w_G,w_G)$-foam given by $p_G\circ s_w(\delta) \circ i_G$, where $s_w(\delta)$ is $\id_w=w\times [0,1]$ the identity foam of $w$ with on every facet $e\times [0,1]$ (with $e\in E(w)$) exactly  $\delta(e)$ dots. The $(w_G,w_G)$-foam $f(\kappa)$ is equal to $p_G\circ i_G$, with dots encoded by $\delta$. A foam diagram will be represented by the $\epsilon$-web drawn together with the red graph, and with some dots added on the edges of the $\epsilon$-web in order to encode $\delta$.
\end{dfn}
We will often assimilate $\kappa=(w,G,\delta)$ with $f(\kappa)$ and it will be seen as an element of $\hom_{K^\epsilon}(P_{w_G},P_{w_G})$. We can rewrite
some of the relations depicted on figure~\ref{fig:localrel} in terms of foam diagrams:
\begin{prop}\label{prop:relfoamdiag} The following relations on foams associated with foam diagrams hold:
  \begin{itemize}
  \item The 3-dots relation: 
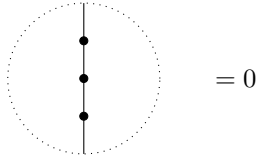
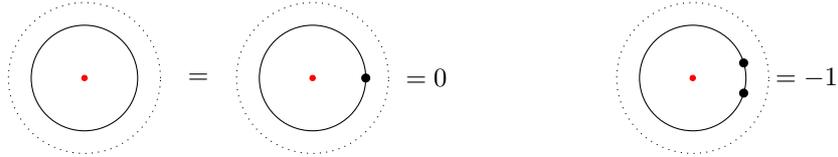
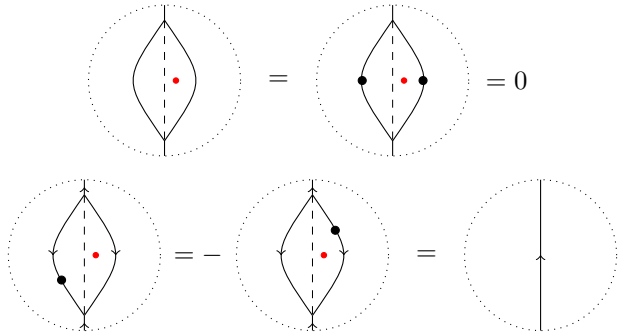
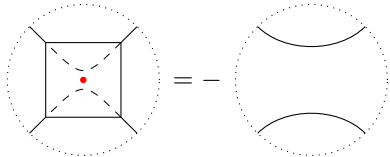
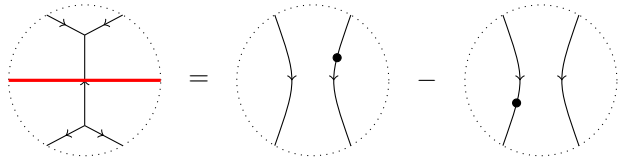
\[
\begin{tikzpicture}
\begin{scope}[yscale = {1}, xscale={1},decoration={markings, mark=at
     position 0.5 with {\arrow{>}}},postaction={decorate}]
\draw[dotted] (0,0) circle (1cm);
\draw (0,-1) -- +(0,2);
\draw[fill] (0, -0.5) circle (1.5pt);
\draw[fill] (0,0) circle (1.5pt);
\draw[fill] (0,0.5) circle (1.5pt);
\node at (2,0) {$=0$};
\end{scope}  
\end{tikzpicture}
\]
\item The sphere relations:
\[
\begin{tikzpicture}
\begin{scope}[yscale = {1}, xscale={1},decoration={markings, mark=at
     position 0.5 with {\arrow{>}}},postaction={decorate}]
\draw[dotted] (0,0) circle (1cm);
\draw (0,0) circle (0.7cm);
\draw[fill, color = red] (0,0) circle (1pt);
\node at (1.5,0) {$=$};
\end{scope}
\begin{scope}[yscale = {1}, xscale={1},decoration={markings, mark=at
     position 0.5 with {\arrow{>}}},postaction={decorate}, xshift=3cm]
\draw[dotted] (0,0) circle (1cm);
\draw (0,0) circle (0.7cm);
\draw[fill] (0.7,0) circle (1.5pt);
\draw[fill, color =red] (0,0) circle (1pt);
\node at (1.5,0) {$=0$};
\end{scope}
\begin{scope}[yscale = {1}, xscale={1},decoration={markings, mark=at
     position 0.5 with {\arrow{>}}},postaction={decorate}, xshift=8cm]
\draw[dotted] (0,0) circle (1cm);
\draw (0,0) circle (0.7cm);
\draw[fill] (0.67,0.2) circle (1.5pt);
\draw[fill] (0.67,-0.2) circle (1.5pt);
\draw[fill, color =red] (0,0) circle (1pt);
\node at (1.5,0) {$=-1$};
\end{scope}  
\end{tikzpicture}
\]\marginpar{dashed instead of dotted}
\item The digon relations:
\[
\begin{tikzpicture}
\begin{scope}[yscale = {1}, xscale={1},decoration={markings, mark=at
     position 0.5 with {\arrow{>}}},postaction={decorate}]
\draw[dotted] (0,0) circle (1cm);
\draw (0,1) -- (0, 0.8) .. controls (-0.55,0) and (-0.55,0) .. (0,-0.8) -- (0,-1);
\draw  (0, 0.8) .. controls (0.55,0) and (0.55,0) .. (0,-0.8);
\draw[fill, color = red] (0.15,0) circle (1pt);
\draw[very thin, dashed] (0, 0.8) -- (0, -0.8);
\node at (1.5,0) {$=$};
\end{scope}
\begin{scope}[yscale = {1}, xscale={1},decoration={markings, mark=at
     position 0.5 with {\arrow{>}}},postaction={decorate}, xshift=3cm]
\draw[dotted] (0,0) circle (1cm);
\draw (0,1) -- (0, 0.8) .. controls (-0.55,0) and (-0.55,0) .. (0,-0.8) -- (0,-1);
\draw  (0, 0.8) .. controls (0.55,0) and (0.55,0) .. (0,-0.8);
\draw[fill] (0.4,0) circle (1.5pt);
\draw[fill] (-0.4,0) circle (1.5pt);
\draw[fill, color = red] (0.15,0) circle (1pt);
\draw[very thin, dashed] (0, 0.8) -- (0, -0.8);
\node at (1.5,0) {$=0$};
\end{scope}  
\end{tikzpicture}
\]
\[
\begin{tikzpicture}
\begin{scope}[yscale = {1}, xscale={1},decoration={markings, mark=at
     position 0.5 with {\arrow{>}}},postaction={decorate}]
\draw[dotted] (0,0) circle (1cm);
\draw[postaction ={decorate}] (0,0.8) -- (0, 1);
\draw[postaction ={decorate}] (0,-1) -- (0, -0.8);
\draw[postaction ={decorate}] (0,0.8).. controls (-0.55,0) and (-0.55,0) .. (0,-0.8);
\draw[postaction ={decorate}] (0,0.8) .. controls (0.55,0) and (0.55,0) .. (0,-0.8);
\draw[fill] (-0.3,-0.33) circle (1.5pt);
\draw[fill, color = red] (0.15,0) circle (1pt);
\draw[very thin, dashed] (0, 0.8) -- (0, -0.8);
\node at (1.5,0) {$=-$};
\end{scope}
\begin{scope}[yscale = {1}, xscale={1},decoration={markings, mark=at
     position 0.5 with {\arrow{>}}},postaction={decorate}, xshift = 3cm]
\draw[dotted] (0,0) circle (1cm);
\draw[postaction ={decorate}] (0,0.8) -- (0, 1);
\draw[postaction ={decorate}] (0,-1) -- (0, -0.8);
\draw[postaction ={decorate}] (0,0.8).. controls (-0.55,0) and (-0.55,0) .. (0,-0.8);
\draw[postaction ={decorate}] (0,0.8) .. controls (0.55,0) and (0.55,0) .. (0,-0.8);
\draw[fill] (0.3,0.33) circle (1.5pt);
\draw[fill, color = red] (0.15,0) circle (1pt);
\draw[very thin, dashed] (0, 0.8) -- (0, -0.8);
\node at (1.5,0) {$=$};
\end{scope}
\begin{scope}[yscale = {1}, xscale={1},decoration={markings, mark=at
     position 0.5 with {\arrow{>}}},postaction={decorate}, xshift = 6cm]
\draw[dotted] (0,0) circle (1cm);
\draw[postaction ={decorate}] (0,-1) -- (0, 1);
\end{scope}  
\end{tikzpicture}
\]
\item The square relations:
\[
\begin{tikzpicture}
\begin{scope}[yscale = {1}, xscale={1},decoration={markings, mark=at
     position 0.5 with {\arrow{>}}},postaction={decorate}]
\draw[dotted] (0,0) circle (1cm);
\draw (45:1) -- (45:0.7) -- (-45:0.7) -- (-45:1);
\draw (135:1) -- (135:0.7) -- (-135:0.7) -- (-135:1);
\draw (45:0.7) -- (135:0.7);
\draw (-45:0.7) -- ( -135:0.7);
\draw[very thin, dashed] (45:0.7) .. controls (0,0) and (0,0) .. (135:0.7);
\draw[very thin, dashed] (-45:0.7) .. controls (0,0) and (0,0) .. (-135:0.7);
\draw[fill, color = red] (0,0) circle (1pt);
\node at (1.5,0) {$=-$};
\end{scope}
\begin{scope}[yscale = {1}, xscale={1},decoration={markings, mark=at
     position 0.5 with {\arrow{>}}},postaction={decorate}, xshift=3cm]
\draw[dotted] (0,0) circle (1cm);
\draw (45:1) .. controls (45:0.5) and (135:0.5) .. (135:1);
\draw(-45:1) .. controls (-45:0.5) and (-135:0.5) .. (-135:1);
\end{scope}  
\end{tikzpicture}
\]
\item The E-relation:
\[
\begin{tikzpicture}
\begin{scope}[yscale = {1}, xscale={1},decoration={markings, mark=at
     position 0.5 with {\arrow{>}}},postaction={decorate}]
\draw[dotted] (0,0) circle (1cm);
\draw[postaction ={decorate}] (0,-0.6) -- (0, 0.6);
\draw[postaction ={decorate}] (0,-0.6) -- (-120:1);
\draw[postaction ={decorate}] (0,-0.6) -- (-60:1);
\draw[postaction ={decorate}] (120:1) -- (0,0.6);
\draw[postaction ={decorate}] (60:1) -- (0,0.6);
\draw[color = red, very thick] (-1,0) -- (1,0);
\node at (1.5,0) {$=$};
\end{scope}
\begin{scope}[yscale = {1}, xscale={1},decoration={markings, mark=at
     position 0.5 with {\arrow{>}}},postaction={decorate}, xshift = 3cm]
\draw[dotted] (0,0) circle (1cm);
\draw[postaction={decorate}] (60:1) .. controls (0.2,0) and (0.2,0) .. (-60:1);
\draw[postaction={decorate}] (120:1) .. controls (-0.20,0) and (-0.2,0) .. (-120:1);
\draw[fill] (0.32,0.3) circle (1.5pt);
\node at (1.5,0) {$-$};
\end{scope}
\begin{scope}[yscale = {1}, xscale={1},decoration={markings, mark=at
     position 0.5 with {\arrow{>}}},postaction={decorate}, xshift = 6cm]
\draw[dotted] (0,0) circle (1cm);
\draw[postaction={decorate}] (60:1) .. controls (0.2,0) and (0.2,0) .. (-60:1);
\draw[postaction={decorate}] (120:1) .. controls (-0.20,0) and (-0.2,0) .. (-120:1);
\draw[fill] (-0.32,-0.3) circle (1.5pt);
\end{scope}  
\end{tikzpicture}
\]
  \end{itemize}
The dashed lines indicate the pairing, and when the orientation of the $\epsilon$-web is not depicted the relation holds for any orientation.
\end{prop}
\begin{proof}
  This is equivalent to some of the relations depicted on figure~\ref{fig:localrel}.
\end{proof}
\begin{lem}\label{lem:fd2idwdots}
  Let $w$ be an $\epsilon$-web and $\kappa = (w,G,\delta)$ a foam diagram, with $G$ a fair paired red graph. Then $f(\kappa)$ is equivalent to a $\ZZ$-linear combination of $s_{w_G}(\delta_i)= f((w_G,\emptyset, \delta_i))$ for $\delta_i$ some dots functions for $w_G$. 
\end{lem}\marginpar{Signaler qu'on travail avec la TQFT pour $R= \QQ$}
\begin{proof}
Thanks to the E-relation of proposition~\ref{prop:relfoamdiag}, one can express $f(\kappa)$ as a $\ZZ$-linear combination of $f((w_j,G_j,\delta_j))$ where the $G_j$'s are red graphs without any edge. Tanks to the sphere, the digon and square relations of  proposition~\ref{prop:relfoamdiag}, each $f((w_j,G_j,\delta_j))$ is equivalent either to 0 or to $\pm f(w_G,\emptyset, \delta'_j)$. This proves the lemma.
\end{proof}
\begin{lem}\label{lem:fde2id}
Let $w$ be an $\epsilon$-web and $\kappa = (w,G,\delta)$ a foam diagram, with $G$ exact, then $f(\kappa)$ is equivalent to a multiple of $w_G\times [0,1]$.
\end{lem}
\begin{proof}
  From the previous lemma we know that $f(\kappa)$ is equivalent to a $\ZZ$-linear combination of $w_G\times [0,1]$ with some dots on it. We will see that the foam $f(\kappa)$ has the same degree as the foam $w_G\times [0,1]$. This will prove the lemma because adding a dot on a foam increases its degree by 2. 

To compute the degree of $f(\kappa)$ we see it as a composition of elementary foams thanks to its definition:
\begin{align*}
\deg f(\kappa) =& \deg(w\times [0,1]) + 2\left( 2\#V(G) - \left(\#E(G)+ \frac{\#\{\textrm{grey half-edges of $G$}\}}{2}  \right) \right) \\
 =& |\partial w| + 2\cdot 0 \\
=& \deg w_G \times [0,1].
\end{align*}
The first equality is due to the decomposition pointed out in remark ~\ref{req:ssandcap} and  because an unzip (or a zip) has degree -1 and a cap (or a cup) has degree +2. The factor 2 is due to the fact $f(\kappa)$ is the composition of $i_G$ and $p_G$. The second one follows from the exactness of $G$. 
\end{proof}
To prove the proposition~\ref{prop:pciisid}, we just need to show that in the situation of the last lemma, the multiple is not equal to zero. In order to evaluate this multiple, we extend foam diagrams to (partially) oriented paired red graphs by the local relation indicated on figure~\ref{fig:fd2oRG}.

\begin{figure}[ht]
  \centering
  \begin{tikzpicture}
    \begin{scope}[yscale = {1}, xscale={1},decoration={markings, mark=at
     position 0.5 with {\arrow{>}}},postaction={decorate}]
\draw[dotted] (0,0) circle (1cm);
\draw (0,-0.6) -- (0, 0.6);
\draw (0,-0.6) -- (-120:1);
\draw (0,-0.6) -- (-60:1);
\draw (120:1) -- (0,0.6);
\draw (60:1) -- (0,0.6);
\draw[color = red, very thick, postaction={decorate}] (-1,0) -- (1,0);
\node at (1.5,0) {$\eqdef$};
\end{scope}
\begin{scope}[yscale = {1}, xscale={1},decoration={markings, mark=at
     position 0.5 with {\arrow{>}}},postaction={decorate}, xshift = 3cm]
\draw[dotted] (0,0) circle (1cm);
\draw (60:1) .. controls (0.2,0) and (0.2,0) .. (-60:1);
\draw (120:1) .. controls (-0.20,0) and (-0.2,0) .. (-120:1);
\draw[fill] (0.32,0.3) circle (1.5pt);
\end{scope}
  \end{tikzpicture}
  \caption{Extension of foam diagrams to oriented red graphs. }
  \label{fig:fd2oRG}
\end{figure}
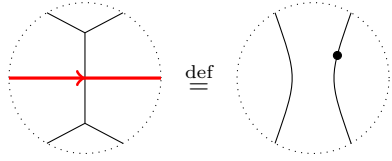

By ``partially oriented'' we mean that some edges may be oriented some may not. If $G$ is partially oriented, and $\kappa$ is a foam diagram with red graph $G$, we say that $\kappa'$ is the \emph{classical foam diagram} associated with $\kappa$ if it obtained from $\kappa$ by applying the relation of figure~\ref{fig:fd2oRG} on every oriented edges. Note that $\kappa$ and $\kappa'$ represent the same foam.
\begin{dfn}
  If $w$ is an $\epsilon$-web, $G$ a red graph for $w$ and $o$ a partial orientation of $G$ we define $\gamma(o)$ to be equal to ${\#\{\textrm{negative edges of $G$}\}}$. A \emph{negative} (or \emph{positive}) edge is an oriented edge of the red graph, and it's negativity (or positivity) is given by figure~\ref{fig:nepe}.
  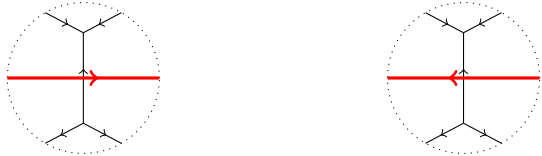
\begin{figure}[ht]
    \centering
    \begin{tikzpicture}
      \begin{scope}[yscale = {1}, xscale={1},decoration={markings, mark=at
     position 0.6 with {\arrow{>}}},postaction={decorate}]
\draw[dotted] (0,0) circle (1cm);
\draw[postaction ={decorate}] (0,-0.6) -- (0, 0.6);
\draw[postaction ={decorate}] (0,-0.6) -- (-120:1);
\draw[postaction ={decorate}] (0,-0.6) -- (-60:1);
\draw[postaction ={decorate}] (120:1) -- (0,0.6);
\draw[postaction ={decorate}] (60:1) -- (0,0.6);
\draw[color = red, very thick, postaction={decorate}] (-1,0) -- (1,0);
\end{scope}
\begin{scope}[yscale = {1}, xscale={1},decoration={markings, mark=at
     position 0.6 with {\arrow{>}}},postaction={decorate}, xshift = 5cm]
\draw[dotted] (0,0) circle (1cm);
\draw[postaction ={decorate}] (0,-0.6) -- (0, 0.6);
\draw[postaction ={decorate}] (0,-0.6) -- (-120:1);
\draw[postaction ={decorate}] (0,-0.6) -- (-60:1);
\draw[postaction ={decorate}] (120:1) -- (0,0.6);
\draw[postaction ={decorate}] (60:1) -- (0,0.6);
\draw[color = red, very thick, postaction={decorate}] (1,0) -- (-1,0);
\end{scope}
    \end{tikzpicture}
    \caption{On the left, a positive edge. On the right, a negative edge.}
    \label{fig:nepe}
  \end{figure}
\end{dfn}
\begin{lem}\label{lem:sumoforientation}
  Let $w$ be an $\epsilon$-web and $G$ a partially oriented red graph with $e$ a non-oriented edge of $G$, then we have the following equality of foams:
 
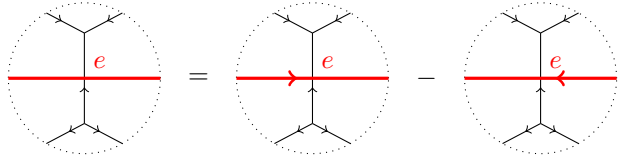
\[   \begin{tikzpicture}
      \begin{scope}[yscale = {1}, xscale={1},decoration={markings, mark=at
     position 0.4 with {\arrow{>}}},postaction={decorate}]
\draw[dotted] (0,0) circle (1cm);
\draw[postaction ={decorate}] (0,-0.6) -- (0, 0.6);
\draw[postaction ={decorate}] (0,-0.6) -- (-120:1);
\draw[postaction ={decorate}] (0,-0.6) -- (-60:1);
\draw[postaction ={decorate}] (120:1) -- (0,0.6);
\draw[postaction ={decorate}] (60:1) -- (0,0.6);
\draw[color = red, very thick] (-1,0)  -- (1,0) node[above, pos=0.6] {$e$};
\node at (1.5,0) {$=$};
\end{scope}
\begin{scope}[yscale = {1}, xscale={1},decoration={markings, mark=at
     position 0.4 with {\arrow{>}}},postaction={decorate}, xshift = 3cm]
\draw[dotted] (0,0) circle (1cm);
\draw[postaction ={decorate}] (0,-0.6) -- (0, 0.6);
\draw[postaction ={decorate}] (0,-0.6) -- (-120:1);
\draw[postaction ={decorate}] (0,-0.6) -- (-60:1);
\draw[postaction ={decorate}] (120:1) -- (0,0.6);
\draw[postaction ={decorate}] (60:1) -- (0,0.6);
\draw[color = red, very thick, postaction={decorate}] (-1,0) -- (1,0) node[pos= 0.6, above] {$e$};
\node at (1.5,0) {$-$};
\end{scope}
\begin{scope}[yscale = {1}, xscale={1},decoration={markings, mark=at
     position 0.4 with {\arrow{>}}},postaction={decorate}, xshift = 6cm]
\draw[dotted] (0,0) circle (1cm);
\draw[postaction ={decorate}] (0,-0.6) -- (0, 0.6);
\draw[postaction ={decorate}] (0,-0.6) -- (-120:1);
\draw[postaction ={decorate}] (0,-0.6) -- (-60:1);
\draw[postaction ={decorate}] (120:1) -- (0,0.6);
\draw[postaction ={decorate}] (60:1) -- (0,0.6);
\draw[color = red, very thick, postaction={decorate}] (1,0) -- (-1,0) node [pos=0.4, above] {$e$};
\end{scope}
    \end{tikzpicture} \]
If $G$ is an un-oriented red graph for $w$ and $\delta$ a dots function for $w$, then:
\[f(w,G,\delta) =\sum_{o} (-1)^{\gamma(o)}f(w, G_o, \delta),\]
where $G_o$ stands for $G$ endowed with the orientation $o$, and $o$ runs through all the $2^{\#E(G)}$ complete orientations of $G$.
\end{lem}
\begin{proof}
  The first equality is the translation of the E-relation (see proposition~\ref{prop:relfoamdiag}) in terms of foam diagrams of partially oriented red graphs. The second formula is the expansion of the first one to all edges of $G$.
\end{proof}
\begin{lem}\label{lem:nonfitting20}
  If $w$ is an $\epsilon$-web, $G$ an exact paired red graph for $w$, $o$ a non-fitting orientation for $G$ and $\delta$ the null dot function on $w$, then  the $(w_g,w_G)$-foam $f(w, G_o, \delta)$ is equivalent to 0.
\end{lem}
\begin{proof}
  The orientation $o$ is a non-fitting orientation. Hence, there is at least one vertex $v$ of $G$ so that $i_o(v)>0$. There are two different situations, either $i_o(v)=1$ or $i_o(v)=2$. Using the definition of a foam diagrams for oriented red graphs (figure~\ref{fig:fd2oRG}), we deduce that $\kappa'$ the classical foam diagram associated with $f(w, G_o, \delta)$ looks around $v$ like one of the three following situations:
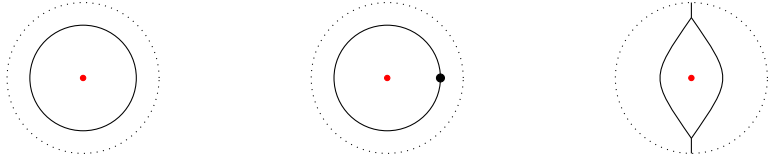
\[
\begin{tikzpicture}
  \begin{scope}[yscale = {1}, xscale={1},decoration={markings, mark=at
     position 0.5 with {\arrow{>}}},postaction={decorate}]
\draw[dotted] (0,0) circle (1cm);
\draw (0,0) circle (0.7cm);
\draw[fill, color = red] (0,0) circle (1pt);
\end{scope}
\begin{scope}[yscale = {1}, xscale={1},decoration={markings, mark=at
     position 0.5 with {\arrow{>}}},postaction={decorate}, xshift=4cm]
\draw[dotted] (0,0) circle (1cm);
\draw (0,0) circle (0.7cm);
\draw[fill] (0.7,0) circle (1.5pt);
\draw[fill, color =red] (0,0) circle (1pt);
\end{scope}

\begin{scope}[yscale = {1}, xscale={1},decoration={markings, mark=at
     position 0.5 with {\arrow{>}}},postaction={decorate}, xshift = 8cm]
\draw[dotted] (0,0) circle (1cm);
\draw (0,1) -- (0, 0.8) .. controls (-0.55,0) and (-0.55,0) .. (0,-0.8) -- (0,-1);
\draw  (0, 0.8) .. controls (0.55,0) and (0.55,0) .. (0,-0.8);
\draw[fill, color = red] (0,0) circle (1pt);
\end{scope}

\end{tikzpicture}
\]
The sphere relations and the digon relations provided by proposition~\ref{prop:relfoamdiag} we see that the foam $f(w, G_o, \delta)$ is equivalent 0.
\end{proof}
\begin{lem}\label{lem:fitting2id}
  If $w$ is an $\epsilon$-web, $G$ an exact paired red graph for $w$, $o$ a fitting orientation for $G$ and $\delta$ the null dots function on $w$, then the $(w_G,w_G)$-foam $f(w, G_o, \delta)$ is equivalent to $(-1)^{\mu(o)} w_G\times I$, where $\mu(o)=\#V(G) + \#\{\textrm{positive digons of $G_o$}\}$ (see definition figure~\ref{fig:5localsituation}).
\end{lem}
\begin{proof}
Let $\kappa'=(w',G',\delta')$ be the classical foam diagram associated with $(w, G_o,\delta)$.
The red graph $G'$ has no edge. Locally, the foam diagram $\kappa'$ corresponds to one of the 5 situation depicted on figure~\ref{fig:5localsituation}.

  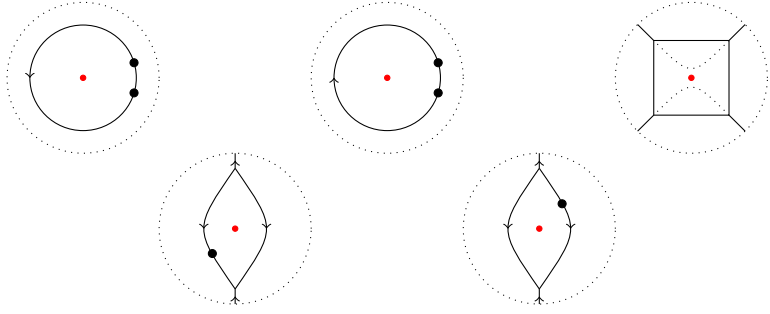
\begin{figure}[ht]
    \centering
    \begin{tikzpicture}
      \begin{scope}[yscale = {1}, xscale={1},decoration={markings, mark=at
     position 0.5 with {\arrow{>}}},postaction={decorate}, xshift=0cm]
\draw[dotted] (0,0) circle (1cm);
\draw[postaction= {decorate}] (0,0) circle (0.7cm);
\draw[fill] (0.67,0.2) circle (1.5pt);
\draw[fill] (0.67,-0.2) circle (1.5pt);
\draw[fill, color =red] (0,0) circle (1pt);
\end{scope}
\begin{scope}[yscale = {-1}, xscale={1},decoration={markings, mark=at
     position 0.5 with {\arrow{>}}},postaction={decorate}, xshift=4cm]
\draw[dotted] (0,0) circle (1cm);
\draw[postaction= {decorate}] (0,0) circle (0.7cm);
\draw[fill] (0.67,0.2) circle (1.5pt);
\draw[fill] (0.67,-0.2) circle (1.5pt);
\draw[fill, color =red] (0,0) circle (1pt);
\end{scope}
\begin{scope}[yscale = {1}, xscale={1},decoration={markings, mark=at
     position 0.5 with {\arrow{>}}},postaction={decorate}, xshift = 8cm]
\draw[dotted] (0,0) circle (1cm);
\draw (45:1) -- (45:0.7) -- (-45:0.7) -- (-45:1);
\draw (135:1) -- (135:0.7) -- (-135:0.7) -- (-135:1);
\draw (45:0.7) -- (135:0.7);
\draw (-45:0.7) -- ( -135:0.7);
\draw[dotted] (45:0.7) .. controls (0,0) and (0,0) .. (135:0.7);
\draw[dotted] (-45:0.7) .. controls (0,0) and (0,0) .. (-135:0.7);
\draw[fill, color = red] (0,0) circle (1pt);
\end{scope}
\begin{scope}[yscale = {1}, xscale={1},decoration={markings, mark=at
     position 0.5 with {\arrow{>}}},postaction={decorate},yshift = -2cm, xshift =2cm]
\draw[dotted] (0,0) circle (1cm);
\draw[postaction ={decorate}] (0,0.8) -- (0, 1);
\draw[postaction ={decorate}] (0,-1) -- (0, -0.8);
\draw[postaction ={decorate}] (0,0.8).. controls (-0.55,0) and (-0.55,0) .. (0,-0.8);
\draw[postaction ={decorate}] (0,0.8) .. controls (0.55,0) and (0.55,0) .. (0,-0.8);
\draw[fill] (-0.3,-0.33) circle (1.5pt);
\draw[fill, color = red] (0,0) circle (1pt);
\end{scope}
\begin{scope}[yscale = {1}, xscale={1},decoration={markings, mark=at
     position 0.5 with {\arrow{>}}},postaction={decorate}, xshift = 6cm, yshift = -2cm]
\draw[dotted] (0,0) circle (1cm);
\draw[postaction ={decorate}] (0,0.8) -- (0, 1);
\draw[postaction ={decorate}] (0,-1) -- (0, -0.8);
\draw[postaction ={decorate}] (0,0.8).. controls (-0.55,0) and (-0.55,0) .. (0,-0.8);
\draw[postaction ={decorate}] (0,0.8) .. controls (0.55,0) and (0.55,0) .. (0,-0.8);
\draw[fill] (0.3,0.33) circle (1.5pt);
\draw[fill, color = red] (0,0) circle (1pt);
\end{scope}
    \end{tikzpicture}
    \caption{The 5 different local situations of a foam diagram $\kappa'$ next to a vertex of $G'$. On the second line, the digon on the left is \emph{positive} and the digon on the right is \emph{negative}.}
    \label{fig:5localsituation}
  \end{figure}

But now using some relations of proposition~\ref{prop:relfoamdiag} we can remove all the vertices of $G'$, we see that $f(w,G_o, \delta)$ is equivalent to $(-1)^{\#V(G')-\#\{\textrm{positive digons}\}}$ because the positive digon is the only one with no minus sign in the relations of prop~\ref{prop:relfoamdiag}. This proves the result because $V(G)=V(G')$.
\end{proof}
\begin{lem}
  If $w$ is an $\epsilon$-web, $G$ an exact paired red graph for $w$ and $o_1$ and $o_2$ two fitting orientations for $G$, then $\mu(o_1) + \gamma(o_1) = \mu(o_2) + \gamma(o_2) $. \marginpar{ toujours nul ?}
\end{lem}
\begin{proof}
  We consider $\kappa'_1=(w',G',\delta_1)$ and $\kappa'_2(w',G',\delta_2)$ the two classical foam diagrams corresponding to $(w,G_{o_1},\delta)$ and $(w,G_{o_2},\delta)$, with $\delta$ the null dots function for $w$. 

The red graph $G'$ has no edge, and the local situation are depicted on figure~\ref{fig:5localsituation}. Consider $v$ a vertex of $G'$, then a side of the face of $w$ corresponding to $v$ is either clockwise or counterclockwise oriented (with respect to this face). From the definition of $\gamma$ we obtain that for $i=1,2$, $\gamma(o_i)$ is equal to the number of dots in $\kappa'_i$ on clockwise oriented edges in $w'$. The dots functions $\delta_1$ and $\delta_2$ differs only next to the digons, so that $\gamma(o_1) - \gamma(o_2)$ is equal to the number of negative digons in $\kappa'_1$ minus the number of negative digons in $\kappa'_2$. So that we have:
\begin{align*}
\gamma(o_1) - \gamma(o_2) &= \mu(o_2) -\mu(o_1) \\
\gamma(o_1)+ \mu(o_1) &= \gamma(o_2)+\mu(o_2).
\end{align*}
\end{proof}
\begin{proof}[Proof of proposition~\ref{prop:pciisid}.]
The foam $p_G\circ i_G$ is equal to $f(w,G,\delta)$ with $\delta$ the null dot function on $w$. From the lemmas~\ref{lem:sumoforientation}, \ref{lem:nonfitting20} and \ref{lem:fitting2id} we have that:
\begin{align*} 
f(w,G,\delta)& = \sum_{\textrm{$o$ fitting orientation of $G$}} (-1)^{\gamma(o)} f(w, G_o, \delta) \\
 =& \sum_{\textrm{$o$ fitting orientation of $G$}} (-1)^{\gamma(o) + \mu(o)} w_G\times [0,1]\\
 = & \pm\#\{\textrm{fitting orientations of $G$}\} w_G\times [0,1]. 
\end{align*}
The red graph $G$ is supposed to be exact. This means in particular that the set of fitting orientation is not empty. So that  $p_G\circ i_G$ is a non-trivial multiple of  $\id_{w_G}=w_G\times [0,1]$. 
\end{proof}
\begin{proof}[Proof of theorem \ref{thm:RG2idempotent}]
  From the proposition~\ref{prop:pciisid}, we know that there exists a non zero integer $\lambda_G$ such that $p_G\circ i_G= \lambda_G w_G$. Hence, $\frac{1}{\lambda_G}i_G\circ p_G$ is an idempotent. It's clear that it's non-zero. It is quite intuitive that it is not equivalent to the identity foam, for a proper proof, see proposition~\ref{prop:idfoam}.
\end{proof}


\subsection{On the identity foam}
\label{sec:identity-foam}

\begin{dfn}
  Let $w$ be an $\epsilon$-web, and $f$ a $(w,w)$-foam, we say that $f$ is \emph{reduced} if every facet of $f$ is diffeomorphic to a disk and if $f$ contains no singular circle (\ie only singular arcs). In particular this implies that every facet of $f$ meets $w\times \{0\}$ or $w\times\{1\}$.
\end{dfn}

The aim of this section is to prove the following proposition:
\begin{prop}\label{prop:idfoam}
  Let $w$ be a non-elliptic $\epsilon$-web. If $f$ is a reduced $(w,w)$-foam which is equivalent (under the foam relations \FR{}) to a non-zero multiple of $w\times [0,1]$, then the underlying pre-foam is diffeomorphic to $w\times [0,1]$ and contains no dot.
\end{prop}

For this purpose we begin with a few technical lemmas:

\begin{lem} \label{lem:foamwithdots}
  Let $w$ be a closed web and $e$ an edge of $w$. Then there exists $f$ a $(\emptyset, w)$-foam which is not equivalent to 0 such that the facet of $f$ touching the edge $e$ contains at least one dot.
\end{lem}
\begin{proof}
We prove the lemma by induction on the number of edges of the web $w$.
  It is enough to consider the case $w$ connected because the functor $\F$ is monoidal.  If the web $w$ is a circle this is clear, since a cap with one dot on it is not equivalent to 0. If $w$ is the theta web, then this is clear as well, since the half theta foam with one dot on the facet meeting $e$ is not equivalent to 0. 

Else,  there exists a square or digon in $w$ somewhere far from $e$. Let us denote $w'$ the web similar to $w$ but with the digon replaced by a single strand or the square smoothed in one way or the other. By induction we can find an $(\emptyset, w')$-foam $f'$ non-equivalent to 0 with one dot on the facet touching $e$. 

Next to the strand or the smoothed square, we consider a digon move or a square move (move upside down the pictures of figure~\ref{fig:smdmc}). Seen as a $(w',w)$-foam it induces an injective map. Therefore, the composition of $f'$ with this $(w',w)$-foam is not equivalent to 0 and has one dot on the facet touching $e$.
\end{proof}
\begin{notation}
  Let $w$ be an $\epsilon$-web, and $e$ be an edge of $w$. We denote by $f(w,e)$ the $(\emptyset,\overline{w}w)$-foam which is diffeomorphic to $w\times [0,1]$ with one dot on the facet $e\times [0,1]$. We denote by $f(w,\emptyset)$ the $(\emptyset,\overline{w} w)$-foam which is diffeomorphic to $w\times [0,1]$ with no dot on it.
\end{notation}
\begin{cor}
  Let $w$ be an $\epsilon$-web, and $e$ an edge of $w$, then $f(w,e)$ is non-equivalent to 0. 
\end{cor}
\begin{proof}
  From lemma~\ref{lem:foamwithdots} we know that for any $w$, there exists a $(w,w)$-foam which is non equivalent to 0 and is the product of $f(w,e)$ with another $(w,w)$-foam. This proves that the $(w,w)$-foam $f(w,e)$ is not equivalent to 0.
\end{proof}
\begin{dfn}
  If $w$ is an $\epsilon$-web. We say that it contains a $\lambda$ (resp.~a $\cap$, resp.~a $H$) if next to the border $w$ looks like one of the pictures of figure~\ref{fig:dfnUHY}.
\begin{figure}[ht]
  \centering
  \begin{tikzpicture}
    \input{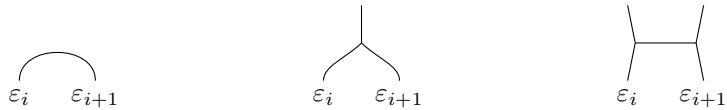}[scale=0.6]
  \end{tikzpicture}
  \caption{From left to right: a $\lambda$, a $\cap$ and a $H$.}
  \label{fig:dfnUHY}
\end{figure}
\end{dfn}
\begin{lem}\label{lem:UHYinNE}
Every non-elliptic $\epsilon$-web contains at least a $\lambda$, a $\cap$ or an $H$.
\end{lem}
\begin{proof}
  The closed web $\bar{w}w$ contains a circle a digon or a square, and this happens only if $w$ contains a $\cap$ a $\lambda$ or a $H$.
\end{proof}
\begin{req}
  In fact, one can ``build'' every non-elliptic web with this three elementary webs. This is done via the ``growth algorithm'' (see~\cite{MR1684195}).
\end{req}
\begin{lem}\label{lem:NEwebwithdots}
  Let $w$ be a non-elliptic $\epsilon$-web. Then the elements of $\left( f(w,e)\right)_{e\in E(w)}$ are pairwise non-equivalent (but they may be linearly dependant).
\end{lem}
\begin{proof}[Sketch of the proof]
  We proceed by induction on the number of edges of $w$. The initiation is straightforward since if $w$ has only one edge there is nothing to prove. We can distinguish several case thanks to lemma~\ref{lem:UHYinNE}:

If $w$ contains a $\cap$, we denote by $e$ the edge of this $\cap$, and by $w'$ the $\epsilon'$-web similar to $w$ but with the cap removed. Suppose that $e_1=e$, then $e_2\neq e$ and, then the $(\emptyset,\overline{w}w)$-foams $f(w,e_1)$ and $f(w,e_2)$ are different because if we cap the cup (we mean $e\times I$) by a cap with one dot on it, on the one hand we obtain a $(\emptyset,\overline{w'}w')$-foam equivalent to $0$ and on the other hand a  $(\emptyset,\overline{w'}w')$-foam equivalent to $f(w',\emptyset)$. Thanks to lemma~\ref{lem:foamwithdots},
we know that this last $(\emptyset,\overline{w'}w')$-foam is not equivalent to 0. If $e_1$ and $e_2$ are different from $e$, it is clear as well, because  $f(w, e_1)$ and $f(w,e_2)$ can be seen as compositions of  $f(w', e_1)$ and $f(w',e_2)$ with a birth (seen as a $(\overline{w'}w', \overline{w}w)$-foam) which is known to correspond to  injective map.

This is the same kind of argument for the two other cases. The digon relations and the square relations instead of the sphere relations.
\end{proof}
\begin{lem}\label{lem:red1touch2wI}
  Let $w$ be an $\epsilon$-web and $f$ a reduced $(w,w)$-foam $f$.  Suppose that every facet touches $w\times\{0\}$ on at most one edge, and touches $w\times\{1\}$ on at most one edge, then it is isotopic to $w\times [0,1]$.
\end{lem}
\begin{proof}
  The proof is inductive on the number of vertices of $w$. If $w$ is a collection of arcs, the foam $f$ has no singular arc. As $f$ is supposed to be reduced, it has no singular circle. Therefore it is a collection of disks which corresponds to the arcs of $w$, and this proves the result in this case.

We suppose now that $w$ has at least one vertex. Let us pick a vertex $v$ which is a neighbour (via an edge that we call $e$) of the boundary $\epsilon$ of $w$. We claim that the singular arc $\alpha$ starting at $v\times\{0\}$ must end at $v\times\{1\}$. 

Indeed, the arc $\alpha$ cannot end on $w\times\{0\}$,  for otherwise, the facet $f$ touching $e$ would touch another edge of $w$. Therefore the arc $\alpha$ ends on $w\times\{1\}$. 
For exactly the same reasons, it has to end on $v\times\{1\}$, so that the facet which touches  $e\times\{0\}$ is isotopic to $e\times I$, now we can remove a neighbourhood of this facet and we are back in the same situation with a $\epsilon'$-web with less vertices, and this concludes.
\end{proof}
\begin{proof}[Proof of proposition~\ref{prop:idfoam}]
  We consider $w$ a non-elliptic $\epsilon$-web. Let $f$ be a reduced $(w,w)$-foam such that $f$ is equivalent to $w\times I$ up to a non-trivial scalar. Because of lemma~\ref{lem:NEwebwithdots}, the foam $f$ satisfies the hypotheses of lemma~\ref{lem:red1touch2wI}, so that $f$ is isotopic to $w\times [0,1]$.
\end{proof}
We conjecture that the proposition~\ref{prop:idfoam} still holds without the non-ellipticity hypothesis. However the proof has to be changed since lemma \ref{lem:NEwebwithdots} cannot be extend to elliptic webs (consider the facets around a digon).
\begin{cor}
  If $w$ is a non-elliptic $\epsilon$-web and $w'$ is an $\epsilon$-web with strictly less vertices than $w$, then if $f$ is a $(w,w')$-foam  and $g$ is a $(w',w)$-foam, then the $(w,w)$-foam $fg$ cannot be equal to a scalar times the identity.
\end{cor}

\begin{landscape}
\begin{figure}
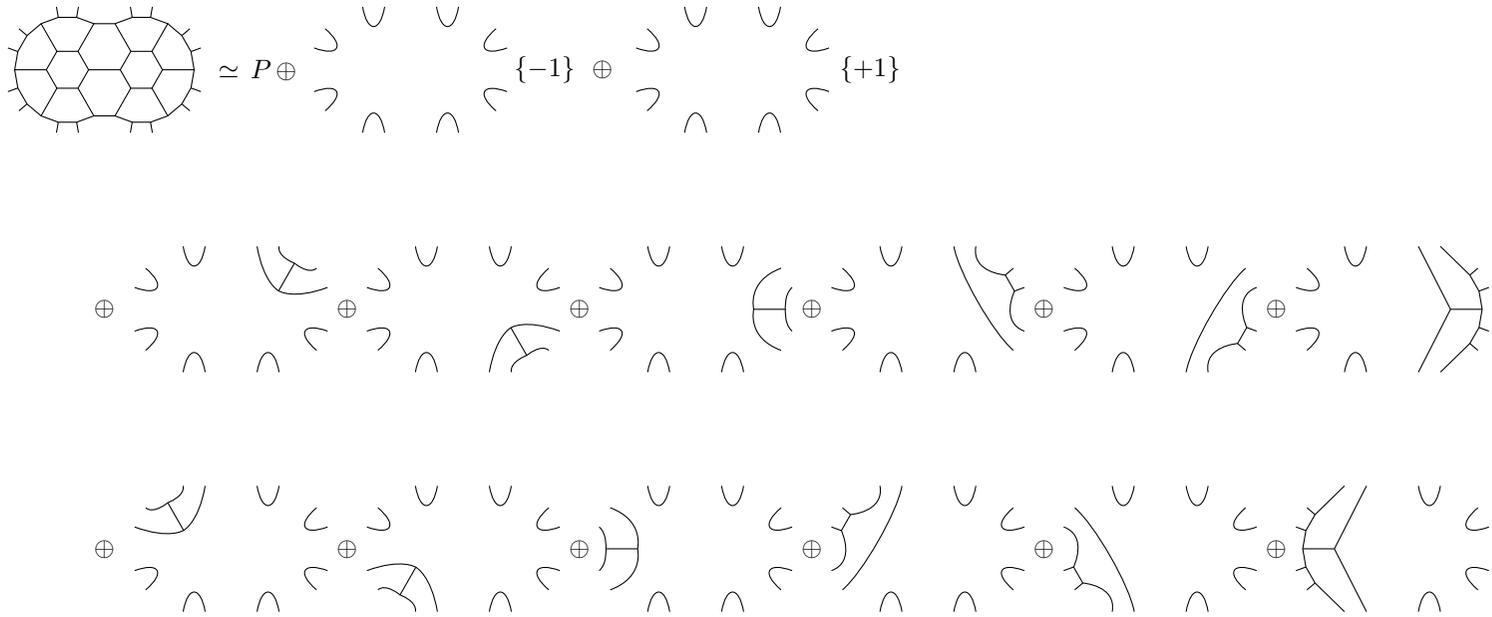

  \begin{tikzpicture} [scale = 0.7]
    \begin{scope}
      \input{\imagesfolder/th_exdecofmodules}
    \end{scope}
    \node at (2.9, 0) {$\simeq \, P \, \oplus$};
      \begin{scope}[xshift = 5.8cm]
        \input{\imagesfolder/th_exdecofmodules1}
      \end{scope}
       \begin{scope}[xshift = 11.9cm]
        \input{\imagesfolder/th_exdecofmodules1}
      \end{scope}
        \node at (8.7, 0) {$\{-1\}\,\,\,\,\oplus$};
        \node at (14.5,0) {$\{+1\}$};
      \begin{scope}[yshift =-4.5cm, xshift = -2cm]
        \node at (2, 0) {$\oplus$};
      \begin{scope}[xshift = 4.4cm]
        \input{\imagesfolder/th_exdecofmodules2}
      \end{scope}
      \node at (6.6, 0) {$\oplus$};
     \begin{scope}[xshift = 8.8cm, yscale= -1]
        \input{\imagesfolder/th_exdecofmodules2}
      \end{scope}
      \node at (11, 0) {$\oplus$};
     \begin{scope}[xshift = 13.2cm, yscale= -1]
        \input{\imagesfolder/th_exdecofmodules3}
      \end{scope}
     \node at (15.4, 0) {$\oplus$};
     \begin{scope}[xshift = 17.6cm, yscale= 1]
        \input{\imagesfolder/th_exdecofmodules4}
      \end{scope} 
     \node at (19.8, 0) {$\oplus$};
     \begin{scope}[xshift = 22.0cm, yscale= -1]
        \input{\imagesfolder/th_exdecofmodules4}
      \end{scope} 
      \node at (24.2, 0) {$\oplus$};
     \begin{scope}[xshift = 26.4cm, yscale= 1]
        \input{\imagesfolder/th_exdecofmodules5}
      \end{scope} 
    \end{scope} 
\begin{scope}[yshift =-9cm, xshift = -2cm]
        \node at (2, 0) {$\oplus$};
      \begin{scope}[xshift = 4.4cm, xscale = -1]
        \input{\imagesfolder/th_exdecofmodules2}
      \end{scope}
      \node at (6.6, 0) {$\oplus$};
     \begin{scope}[xshift = 8.8cm, yscale= -1, xscale = -1]
        \input{\imagesfolder/th_exdecofmodules2}
      \end{scope}
      \node at (11, 0) {$\oplus$};
     \begin{scope}[xshift = 13.2cm, yscale= -1, xscale = -1]
        \input{\imagesfolder/th_exdecofmodules3}
      \end{scope}
     \node at (15.4, 0) {$\oplus$};
     \begin{scope}[xshift = 17.6cm, yscale= 1, xscale = -1]
        \input{\imagesfolder/th_exdecofmodules4}
      \end{scope} 
     \node at (19.8, 0) {$\oplus$};
     \begin{scope}[xshift = 22.0cm, yscale= -1, xscale = -1]
        \input{\imagesfolder/th_exdecofmodules4}
      \end{scope} 
      \node at (24.2, 0) {$\oplus$};
     \begin{scope}[xshift = 26.4cm, yscale= 1, xscale = -1]
        \input{\imagesfolder/th_exdecofmodules5}
      \end{scope} 
    \end{scope} 
  \end{tikzpicture}
\hspace{1cm}
  \caption{Example of a decomposition of a web-module into indecomposable modules. All direct factors which are web-modules are obtained through idempotents associated with red graphs. The module $P$ is not a web-module but is a projective indecomposable module.}
\end{figure}
\end{landscape}

\section{Characterisation of indecomposable web-modules}
\label{sec:kup2RG}

\subsection{General View}
\label{sec:general-view}

The lemma~\ref{lem:monic2indec} states that the indecomposability of a web-modules $P_w$ can be deduced from the Laurent polynomial $\kup{\overline{w}w}$. In this section we will show a reciprocal statement. We first need a definition: 

\begin{dfn}
   Let $\epsilon$ be an admissible sequence of signs of length $n$, an $\epsilon$-web $w$ is said to be \emph{virtually indecomposable} if $\kup{\overline{w}w}$ is a monic symmetric Laurent polynomial of degree $n$.
An $\epsilon$-web which is not virtually indecomposable is \emph{virtually decomposable}. If $w$ is a virtually decomposable $\epsilon$-web, we define the \emph{level of $w$} to be the integer $\frac12 (\deg \kup{\overline{w}w} -n)$.
\end{dfn}
Despite of its fractional definition, the level is an integer.
With this definition, lemma~\ref{lem:monic2indec} can be rewritten:
\begin{lem}
\label{lem:VI2indec}
  If $w$ is a virtually indecomposable $\epsilon$-web, then $M(w)$ is an indecomposable $K^\epsilon$-module. 
\end{lem}
The purpose in this section is to prove a reciprocal statement in order to have:
\begin{thm}\label{thm:Charac}
  Let $\epsilon$ be an admissible sequence of signs of length $n$, and $w$ an $\epsilon$-web. Then the $K^\epsilon$-module $P_w$ is indecomposable if and only if $w$ is virtually indecomposable.
\end{thm}

\begin{req}
  Note that we do not suppose that $w$ is non-elliptic, but as a matter of fact, if $w$ is elliptic then $\kup{\overline{w}w}$ is not monic of degree $n$ and the module $P_w$ is decomposable.
\end{req}

To prove the unknown direction of theorem~\ref{thm:Charac} we use red graphs developed in the previous section and will show a more precise version of the theorem:

\begin{thm}\label{thm:thmwithRG}
  If $w$ is a non-elliptic virtually decomposable $\epsilon$-web of level $k$, then $w$ contains an admissible red graph of level $k$, hence $\mathrm{End}_{K^\epsilon}(P_w)$ contains a non-trivial idempotent and $P_w$ is decomposable.
\end{thm}

\begin{proof}[Proof of theorem \ref{thm:Charac} assuming  theorem \ref{thm:thmwithRG}] Let $w$ be a virtually decomposable $\epsilon$-web and let us denote by $k$ its level. From theorem~\ref{thm:thmwithRG} we know that there exists a red graph $G''$ of level $k$. But then, thanks to proposition~\ref{prop:2admissible}, there exists $G'$ a sub red graph of $G''$ which is admissible. And finally, the proposition \ref{prop:exist-exact} shows the existence of an exact red graph $G$ in $w$. We can apply theorem \ref{thm:RG2idempotent} to $G$ and this tells that $P_w$ is decomposable.  
\end{proof}

The proof of theorem \ref{thm:thmwithRG} is a recursion on the number of edges of the web $w$. But for the recursion to work, we need to handle elliptic webs as well. We will actually show the following:

\begin{prop}\label{prop:techRG}
  \begin{enumerate}
  \item\label{it:ptRG1} If $w$ is a $\partial$-connected $\epsilon$-web which is virtually decomposable of level $k\geq 1$ then there exists $S$ a stack of nice red graphs for $w$ of level greater or equal to $k$ such that $w_S$ is $\partial$-connected.
  \item\label{it:ptRG2} If $w$ is a $\partial$-connected $\epsilon$-web which is virtually decomposable of level $k\geq 1$, contains no digon and contains exactly one square which is supposed to be adjacent the unbounded face then there exists a nice red graph $G$ in $w$ of level greater or equal to $k$ such that $w_G$ is $\partial$-connected.
  \item\label{it:ptRG3} If $w$ is a non-elliptic $\epsilon$-web which is virtually decomposable of level $k\geq 0$ then there exists a nice red graph $G$ in $w$ of level greater of equal to $k$ such that $w_G$ is $\partial$-connected.
  \end{enumerate}
\end{prop}

Before proving the proposition we need to introduce \emph{stacks of red graphs} (see below), and the notion of $\partial$-connectedness (see section \ref{sec:part-conn}).
Then we will prove the proposition~\ref{prop:techRG} thanks to a technical lemma (lemma \ref{lem:tech}) which will be proven in section~\ref{sec:proof-lemmatech} after an alternative glance on red graphs (section~\ref{sec:new-approach-to-red-graph}).

\begin{req}
  It is easy to see that a non-elliptic superficial $\epsilon$-web contains no red graphs of non-negative level, hence this result is strictly stronger than the result of~\cite{LHR1}. 
\end{req}



\begin{dfn}
  Let $w$ be an $\epsilon$-web, \emph{a stack of red graphs $S=(G_1, G_2,\dots, G_l)$ for $w$} is a finite sequence of paired red graphs such that $G_1$ is a red graph of $w_1\eqdef w$, $G_2$ is a red graph of $w_2\eqdef w_{G_1}$, $G_3$ is a red graph of $w_3\eqdef (w_{G_1})_{G_2} = (w_2)_{G_2}$ etc. We denote $(\cdots((w_{G_1})_{G_2})\cdots)_{G_l}$  by $w_S$ and we denote $l$ by $l(S)$ and we say that it is \emph{the length of $S$}. We define the level of a stack to be the sum of the levels of the red graphs of the stack.\marginpar{change level into level and $k$-irregular into $k$ -virtually decomposable} 
\end{dfn}
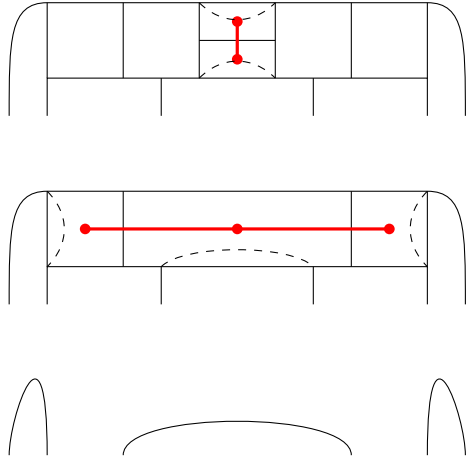
\begin{figure}[ht]
  \centering
  \begin{tikzpicture}
    \begin{scope}
  \draw (-3, 0).. controls +(0,1) and +(-0.5,0) .. (-2.5, 1.5) -- (2.5,1.5) ..  controls +(0.5,0) and +(0,1) .. (3,0);
\draw (-2.5,0)-- (-2.5, 0.5) -- (2.5, 0.5) -- (2.5,0);
\draw (-2.5, 0.5) -- +(0, 1);
\draw (-1.5, 0.5) -- +(0, 1);
\draw (-0.5, 0.5) -- +(0, 1);
\draw (0.5, 0.5) -- +(0, 1);
\draw (1.5, 0.5) -- +(0, 1);
\draw (2.5, 0.5) -- +(0, 1);
\draw (-0.5,1)-- +(1,0); 
\draw (1, 0) -- +(0, 0.5);
\draw (-1, 0) -- +(0, 0.5);
\fill[red] (0, 0.75) circle (2pt);
\fill[red] (0, 1.25) circle (2pt);
\draw [red, very thick] (0, 0.75) --  (0, 1.25);
\draw[dashed] (-0.5, 1.5) .. controls +(0.3,-0.3) and +(-0.3,-0.3) .. (0.5, 1.5);
\draw[dashed] (-0.5, 0.5) .. controls +(0.3,0.3) and +(-0.3,0.3) .. (0.5, 0.5);
\end{scope}
\begin{scope}[yshift = -2.5cm]
  \draw (-3, 0).. controls +(0,1) and +(-0.5,0) .. (-2.5, 1.5) -- (2.5,1.5) ..  controls +(0.5,0) and +(0,1) .. (3,0);
\draw (-2.5,0)-- (-2.5, 0.5) -- (2.5, 0.5) -- (2.5,0);
\draw (-2.5, 0.5) -- +(0, 1);
\draw (-1.5, 0.5) -- +(0, 1);
\draw (1.5, 0.5) -- +(0, 1);
\draw (2.5, 0.5) -- +(0, 1);
\draw (1, 0) -- +(0, 0.5);
\draw (-1, 0) -- +(0, 0.5);
\fill[red] (-2, 1) circle (2pt);
\fill[red] (0, 1) circle (2pt);
\fill[red] (2, 1) circle (2pt);
\draw [red, very thick] (2, 1) --  (-2, 1);
\draw[dashed] (-1, 0.5) .. controls +(0.3,0.3) and +(-0.3,0.3) .. (1, 0.5);
\draw[dashed] (-2.5, 1.5) .. controls +(0.3,-0.3) and +(0.3,0.3) .. (-2.5, 0.5);
\draw[dashed] (2.5, 1.5) .. controls +(-0.3,-0.3) and +(-0.3,0.3) .. (2.5, 0.5);
\end{scope}
\begin{scope}[yshift = -4.5cm]
\draw (-1.5, 0) .. controls +(0,0.6) and +(0,0.6) .. (1.5, 0);
\draw (-3, 0) .. controls +(0,0.5) and +(0,2) .. (-2.5, 0);
\draw (3, 0) .. controls +(0,0.5) and +(0,2) .. (2.5, 0);
\end{scope}
  \end{tikzpicture}
  \caption{A stack of red graphs of length 2.}
  \label{fig:stack}
\end{figure}

\begin{dfn}
A stack of red graphs is \emph{nice} if all its red graphs are nice. Note that in this case the pairing condition on red graphs is empty.
\end{dfn}


\subsection{The $\partial$-connectedness}
\label{sec:part-conn}

\begin{dfn}
  An $\epsilon$-web is \emph{$\partial$-connected} if every connected component of $w$ touches the border.
\end{dfn}
A direct consequence is that a $\partial$-connected $\epsilon$-web contains no circle. 
\begin{lem}
  A non-elliptic $\epsilon$-web is $\partial$-connected. 
\end{lem}
\begin{proof}
  An $\epsilon$-web which is not $\partial$-connected has a closed connected component, this connected component contains at least a circle, a digon or a square and hence is elliptic.
\end{proof}
\begin{lem}
  Let $w$ be a $\partial$-connected $\epsilon$-web with a digon, the web $\epsilon$-web $w'$ equal to $w$ except that the digon reduced (see figure~\ref{fig:reddig}) is still $\partial$-connected. In other words $\partial$-connectedness is preserved by digon-reduction.
\end{lem}
\begin{figure}[ht]
  \centering
  \begin{tikzpicture}
    \input{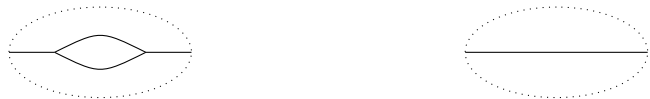}
  \end{tikzpicture}
  \caption{On the left $w$, on the right $w'$.}
  \label{fig:reddig}
\end{figure}

\begin{proof}
This is clear because every path in $w$ can be projected onto a path in $w'$.  
\end{proof}
Note that $\partial$-connectedness is not preserved by square reduction, see for example figure~\ref{fig:sq2partial}. 
\begin{figure}[ht]
  \centering
  \begin{tikzpicture}[scale=1.2]
    \input{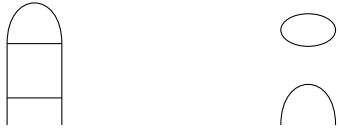}
  \end{tikzpicture}
  \caption{The $\partial$-connectedness is not preserved by square reduction.}
  \label{fig:sq2partial}
\end{figure}
However we have the following lemma:
\begin{lem}   If $w$ is a $\partial$-connected $\epsilon$-web which contains a square $S$ then one of the two $\epsilon$-webs obtained from $w$ by a reduction of $S$ (see figure~\ref{fig:tworeduction}) is $\partial$-connected.
 \end{lem}
 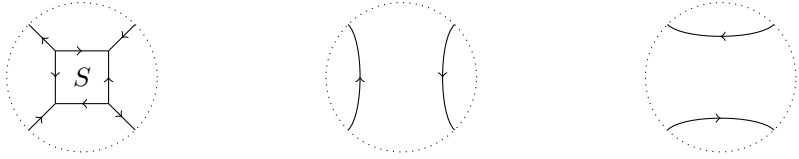
\begin{figure}[!ht]
   \centering
   \begin{tikzpicture}[scale=0.7]
     \begin{scope}[scale={1},decoration={markings, mark=at
     position 0.5 with {\arrow{>}}},postaction={decorate}]
   \draw[dotted] (1.5,1.5) circle (1.414 cm);
   \node at (1.5,1.5) {$S$};
   \draw[postaction = {decorate}] (0.5,0.5)--(1,1);
   \draw[postaction = {decorate}] (2.5,2.5)--(2,2);
   \draw[postaction = {decorate}] (2,1)--(1,1);
   \draw[postaction = {decorate}] (2,1)--(2,2);
   \draw[postaction = {decorate}] (2,1)--(2.5,0.5);
   \draw[postaction = {decorate}] (1,2)--(2,2);
   \draw[postaction = {decorate}] (1,2)--(1,1);
   \draw[postaction = {decorate}] (1,2)--(0.5,2.5);
 \end{scope}
\begin{scope}[xshift= 6cm, scale={1},decoration={markings, mark=at
     position 0.5 with {\arrow{>}}},postaction={decorate}]
   \draw[dotted] (1.5,1.5) circle (1.414 cm);
   \draw[postaction = {decorate}] (0.5,0.5).. controls +(0.3,0.3) and +(0.3,-0.3) ..(0.5,2.5);
   \draw[postaction = {decorate}] (2.5,2.5).. controls +(-0.3,-0.3) and +(-0.3,0.3) ..(2.5,0.5);
 \end{scope}
\begin{scope}[xshift= 12cm, scale={1},decoration={markings, mark=at
     position 0.5 with {\arrow{>}}},postaction={decorate}]
   \draw[dotted] (1.5,1.5) circle (1.414 cm);
   \draw[postaction = {decorate}] (0.5,0.5).. controls +(0.3,0.3) and +(-0.3,0.3) ..(2.5,0.5);
   \draw[postaction = {decorate}] (2.5,2.5).. controls +(-0.3,-0.3) and +(0.3,-0.3) ..(0.5,2.5);
 \end{scope}
   \end{tikzpicture}
   \caption{On the right the $\epsilon$-web $w$ with the square $S$, on the middle and on the right, the two reductions of the square $S$.}
   \label{fig:tworeduction}
 \end{figure}
 \begin{proof} Consider the oriented graph $\tilde{w}$ obtained from $w$ by removing the square $S$ and the 4 half-edges adjacent to it (see figure~\ref{fig:sqremoved}). 
   \begin{figure}[ht]
     \centering
     \begin{tikzpicture}[scale= 0.7]
       \begin{scope}[scale={1},decoration={markings, mark=at
     position 0.5 with {\arrow{>}}},postaction={decorate}]
   \draw[dotted] (1.5,1.5) circle (1.414 cm);
   \node at (1.5,1.5) {$S$};
   \draw[postaction = {decorate}] (0.5,0.5)--(1,1);
   \draw[postaction = {decorate}] (2.5,2.5)--(2,2);
   \draw[postaction = {decorate}] (2,1)--(1,1);
   \draw[postaction = {decorate}] (2,1)--(2,2);
   \draw[postaction = {decorate}] (2,1)--(2.5,0.5);
   \draw[postaction = {decorate}] (1,2)--(2,2);
   \draw[postaction = {decorate}] (1,2)--(1,1);
   \draw[postaction = {decorate}] (1,2)--(0.5,2.5);
 \end{scope}
\begin{scope}[xshift= 8cm, scale={1},decoration={markings, mark=at
     position 0.5 with {\arrow{>}}},postaction={decorate}]
   \draw[dotted] (1.5,1.5) circle (1.414 cm);
   \draw[postaction = {decorate}] (0.5,0.5)--(0.8,0.8);
   \draw[postaction = {decorate}] (2.5,2.5)--(2.2,2.2);
   \draw[postaction = {decorate}] (2.2,0.8)--(2.5,0.5);
   \draw[postaction = {decorate}] (0.8,2.2)--(0.5,2.5);
 \end{scope}
     \end{tikzpicture}
     \caption{On the left $w$, on the right $\tilde{w}$.}
     \label{fig:sqremoved}
   \end{figure}
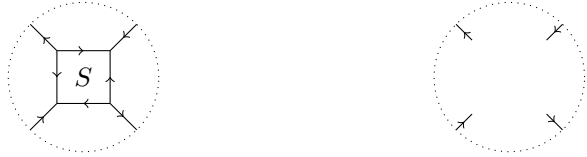
We obtain a graph with 4 less cubic vertices than $w$ and 4 more vertices of degree 1 than $w$. We call $E_S$ the cyclically ordered set of the 4 vertices of $\tilde{w}$ of degree 1 next to the removed square $S$. The orientations of the vertices in $E_S$ are $(+,-,+,-)$. Note that in $\tilde{w}$, the flow modulo 3 is preserved everywhere, so that the sum of orientation of vertices of degree 1 of any connected component must be equal to 0 modulo 3. Suppose now that there is a connected component $t$ of $\tilde{w}$ which has all its vertices of degree 1 in $E_S$, the flow condition implies that either all vertices of $E_S$ are vertices of $t$ or exactly two consecutive vertices of $E_S$ are vertices of $t$, or that $t$ has no vertex of degree 1. The first situation cannot happen because by adding the square to $t$ we would construct a free connected component of $w$ which is supposed to be $\partial$-connected, the last situation neither for the same reason. So the only thing that can happen is the second situation. If there were two different connected components $t_1$ and $t_2$ of $\tilde{w}$ such that $t_1$ and $t_2$ have all their vertices of degree 1 in $E_S$, then adding the square to $t_1\cup t_2$ would lead to a free connected component of $w$, so their is at most one connected component of $\tilde{w}$ with all this vertex of degree 1 in $E_S$ call this vertices $e_+$ and $e_-$, and call $e'_+$ and $e'_-$ the two other vertices of $E_S$ (the indices gives the orientation). If we choose $w'$ to be the $\epsilon$-web corresponding  to the smoothing which connects $e_+$ with $e'_-$ and $e_-$ with $e'_+$, then $w'$ is $\partial$-connected.
\end{proof}
\begin{dfn}
  Let $w$ be a $\partial$-connected $\epsilon$-web and $S$ a square in $w$. The square $S$ is \emph{a $\partial$-square} if the two $\epsilon$-webs $w_{=}$ and $w_{||}$  obtained from $w$  by the two reductions by the square $S$ are $\partial$-connected.
\end{dfn}
 \begin{lem}\label{lem:pc2ps}
   If $w$ is a $\partial$-connected web, then either it is non-elliptic, or it contains either a digon or a $\partial$-square.
 \end{lem}
 \begin{proof}
   Suppose that $w$ is not non-elliptic. As $w$ is  $\partial$-connected it contains no circle. If  must contains at least a digon or a square, if it contains a digon we are done, so suppose $w$ contains no digon. We should show that at least one square is a $\partial$-square. Suppose that there is no $\partial$-square, it means that for every square $S$, there is a reduction such that the $\epsilon$-web resulting $w_{s(S)}$ obtained by replacing $w$ by the reduction has a free connected component $t_S$. Let us consider a square $S_0$ such that $t_{S_0}$ is as small as possible (in terms of number of vertices for example). The web $t_{S_0}$ is closed and connected, so that either it is a circle, or it contains a digon or at least two square. If $t_{S_0}$ is a circle then $w$ contains a digon just next to the square $S_0$, and we excluded this case (see figure \ref{fig:circle2digon}). 
   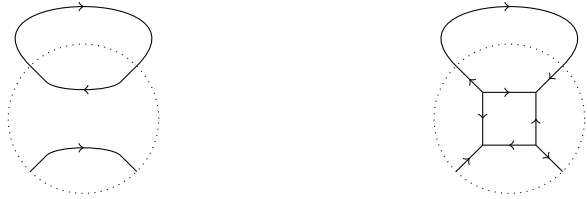
\begin{figure}[ht]
     \centering
     \begin{tikzpicture}[scale= 0.7]
       \begin{scope}[xshift= 8cm,scale={1},decoration={markings, mark=at
     position 0.5 with {\arrow{>}}},postaction={decorate}]
   \draw[dotted] (1.5,1.5) circle (1.414 cm);
   \node at (1.5,1.5) {};
   \draw[postaction = {decorate}] (0.5,0.5)--(1,1);
   \draw[postaction = {decorate}] (2.5,2.5)--(2,2);
   \draw[postaction = {decorate}] (2,1)--(1,1);
   \draw[postaction = {decorate}] (2,1)--(2,2);
   \draw[postaction = {decorate}] (2,1)--(2.5,0.5);
   \draw[postaction = {decorate}] (1,2)--(2,2);
   \draw[postaction = {decorate}] (1,2)--(1,1);
   \draw[postaction = {decorate}] (1,2)--(0.5,2.5);
\draw[postaction = {decorate}] (0.5,2.5).. controls (-1,4) and (4,4).. (2.5, 2.5); 
 \end{scope}
\begin{scope}[xshift= 0cm, scale={1},decoration={markings, mark=at
     position 0.5 with {\arrow{>}}},postaction={decorate}]
   \draw[dotted] (1.5,1.5) circle (1.414 cm);
   \draw[postaction = {decorate}] (0.5,0.5)--(0.8,0.8) .. controls (1, 1) and (2,1) ..(2.2,0.8)--(2.5,0.5);
   \draw[postaction = {decorate}] (2.5,2.5) --  (2.2,2.2) .. controls (2,2) and (1,2) .. (0.8,2.2)--(0.5,2.5);
   \draw[postaction = {decorate}] (0.5,2.5).. controls (-1,4) and (4,4).. (2.5, 2.5); 
 \end{scope}
     \end{tikzpicture}
     \caption{on the left $w_{S_0}$, on the right $w$. If $t_{S_0}$ is a circle, then $w$ contains a digon.}
     \label{fig:circle2digon}
   \end{figure}
If it contains a digon, the digon must be next to where $S_0$ was smoothed else the digon would already be in $w$. It appears hence that the digon comes from a square $S_1$ in $w$ ($S_1$ is adjacent to $S_0$), and $t_{S_1}$ has two vertices less than $T_{S_0}$ which is excluded (see figure~\ref{fig:digon2square}).
\begin{figure}[ht]
  \centering
  \begin{tikzpicture}[scale=0.7]
    \begin{scope}[xshift= 8cm,scale={1},decoration={markings, mark=at
     position 0.5 with {\arrow{>}}},postaction={decorate}]
   \draw[dotted] (1.5,1.5) circle (1.414 cm);
   \node at (1.5,1.5) {};
   \draw[postaction = {decorate}] (0.5,0.5)--(1,1);
   \draw[postaction = {decorate}] (2.5,2.5)--(2,2);
   \draw[postaction = {decorate}] (2,1)--(1,1);
   \draw[postaction = {decorate}] (2,1)--(2,2);
   \draw[postaction = {decorate}] (2,1)--(2.5,0.5);
   \draw[postaction = {decorate}] (1,2)--(2,2);
   \draw[postaction = {decorate}] (1,2)--(1,1);
   \draw[postaction = {decorate}] (1,2)--(0.5,2.5);
   \draw[postaction = {decorate}] (0.5,2.5).. controls (0,3) and +(0,0) ..(0,3);
   \draw[postaction= {decorate}] (3,3) -- (2.5, 2.5); 
   \draw[postaction= {decorate}] (3,3) -- (3.5, 3.5);       
   \draw[postaction={decorate}](-0.5, 3.5) -- (0,3);
   \draw[postaction= {decorate}] (3,3)-- (0, 3);
 
 \end{scope}
\begin{scope}[xshift= 0cm, scale={1},decoration={markings, mark=at
     position 0.5 with {\arrow{>}}},postaction={decorate}]
   \draw[dotted] (1.5,1.5) circle (1.414 cm);
   \draw[postaction = {decorate}] (0.5,0.5)--(0.8,0.8) .. controls (1, 1) and (2,1) ..(2.2,0.8)--(2.5,0.5);
   \draw[postaction = {decorate}] (2.5,2.5) --  (2.2,2.2) .. controls (2,2) and (1,2) .. (0.8,2.2)--(0.5,2.5);
   \draw[postaction = {decorate}] (0.5,2.5).. controls (0,3) and +(0,0) ..(0,3);
   \draw[postaction= {decorate}] (3,3) -- (2.5, 2.5); 
  \draw[postaction= {decorate}] (3,3) -- (3.5, 3.5);       
   \draw[postaction={decorate}](-0.5, 3.5) --(0,3);
   \draw[postaction= {decorate}] (3,3)-- (0, 3);
 \end{scope}
  \end{tikzpicture}
  \caption{On the left $w_{S_0}$, on the right $w$. If $t_{S_0}$ contains a digon then $w$ contains a square adjacent to $S_0$.}
  \label{fig:digon2square}
\end{figure}
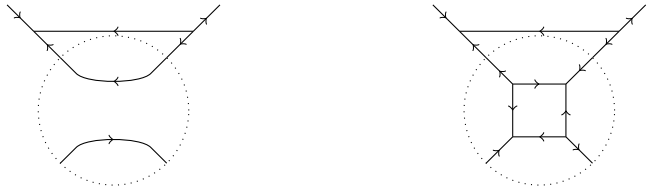
The  closed web $t_{S_0}$ contains at least two squares so that we can pick up one, we denote it by $S'$, which is far from $S_0$ and hence comes from a square in $w$. Now at least one of the two smoothings of the square $S'$ must disconnect $t_{s_0}$ else the square $S'$ would be a $\partial$-square in $w$. But as it disconnects $t_{S_{0}}$, $t_{S'}$ is a strict sub graph of $t_{S_0}$, and this contradict the minimality of $S_0$. And this concludes that $w$ must contain a $\partial$-square.
 \end{proof}

\subsection{Proof of proposition~\ref{prop:techRG}}
\label{sec:proof-propkey}

In this section we prove the proposition~\ref{prop:techRG} admitting the following technical lemma:

\begin{lem}\label{lem:tech}
  Let $w$ be a $\partial$-connected $\epsilon$-web which contains, no digon and one square which touches the unbounded face. Let $G$ be a nice red graphs of $w$ and $G'$ a nice red graph of $w_{G}$ such that $w_G$ and $w_{G'}$ are $\partial$-connected, then there exists $G''$ a red graph of $w$ such that $(w_{G})_{G'} = w_{G''}$ and the level of $G''$ is bigger or equal to the level of $G$ plus the level of $G'$.
\end{lem}

This lemma says that under certain condition one can ``flatten'' two red graphs.

\begin{proof}[Proof of proposition~\ref{prop:techRG}]
As we announced this will be done by recursion on the number of edges of $w$. 
We supposed than \ref{it:ptRG1}, \ref{it:ptRG2} and \ref{it:ptRG3} hold for all $\epsilon$-webs with strictly less than $n$ edges, and we 
consider an $\epsilon$-web with $n$ edges. Note that whenever $w$ is non-elliptic the statement \ref{it:ptRG3} is stronger than the statement \ref{it:ptRG1}, so that we won't prove \ref{it:ptRG1} in this case.
We first prove \ref{it:ptRG1}:

If $w$ contains a digon, then we apply the result on $w'$ the $\epsilon$-web similar to $w$ but with the digon reduced (\ie replaced by a single strand). 
The red graph $G$ which consist of only one edge (the digon) and no edge is nice and has level equal to 1 (see figure~\ref{fig:rganddig}).
\begin{figure}[ht]
  \centering
  \begin{tikzpicture}[scale = 0.7]
    \input{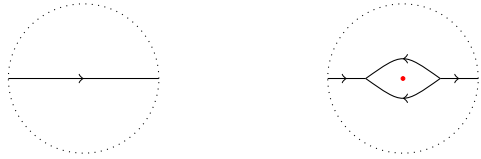}
  \end{tikzpicture}
  \caption{On the left $w'$, on the right $w$ with the red graph $G$.}
  \label{fig:rganddig}
\end{figure}

If $w'$ is not virtually decomposable or virtually decomposable of level 0, then $w$ is virtually decomposable of level 1. In this case, the stack with only one red graph equal to $G$ is convenient and we are done.
Else we know that $w'$ is of level $k-1$ and that there exists a nice stack of red graphs $S'$ of level $k-1$ in $w'$ and we consider the stack $S$ equal to the concatenation of $G$ with $S'$, it is a nice stack of red graphs of level $k$ and we are done.

Suppose now that the $\epsilon$-web $w$ contains no digon, but a square, then it contains a $\partial$-square (see lemma \ref{lem:pc2ps}). Suppose that  the level of $w$ is $k\geq 1$ (else there is nothing to show), then at least one of the two reductions is virtually decomposable of level $k$ (this is a Cauchy-Schwartz inequality see~\cite[Section 1.1]{LHRThese} for details). Then we consider $w'$ the $\epsilon$-web obtained by a reduction of the square so that it is of level $k$. From the induction hypothesis we know that there exists a stack of red graphs $S'$ in $w'$ of level $k$. If all the red graphs of $S'$ are far from the location of the square, then we can transform the stack $S'$ into a stack of $w$ with the same level. Else, we consider $G'$ the first red graph of $S'$ which is close from the square location and according to the situation we define $G$ by the moves given on figure~\ref{fig:S2Sp}.

\begin{figure}[ht]
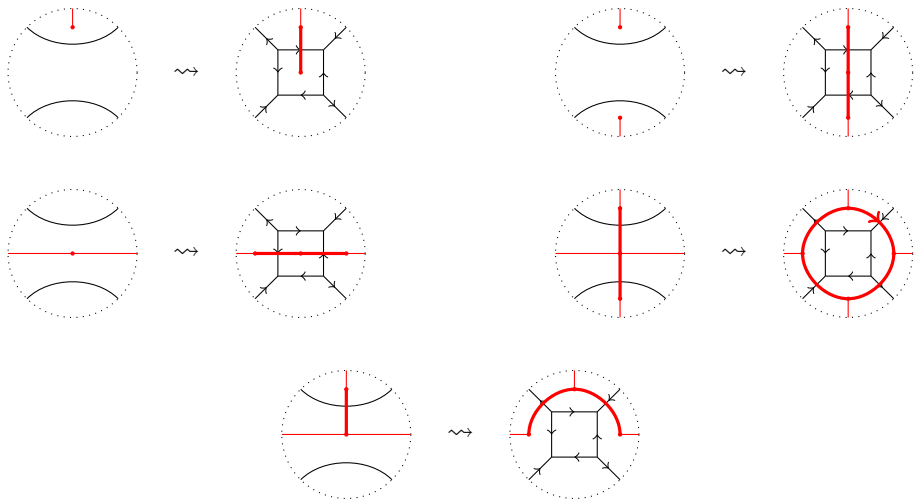

  \centering
  \begin{tikzpicture}[scale=0.6]
    \begin{scope}[xshift = 0cm, yshift= 8cm]
      \input{\imagesfolder/sw_fig1}
    \end{scope}
    \begin{scope}[xshift = 0cm, yshift= 4cm]
      \input{\imagesfolder/sw_fig2}
    \end{scope}
    \begin{scope}[xshift = 6cm, yshift= 0cm]
      \input{\imagesfolder/sw_fig3}
    \end{scope}
    \begin{scope}[xshift = 12cm, yshift= 8cm]
      \input{\imagesfolder/sw_fig4}
    \end{scope}
    \begin{scope}[xshift = 12cm, yshift= 4cm]
      \input{\imagesfolder/sw_fig5}
    \end{scope}
  \end{tikzpicture}
  \caption{Transformations of $G'$ to obtain $G$.}
  \label{fig:S2Sp}
\end{figure}
Replacing $G'$ by $G$ we can transform, the stack $S'$ into a stack for the $\epsilon$-web $w$. 

We now prove \ref{it:ptRG2}.

From what we just did, we know that $w$ contains a nice stack of red graphs of level $k$. Among all the nice stacks of red graphs of $w$ with level greater or equal to $k$, we choose one with a minimal length, we call it $S$. If its length were greater or equal to $2$, then lemma \ref{lem:tech} would tell us that we could take the first two red graphs and replace them by just one red graph with a level bigger or equal to the sum of their two levels, so that $S$ would not be minimal, this prove that $S$ has length $1$, therefore, $w$ contains a nice red graph of level at least $k$.

We now prove \ref{it:ptRG3}.

The border of $w$ contains at least a $\cap$, a $\lambda$, or an $H$ (see figure~\ref{fig:dfnUHY}). In the two first cases, we can consider $w'$ the $\epsilon$-web with the $\cap$ removed or the $\lambda$ replaced by a single strand, then $w'$ is non-elliptic and virtually decomposable of level $k$ and there exists a nice red graph in $w'$ of level at least $k$, this red graph can be seen as a red graph of $w$, and we are done. If the border of $w$ contains no $\lambda$ and no $\cap$, then it must contains an $H$. There are two ways to reduce the $H$ (see figure~\ref{fig:Hreduced}). At least one of the two following situation happens: 
$w_{||}$ is virtually decomposable of level $k$ or $w_{\_}$ is virtually decomposable of level $k+1$. 
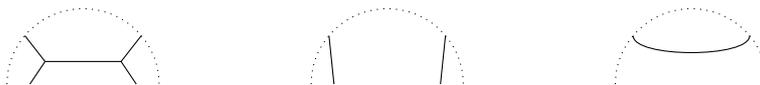
\begin{figure}[ht]
  \centering
  \begin{tikzpicture}
     \begin{scope}
  \draw[dotted] (-1,0) arc (180:0:1cm);
  \draw (-0.7,0) -- (-0.5, 0.3) -- (140:1);
\draw (0.7,0) -- (0.5, 0.3) -- (40:1);
\draw (-0.5,0.3) -- (0.5,0.3);
\end{scope}
\begin{scope}[xshift = 4cm]
  \draw[dotted] (-1,0) arc (180:0:1cm);
  \draw (-0.7,0)  -- (140:1);
\draw (0.7,0)  -- (40:1);
\end{scope}
\begin{scope}[xshift= 8cm]
  \draw[dotted] (-1,0) arc (180:0:1cm);
  \draw (40:1).. controls +(0,-0.3) and +(0,-0.3) .. (140:1);
\end{scope}
  \end{tikzpicture}
  \caption{The $H$ of $w$ (on the left) and its two reductions: $w_{||}$ (on the middle) and $w_=$ (on the right).}
  \label{fig:Hreduced}
\end{figure}
In the first situation, one can do the same reasoning as before: 
$w_{||}$ being non-elliptic, the induction hypothesis gives that we can find a nice red graph of level at least $k$ in $w_{||}$, this red graph can be seen as a red graph of $w$ and we are done. 
In the second situation, we consider $w_{\_}$, we can apply the induction hypothesis to $w_{\_}$ (we are either in case \ref{it:ptRG2} or in case \ref{it:ptRG3}), so we can find a nice red graph of level at least $k+1$, coming back to $H$ this gives us a red graph of level at least $k$ (but maybe not nice), and we can conclude via the lemma~\ref{lem:RGinNE2niceRG}.
\end{proof}

\subsection{A new approach to red graphs.}
\label{sec:new-approach-to-red-graph}

In this section we give an alternative approach to red graphs: instead of starting with a web and simplifying it with a red graph we construct a red graph from a web and a simplification of this web. For this we need a property of webs that we did not use so far.
\begin{prop}
Let $w$ be a closed web, then it admits a (canonical) face-3-colouring with the unbounded face coloured $c\in \ZZ/3\ZZ$. We call this colouring \emph{the face-colouring of base $c$} of $w$. When $c$ is not mentioned it is meant to be $0$.
\end{prop}
\begin{proof}
We will colour connected components of $\RR^2\setminus w$ with elements of $\ZZ/3\ZZ$. 
We can consider the only unbounded component $U$ of $\RR^2\setminus w$. We colour it by $c$, then for each other connected component $f$, we consider $p$ an oriented path from a point inside $U$ to a point inside $f$, which crosses the $w$ transversely, we then define the colour of $f$ to be the sum (modulo 3) of the signs of the intersection of the path $p$ with $w$ (see figure~\ref{fig:signcol} for signs convention). This does not depend on the path because in $w$ the flow is always preserved modulo 3. And, by definition, two adjacent faces are separated by an edge, so that they do not have the same colour.
\begin{figure}[ht]
  \centering
  \begin{tikzpicture}
    \begin{scope}
\draw[->] (0,0) -- (0,2);
\draw[dashed,->] (-1,0.8) .. controls +(0.4,-0.2) and +(-0.4,0.2) ..(1,1.2);
\node at (-0.8, 0.5) {$p$};
\end{scope}
\begin{scope}[xshift= 4cm]
\draw[->] (0,0) -- (0,2);
\draw[dashed,<-] (-1,0.8) .. controls +(0.4,-0.2) and +(-0.4,0.2) ..(1,1.2);
\node at (-0.8, 0.5) {$p$};
\end{scope}
  \end{tikzpicture}
  \caption{On the left a positive crossing, on the right a negative one. The path is dashed and the web is solid.}
  \label{fig:signcol}
\end{figure}
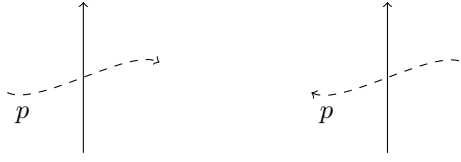
\end{proof}

\begin{cor}
  Let $w$ be an $\epsilon$-web, then the connected component of $\RR\times \RR_+ \setminus w$ admits a (canonical) 3-colouring with the unbounded connected component coloured by $c$. We call this colouring \emph{the face-colouring of base $c$} of $w$.
\end{cor}

\begin{proof}
  We complete $w$ with $\overline{w}$ and we use the previous proposition to obtain a colouring of the faces. This gives us a canonical colouring for $\RR\times \RR_+ \setminus w$.
\end{proof}

Note that in this corollary  it is important to consider the connected component of  $\RR\times \RR_+ \setminus w$ instead of the connected component of $\RR^2 \setminus w$. Let us formalise this in a definition.

\begin{dfn}
If $w$ is an $\epsilon$-web, the \emph{regions} of $w$ are the connected components of $\RR\times \RR_+ \setminus w$. The \emph{faces} of $w$ are the regions which do not intersect $\RR \times \{0\}$.
\end{dfn}

\begin{dfn}
  Let $w$ be an $\epsilon$-web, an $\epsilon$-web $w'$ is a
  \emph{simplification of $w$} if 
  \begin{itemize}
  \item the set of vertices of $w'$ is
  included in the set of vertices of $w$,
\item  every edge $e$ of $w'$ is
  divide into an odd number of intervals
  $([a_i,a_{i+1}])_{i\in[0,2k]}$ such that for every $i$ in $[0, k]$,
  $[a_{2i},a_{2i+1}]$ is an edge of $w$ (with matching orientations)
  and for every $i$ in $[0, k-1]$, $[a_{2i+1},a_{2i+2}]$ lies in the
  faces of $w$ opposite to $[a_{2i}, a_{2i+1}]$ with respect to $a_{2i+1}$ (see figure~\ref{fig:simplificationedge}).
  \end{itemize}
  \begin{figure}[ht]
    \centering
    \begin{tikzpicture}[scale=0.7]
      \begin{scope}[decoration={markings, mark=at
     position 0.7 with {\arrow{>}}},postaction={decorate}]
\draw[dotted] (0,0) circle (1.5);
\draw[orange, line width= 2pt, postaction = decorate ] (290:1.5).. controls +(0,0) and (0, -0.5).. (0,0) -- (0,1.5);
\draw[->] (0,0) -- (0,1.5);
\draw[->] (0,0) -- (210:1.5);
\draw[->] (0,0) -- (330:1.5);
\node at (0.35,0.2) {$a_k$};
\end{scope}
    \end{tikzpicture}
    \caption{Local picture around $a_k$. The edge of $w'$ is orange and large, while the $\epsilon$-web $w$ is black and thin.}
    \label{fig:simplificationedge}
  \end{figure}
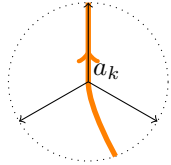
\end{dfn}
\begin{figure}[ht]
  \centering
  \begin{tikzpicture}[scale= 0.8]
\begin{scope}[xshift =0cm, orange, line width= 3pt]
   \foreach \n/\a in {a/30,b/90,c/150,d/210,e/270,f/330}
{
\draw (\a:2.5) ++(\a+60:1) -- ++(\a+180:0.5).. controls (\a+30 :2) and (\a+30:2) .. ++(\a+120:1)-- +(\a+60:0.5);
}
\end{scope}
 \begin{scope}
   \foreach \n/\a in {a/30,b/90,c/150,d/210,e/270,f/330}
{
\draw (\a:1) -- (\a:2);
\draw (\a:1) -- (\a+60:1);
\draw (\a:2) -- ++(\a-60:1) -- ++(\a-120:1) -- ++(\a-180:1);
\draw (\a:2) ++(\a-60:1) -- ++(\a:0.5);
\draw (\a:2) ++(\a-60:1) ++(\a-120:1) -- ++(\a-60:0.5);
}
\end{scope}
  \end{tikzpicture}
  \caption{The $\epsilon$-web $w$ (on black) and $w_0$ (in orange) of proposition~\ref{prop:Pwdec} seen in terms of simplification.}
  \label{fig:exsimplification}
\end{figure}
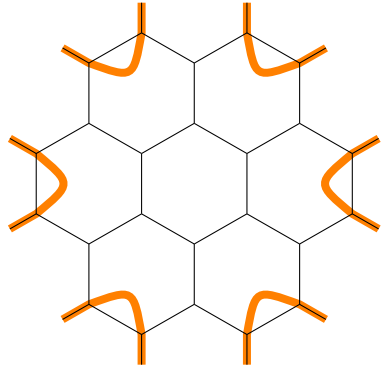

\begin{lem}\label{lem:coherent-col}
Let $w$ be a $\epsilon$-web and $w'$ a $\partial$-connected simplification of $w$ If $e$ is an edge of $w$ which is as well a (part of an) edge of $w'$, then in the face-colourings of base $c$ of $w$ and $w'$, the the regions adjacent to $e$ in $w$ and in $w'$ are coloured in the same way.
\end{lem}

\begin{proof}
  This is an easy recursion on how $e$ is far from the border.
\end{proof}



Note that in this definition the embedding of $w'$ with respect to $w$ is very important.
\begin{dfn}
   Let $w$ be an $\epsilon$-web and $w'$ a simplification of $w$. We consider the face-colourings of $w$ and $w'$. A face $f$ of $w$ lies in one or several regions of $w'$. This face $f$ is \emph{essential with respect to $w'$} if all regions of $w'$ it intersects do not have the same colour as $f$.
\end{dfn}
\begin{req}
  We could have write this definition with region of $w$ instead of faces, but it is easy to see that a region of which is not a face $w$ is never essential.
\end{req}
\begin{lem}\label{lem:intersect2essential}
   Let $w$ be a $\partial$-connected $\epsilon$-web and $w'$ a $\partial$-connected simplification of $w$. If a face $f$ of $w$ is not essential with respect to $w'$ then it intersects only one region of $w'$.
\end{lem}
\begin{proof}
  Consider a face $f$ of $w$  which intersects more than one region of $w'$. We will prove that it is essential with respect to $w'$. Consider an edge $e'$ of $w'$ which intersects $f$ (there is at least one by hypothesis), when we look next to the border of $f$ next to $e'$ we find a vertex $v$ of $w$ (see figure~\ref{fig:localfacefp}).

  \begin{figure}[ht]
    \centering
    \begin{tikzpicture}[scale=0.7]
      \begin{scope}[decoration={markings, mark=at
     position 0.7 with {\arrow{>}}},postaction={decorate}]
\fill[gray] (0,0) -- (210:1.5) arc (210:330:1.5) -- cycle;
\node at (-0.5, -0.9) {$f'$};
\draw[dotted] (0,0) circle (1.5);
\draw[orange, line width= 2pt] (290:1.5).. controls +(0,0) and (0, -0.5).. (0,0) -- (0,1.5);
\draw (0,0) -- (0,1.5);
\draw (0,0) -- (210:1.5);
\draw (0,0) -- (330:1.5);
\node at (0.3,0.2) {$v$};
\node[orange] at (310:1.1) {$e'$};
\end{scope}
    \end{tikzpicture}
    \caption{A part of the face $f'$ next to an edge $e'$ of $w'$. Above $v$ the colours of $w$ and $w'$ are coherent thanks to lemma~\ref{lem:coherent-col}.}
    \label{fig:localfacefp}
  \end{figure}
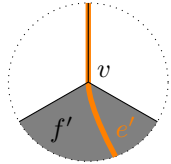
We want to prove that none of the faces of $w'$ which are adjacent to $e'$ has the same colour as the face $f$. This follows from the lemma \ref{lem:coherent-col}, and from the fact that the part of $e'$  above $v$ is an edge of $w$ (see figure~\ref{fig:localfacefp}).
\end{proof}
\begin{cor}\label{cor:how-are-essential-faces}
   Let $w$ be a $\partial$-connected $\epsilon$-web and $w'$ a $\partial$-connected simplification of $w$. If a face $f$ of $w$ intersects a region of $w'$ which has the same colour, it is not essential.
\end{cor}


\begin{prop}
   Let $w$ be a $\partial$-connected $\epsilon$-web (this implies that every face of $w$ is diffeomorphic to a disk) and $w'$ a $\partial$-connected simplification of $w$. Then there exists a (canonical) paired red graph $G$ such that $w'$ is equal to $w_G$. We denote it by $G_{w\to w'}$.
\end{prop}
\marginpar{face vs connected component of...}
\begin{proof}
  We consider the canonical colourings of the faces of $w$ and $w'$. The red graph $G$ is the induced sub-graph of $w^\star$ (the dual graph of $w$) whose vertices are essential faces of $w$ with respect to $w'$. The pairing is given by the edges of $w'$. We need to prove first that this is indeed a red graph, and in a second step that $w_G = w'$.
We consider a vertex $v$ of $w$ and the 3 regions next to it. We want to prove that at least one of the 3 regions is not essential with respect to $w'$. If the vertex $v$ is a vertex of $w'$ then lemma~\ref{lem:coherent-col} and corollary \ref{cor:how-are-essential-faces} give that none of the three regions is essential. Else, $v$ either lies inside an edge of $w'$ or it lies in a face of $w'$ (see figure~\ref{fig:3possvertex}).

\begin{figure}[ht]
  \centering
  \begin{tikzpicture}[scale=0.7]
    \begin{scope}[decoration={markings, mark=at
     position 0.7 with {\arrow{>}}},postaction={decorate}]
\draw[dotted] (0,0) circle (1.5);
\draw[orange, line width= 2pt] (0,0) -- (0,1.5);
\draw[orange, line width= 2pt] (0,0) -- (210:1.5);
\draw[orange, line width= 2pt] (0,0) -- (330:1.5);
\draw (0,0) -- (0,1.5);
\draw (0,0) -- (210:1.5);
\draw (0,0) -- (330:1.5);
\node at (0.3,0.3) {$v$};
\end{scope}
\begin{scope}[xshift = 4cm,decoration={markings, mark=at
     position 0.7 with {\arrow{>}}},postaction={decorate}]
\draw[dotted] (0,0) circle (1.5);
\draw[orange, line width= 2pt] (290:1.5).. controls +(0,0) and (0, -0.5).. (0,0) -- (0,1.5);
\draw (0,0) -- (0,1.5);
\draw (0,0) -- (210:1.5);
\draw (0,0) -- (330:1.5);
\node at (0.3,0.3) {$v$};
\end{scope}
\begin{scope}[xshift = 8cm, decoration={markings, mark=at
     position 0.7 with {\arrow{>}}},postaction={decorate}]
\draw[dotted] (0,0) circle (1.5);
\draw (0,0) -- (0,1.5);
\draw (0,0) -- (210:1.5);
\draw (0,0) -- (330:1.5);
\node at (0.3,0.3) {$v$};
\end{scope}
  \end{tikzpicture}
  \caption{The three configurations for the vertex $v$ of $w$: it is a vertex of $w'$ (on the left), it lies inside an edge of $w'$ (on the middle), it lies inside a region of $w'$ (on the right).}
  \label{fig:3possvertex}
\end{figure}
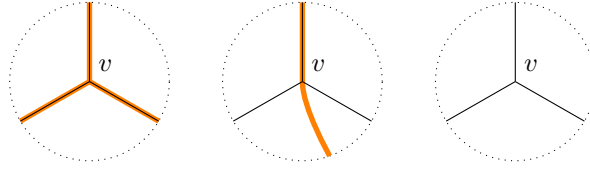

Consider the first situation: one of the 3 regions intersects two different regions of $w'$ hence it is essential thanks to lemma~\ref{lem:intersect2essential}, the two others are not thanks to corollary~\ref{cor:how-are-essential-faces}.

In the last situation, the 3 regions have different colours so that one of them has the same colour than the colour of the region of $w'$ where $v$ lies in, this region is therefore not essential (corollary \ref{cor:how-are-essential-faces}). This shows that $G$ is a red graph (we said nothing about the admissibility).

Let us now show that $w'=w_G$. We consider a collection $(N_{f})_{f\in V(G)}$ of regular neighbourhoods of essential faces of $w$ with respect to $w'$. Let us first show that for every essential face $f$ of $w$, if $N_f$ is a regular neighbourhood of $f$, the restriction of $w_G$ and $w'$ matchs. As $f$ is essential it is a vertex of $G$. Then the restriction of $w_G$ to $N_f$ is just a collection of strands joining the border to the border, just as $w'$.

In $\RR\times \RR_+\setminus \left(\bigcup_{f\in V(G)}N_f\right)$  the $\epsilon$-webs $w'$ and $w_G$ are both equal to $w$. This complete the proof.
\marginpar{illustration with $w$ and $w_0$}
\end{proof}
Note that $G_{w\to w'}$ depends on how $w'$ is embedded to see it as a simplification of $w$.

\begin{dfn}
Let $w$ a $\epsilon$-web and $w'$ a simplification of $w$, then the simplification is \emph{nice}, if for every region $r$ of $w$, $r\cap w'$ is either the empty set or connected.
\end{dfn}
We have the natural lemma:
\begin{lem}
  Let $w$ be a $\partial$-connected $\epsilon$-web and $w'$ a $\partial$-connected simplification of $w$. The simplification is nice if and only if the red graph $G_{w\to w'}$ is nice
\end{lem}
\begin{proof}
  Thanks to lemma \ref{lem:intersect2essential}, only essential faces of $w$ with respect to $w'$ can have non trivial intersection with $w'$, and for an essential face $f$, twice the number of connected component of $f\cap w'$ is equal to the exterior degree of the vertex of $G_{w\to w'}$ corresponding to $f$.
\end{proof}

\begin{lem}\label{lem:essfacesmatters}
  If $w$ is a $\partial$-connected $\epsilon$-web, and $w'$ is a $\partial$-connected simplification of $w$. Then the level of $G_{w\to w'}$ is given by the following formula:
\[
i(G_{w\to w'}) = 2\#\{\textrm{essential faces of $w$ wrt. $w'$}\} - \frac{\#V(w) - \# V(w')}{2}.
\]
\end{lem}
This shows that the embedding of $w'$ influences the level of $G_{w\to w'}$ only on the number of essential faces of $w$ with respect  to $w'$
\begin{proof}
  The level of a red graph $G$ is given by:
\[
i(G) = 2\#V(G) - \#E(G) - \sum_{f\in V(G)}\frac{\ed(f)}{2}.
\]
By definition of $G_{w\to w'}$, we have:
\[ \{\textrm{essential faces of $w$ wrt. $w'$}\} = V(G_{w\to w'}).\]
The only thing to realise is that 
we have:
\[
2\left(\#E(G)_{w\to w'} + \sum_{f\in V(G_{w\to w'})}\frac{\ed(f)}{2}\right) =  \#V(w) - \# V(w'),
\] and this follows from the definition of $w_{G_{w\to w'}}=w'$.
\end{proof}

\begin{dfn}\label{dfn:avoidfill}
  If $f$ is a face of $w$, $w'$ a simplification of $w$ and $r$ a region of $w'$, we say that $f$ avoids $r$ if $f\cap r=\emptyset$ or if  the boundary of $r$ in each connected component of $f\cap r$ joins two consecutive vertices of $f$ (see figure \ref{fig:avoiding}). In the first case we say that $f$ avoid $r$ trivially.
  \begin{figure}[ht]
    \centering
    \begin{tikzpicture}
      \input{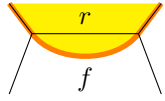}
    \end{tikzpicture}
    \caption{The local picture of a face $f$ (in white) of $w$ (in black)  non-trivially avoiding a region $r$ (in yellow) of $w'$ (in orange).}
    \label{fig:avoiding}
  \end{figure}
  If $f$ is an essential face of $w$ with respect to $w'$ and $r$ is a region of $w'$, we say that \emph{$f$ fills $r$}, if $f$ does not avoid $r$. If $F'$ is a set of region of $w'$ we say that \emph{$f$ fills (\resp avoids) $F'$} if it fills at least one region of $F'$ (\resp avoids all the regions of $F'$). We define:
\[
n(f,F') \eqdef \#\{r\in F' \,\textrm{such that $f$ fills $r$}\}.
\]
If $G'$ is a red graph of $w'$, we write $n(f,G')$ for $n(f,V(G'))$. 
\end{dfn}
With the same notations, and with $F$ a set of face of $w$, we have the following equality:
\begin{align}
\#F = \#\{\textrm{faces $f$ of $F$ avoiding $F'$}\} + \sum_{f'\in F'}\sum_{\substack{f \in F \\ \textrm{$f$ fills $f'$}}} \frac1{n(f,F')}. \label{eq:sumofcontrib}
\end{align}

\begin{lem}
  Let $w$ be a $\partial$-connected $\epsilon$-web and $w'$ a nice $\partial$-connected simplification of $w$. Let $F'$ be a collection of faces of $w'$, then for every face $f$ of $w$, we have:  $n(f, F') \leq 2$. 
\end{lem}
\begin{proof}
This is clear since $f\cap w'$ consists of at most one strand, so that it intersects at most 2 faces of $F'$.
\end{proof}

\begin{req}\label{rqe:hex2bign}
  Let $w$ be a $\epsilon$-web, $w'$ a nice $\partial$-connected simplification of $w$ and $f$ an essential face of $w$ with respect to $w'$. Suppose that $f$ has at least 6 sides of $w$. Suppose furthermore that it intersects two regions $r_1$ and $r_2$ of $w'$, then either it (non-trivially) avoids one of them, either it fills both of them. If $f$ avoids $r_2$ then at least two neighbours (in $G_{w\to w'})$ of $f$ fill $r_1$ (see picture~\ref{fig:fillhex}). 
If on the contrary $f$ has just one neighbour which fills $r_1$, then $f$ fills  $r_2$. Under this condition, for any collection $F'$ of regions of $w'$ with $\{r_1, r_2\} \subseteq F'$ we have: $n(f, F')=2$.
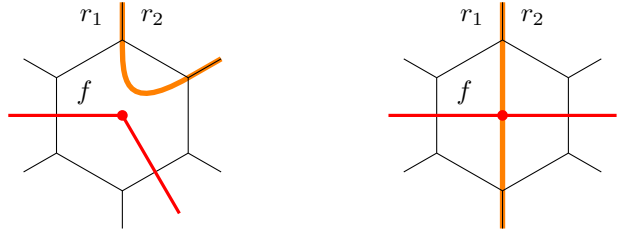
\begin{figure}[ht]
  \centering
  \begin{tikzpicture}
    \begin{scope}[decoration={markings, mark=at
     position 0.7 with {\arrow{>}}},postaction={decorate}]
\draw[orange, line width= 2pt] (90:1.5) -- (90: 1) .. controls (90:0.5) and (30:0.5 -0.5).. (30:1) -- (30:1.5);
\draw (90:1) -- (150:1) -- (210:1) -- (270:1) -- (330:1) -- (30:1)--cycle;
\draw (90:1) -- (90:1.5);
\draw (150:1) -- (150:1.5);
\draw (210:1) -- (210:1.5);
\draw (270:1) -- (270:1.5);
\draw (330:1) -- (330:1.5);
\draw (30:1) -- (30:1.5);
\fill[red] (0,0) circle (2pt);
\draw[very thick, red] (0,0) -- (180:1.5);
\draw[very thick, red] (0,0) -- (300:1.5);
\node at (0.4,1.3) {$r_2$};
\node at (-0.4,1.3) {$r_1$};
\node at (-0.5,0.3) {$f$};
\end{scope}

\begin{scope}[xshift = 5cm, decoration={markings, mark=at
     position 0.7 with {\arrow{>}}},postaction={decorate}]
\draw[orange, line width= 2pt] (90:1.5) -- (270: 1.5);
\draw (90:1) -- (150:1) -- (210:1) -- (270:1) -- (330:1) -- (30:1)--cycle;
\draw (90:1) -- (90:1.5);
\draw (150:1) -- (150:1.5);
\draw (210:1) -- (210:1.5);
\draw (270:1) -- (270:1.5);
\draw (330:1) -- (330:1.5);
\draw (30:1) -- (30:1.5);
\fill[red] (0,0) circle (2pt);
\draw[very thick, red] (0,0) -- (0:1.5);
\draw[very thick, red] (0,0) -- (180:1.5);
\node at (0.4,1.3) {$r_2$};
\node at (-0.4,1.3) {$r_1$};
\node at (-0.5,0.3) {$f$};
\end{scope}
  \end{tikzpicture}
  \caption{On the left $f$ avoids $r_2$, on the right it fills $r_1$ and $r_2$.}
  \label{fig:fillhex}
\end{figure}
\end{req}
\begin{dfn}\label{dfn:sigma}
We set $\displaystyle{\sigma(f', F\to F') \eqdef  \sum_{\substack{f \in F \\ \textrm{$f$ fills $f'$}}} \frac1{n(f,F')} }$. If $G$ is a red graph for $w$ and $G'$ a red graph for $w'$ we write $\sigma(f', G\to G' )$ for $\sigma(f', V(G)\to V(G'))$.
\end{dfn}

\subsection{Proof of lemma~\ref{lem:tech}}
\label{sec:proof-lemmatech}

In this section we use the point of view developed in section~\ref{sec:new-approach-to-red-graph} to prove the lemma~\ref{lem:tech}. We restate it with this new vocabulary:
\begin{lem}\label{lem:techNV}
Let $w$ be an a $\partial$-connected $\epsilon$-web which contains no digon and exactly one square. We suppose furthermore that this square touches the unbounded face. Let $G$ be a nice red graphs of $w$ and $G'$ a nice red graph of $w'=w_{G}$, then there exists $\widetilde{w}$ a nice simplification of $w$ such that:
\begin{enumerate}[A)]
\item\label{it:condA} the $\epsilon$-webs $(w_{G})_{G'}$ and $\widetilde{w}$ are isotopic,
\item\label{it:condB} the following equality holds:
\[
\#V( \widetilde{G}) \geq \# V(G) + \# V(G'),
\]
\end{enumerate}
where $\widetilde{G}$ denote the red graph $G_{w\to \widetilde{w}}$.
\end{lem}
\begin{proof}
  Because of the condition~\ref{it:condA}, we already know the isotopy class of the web $\widetilde{w}$. To describe it completely, we only need to specify how $\widetilde{w}$ is embedded. For each face $f'$ of $w'$ which is a vertex of $G'$, let us denote $N_{f'}$ a regular neighbourhood of $f'$. We consider $U$ the complementary of $\bigcup_{f'} N_{f'}$. Provided this is done in a coherent fashion, it's enough to specify how $\widetilde{w}$ looks like in $U$ and in $N_{f'}$ for each face $f'$ of $w'$.

If $f'$ is a face of $w'$ which is in $G'$, we consider two different cases: 
\begin{enumerate}
\item\label{it:face0} the face $f'$ corresponds to a vertex of $G'$ with exterior degree equal to 0,
\item\label{it:face2}the face $f'$ corresponds to a vertex of $G'$ with exterior degree equal to 2.
\end{enumerate}
These are the only cases to consider since $G'$ is nice.

Let us denote by $w''$ the $\epsilon$-web $(w')_{G'}$. We want $\widetilde{w}$ and $w''$ to be isotopic. So let us look at $w''\cap U$ and at $w''\cap N_{f'}$ in the two cases.

Around $U$, the $\epsilon$-web $w'$ does not ``see'' the red graph $G'$, so that $U\cap w''= U\cap w'_{G'} = U\cap w'$. 

If the face $f'$ has exterior degree equal to 0 (case~\ref{it:face0}), then we have: $N_{f'}\cap w'' =  U\cap w'_{G'}= \emptyset$. \marginpar{example ?}

If the face $f'$ has exterior degree equal to 2 (case~\ref{it:face2}), then we have: $N_{f'}\cap w'' =  U\cap w'_{G'}$ is a single strands cutting $N_{f'}$ into two parts. \marginpar{example}

We embed $\widetilde{w}$ such that $U\cap \widetilde{w}$ and  $U\cap w''$ are equal and for each face $f'$ corresponding to a vertex of $G'$,
$N_{f'}\cap \widetilde{w}$ and $N_{f'}\cap w''$ are isotopic (relatively to the boundary).

We claim that if $f'$ is a vertex with external degree equal to 0 then:
\begin{align}\begin{cases}\sigma (f', \widetilde{G}\to G') \geq \sigma(f', G \to G') +\frac12 & \textrm {if $S \subseteq N_{f'}$,} \\
\sigma (f', \widetilde{G}\to G') \geq \sigma(f', G \to G') +1 & \textrm {if $S\nsubseteq N_{f'}$,}
\end{cases}\label{eq:sigmaface0}
\end{align}  
where $S$ is the square of $w$.

The restriction\footnote{We only consider the vertices of $G$ which fill $f'$.} $G_{f'}$ of $G'$ to $f'$ is a graph which satisfies the following conditions:
\begin{itemize}
\item it is bi-coloured (because the vertices of $G$ are essential faces of $w$ with respect to $w'$),
\item it is naturally embedded in a disk because $N_{f'}$ is diffeomorphic to a disk,
\item the degree of the vertices inside the disks have degree at least three (because the only possible square of $w$ touches the border) and the vertices on the border (these are the one which intersect an other region of $w'$) have degree at least 1.
\end{itemize}

The regions of $G_{f'}$ and the vertices of one of the two colours of $G_{f'}$ become vertices of $\widetilde{G}$ (see example depicted on figure~\ref{fig:exampleed0}). To proves the inequality (\ref{eq:sigmaface0}), one should carefully count regions of $G_{f'}$. There are two different cases in  (\ref{eq:sigmaface0}) because the remark~\ref{rqe:hex2bign} do not apply to the square.
Hence we can apply the lemma~\ref{lem:techtech0} which proves (\ref{eq:sigmaface0}).

\begin{figure}[ht]
  \centering
  \begin{tikzpicture}[scale= 1.4]
    \input{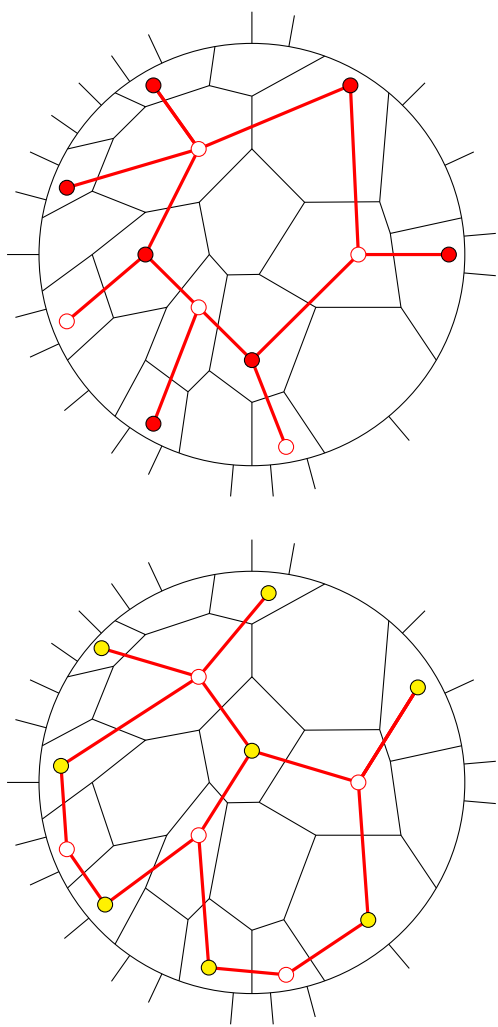}
  \end{tikzpicture}
  \caption{Example of the procedure to define $\widetilde{G}$ when the exterior degree of $f'$ is equal to 0.}
  \label{fig:exampleed0}
\end{figure}

The only thing remaining to specify is how $\widetilde{w}$ looks in $N_{f'}$ where $f'$ is a vertex of $G'$ of exterior degree equal to 2. Note that we need to embed $\widetilde{w}$ so that it is a nice simplification of $w$. We claim that it is always possible to find such an embedding so that the inequality (\ref{eq:sigmaface0}) is satisfied. In this case the graph $G_{f'}$ is in the same situation as before, but is important to notice that the faces of $G_{f'}$ have at least 6 sides (this is a consequence of proposition~\ref{prop:largecycle}).

The vertices of $\widetilde{G}$ are the regions of $G_{f'}$ and the vertices of $G_{f'}$ of one of the two colours on one side of the strand and the vertices of  $G_{f'}$ of the other colour on the other side of the strand (see figure~\ref{fig:exampleed0} for an example). Hence in order to show that the inequality (\ref{eq:sigmaface0}) holds, one should carefully count the regions and the vertices of $G_{f'}$, this is done by lemma~\ref{sec:case-with-exterior2}.

\begin{figure}[ht]
  \centering
  \begin{tikzpicture}[scale =1.4]
    \input{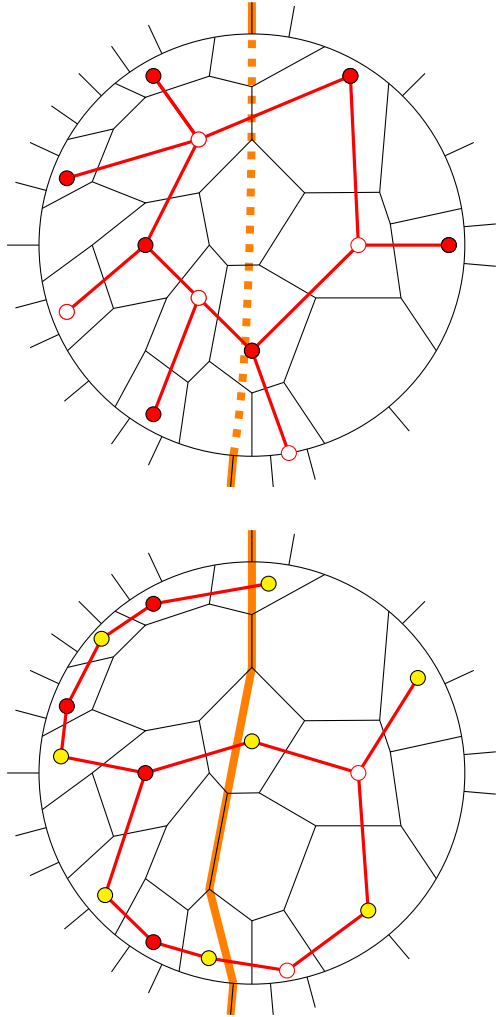}
  \end{tikzpicture}
  \caption{Example of the procedure to define $\widetilde{G}$ when the exterior degree of $f'$ is equal to 2.}
  \label{fig:exampleed2}
\end{figure}

So now we have a simplification $\widetilde{w}$ of $w$, such that the graph $\widetilde{G}=G_{w\to \widetilde{w}}$ satisfies (\ref{prop:largecycle}) for each region $f'$ of $G'$. The square $S$ of $w$ is in at most one $N_{f'}$ so that if we sum  (\ref{prop:largecycle}) for all the vertices of $f'$, we obtain:
\[
\sum_{f'\in F'} \sigma(f', \widetilde{G}\to G') \geq \sum_{f'\in F'} \sigma(f', G \to G') + \#V(G') - \frac12,
\]
and using (\ref{eq:sumofcontrib}), we have:
\[
\#V(\widetilde{G}) \geq \#V(G') + \#V(G) - \frac12,
\]
but $\#V(\widetilde{G})$ being an integer we have $V(\widetilde{G}) \geq \#V(G') + \#V(G)$
\end{proof}

\subsection{Proof of combinatorial lemmas}
\label{sec:proof-techn-lemm}

This proof is dedicated to the two technical lemmas used in the last section. We first introduce the ad-hoc objects and then state and prove the lemmas.

\begin{dfn}
  A \emph{$D$-graph} is a graph $G$ embedded in the disk $D^2$. The set of vertices $V(G)$ is partitioned in two sets: $V^\partial(G)$ contains the vertices lying on $\partial D^2$, while $V^\textrm{in}$ contains the others. The set $F(G)$ of connected components of $D^2\setminus G$ is partitioned into two sets $F^\textrm{in}$ contains the connected component included in $\mathring{D^2}$, while $F^\partial$ contains the others.

  A $D$-graph is said to be \emph{non-elliptic} if: 
  \begin{itemize}
  \item every vertex $v$ of $V^\textrm{in}$ has degree greater or equal to 3,
  \item every vertex $v$ of $V^\partial$ has degree greater or equal than 1,
  \item the faces of $F^\mathrm{in}$ are of size at least 6.
  \end{itemize}

A \emph{coloured $D$-graph} is a $D$-graph $G$ together with: 
\begin{itemize}
\item a vertex-2-colouring (by $\green$ and $\blue$) of the vertices of $G$ (this implies that $G$ is bipartite),
\item a subdivision of $\D^2$ into two intervals (we allow one interval to be the empty set and the other one to be the full circle, in this case we say that $G$ is \emph{circled-coloured}): a $\green$ one and a $\blue$ one (denoted by $I_\blue$ and $I_\green$) when they are real intervals we defines $x$ and $y$ to be the two intersection points of $I_\green$ and $I_\blue$ with the convention that when one scans $\partial D^2$ clockwise,  one see $x$, then $I_\green$, then $y$ and finally $I_\blue$.
\end{itemize}

The vertices of $V^\partial$ are supposed neither on $x$ nor on $y$. The colour of a vertex is not supposed to fit the colour of the interval it lies on. We set $V_{\green}$ (\resp $V_\blue$) the set of $\green$, (\resp $\blue$) vertices, and  
 $V_{\green}^\partial$, $V_{\green}^\mathrm{in}$, $V_{\blue}^\partial$ and $V_{\blue}^\mathrm{in}$  in the obvious way.

If $G$ is a coloured $D$-graph, and $v$ is a vertex of $V^\partial$ of we set:
\[
n(v) = \begin{cases} 2 & \textrm{if $v$ has degree 1 and the colour of $v$ fits the colour of the interval,} \\ 1 & \textrm{else.}
\end{cases}
\]
If $v$ is a vertex of $V^\mathrm{in}$, we set $n(v)=1$. Note that this definition of $n$ is a translation of the $n$ of the previous section (see remark~\ref{rqe:hex2bign}).
\end{dfn}

\subsubsection{Case with exterior degree equal to 0}
\label{sec:case-with-exterior0}

\begin{lem} \label{lem:techtech0}
  Let $G$ be a non-elliptic circled-coloured $D$-graph (with the circle coloured by  a colour $c$), then:
\[\#F \geq 1 + \sum_{v\in V_c} \frac{1}{n(v)}.
\]
\end{lem}
\begin{proof}By symmetry, we may suppose that $c=\green$.
To show this we consider the graph $H$ obtain by gluing to copies of $G$ along the boundary of $D^2$ this is naturally embedded in the sphere. We write the Euler characteristic:
\begin{align}
\#F(H) - \#E(H) + \#V(H) = 1 + \#C(H), \label{eq:eulerchar}
\end{align}
where $C(H)$ is the set of connected components of $H$. We have the following equalities:
\begin{align*}
  \#F(H) &= 2\#F^{\textrm{in}}(G) + \#F^{\partial}(G),\\
  \#F^{\partial}(G) &= {\#V^{\partial}(G)} + 1 -\#C(H), \\
  \#E(H) &= 2 \#E(G) = \sum_{v\in V(G)} \deg(v) = 2\sum_{v\in V_{\green}(G)} \deg(v),\\
    \#V(H) &= 2\#V^{\textrm{in}}(G) + \#V^{\partial}(G)
\end{align*}
So that we can rewrite (\ref{eq:eulerchar}):
\begin{align}
2\#F^{\textrm{in}}(G) + 2\#F^{\partial}(G) +2\#V^{\textrm{in}}(G)  = 2 +  2 \#E(G). \label{eq:EX2}
\end{align}
Now we use the what we know about degrees of the vertices:
\begin{align*}
  \#E(G)&\geq \frac{3}2 \# V^{\textrm{in}}(G)+ \frac12 \#V^{\partial,1}(G) +1 \#V^{\partial,>1}(G), \\
  \#E(G)&\geq  3 \# V_{\green}^{\textrm{in}}(G)+ \#V_{\green}^{\partial,1}(G) +2 \#V_{\green}^{\partial,>1}(G).
\end{align*}
Where $V^{\partial,1}$ (\resp $V^{\partial,>1}$) denotes the subset of $V^\partial$ with degree equal to $1$ (\resp with degree strictly bigger than 1). 
If we sum $\frac23$ of the first inequality and $\frac13$ of the second one, and inject this in~(\ref{eq:EX2}) we obtain:
\begin{align*}
   \#F(G) + \#V^{\textrm{in}}(G) &\geq  1+  \#E(G) \\
\#F(G) + \#V^{\textrm{in}}(G)    &\geq V^{\textrm{in}}(G)+ \frac13 \#V^{\partial,1}(G) +\frac23 \#V^{\partial,>1}(G) \\
   & \quad + \# V_{\green}^{\textrm{in}}(G) + \frac13\#V_{\green}^{\partial,1}(G) + \frac{2}3 \#V_{\green}^{\partial,>1}(G) \\
 \#F(G) &\geq \# V_{\green}^{\textrm{in}}(G)+ \frac23 \#V_{\green}^{\partial,1}(G) +\frac43 \#V_{\green}^{\partial,>1}(G) \\
&\quad + \frac13\#V_{\blue}^{\partial,1}(G) + \frac{2}3 \#V_{\blue}^{\partial,>1}(G) \\
&\geq  \# V_{\green}^{\textrm{in}}(G)+ \frac12 \#V_{\green}^{\partial,1}(G) + \#V_{\green}^{\partial,>1}(G) \\
& \geq \sum_{v\in V_\green} \frac{1}{n(v)}.
\end{align*}
\end{proof}

\subsubsection{Case with exterior degree equal to 2}
\label{sec:case-with-exterior2}

\begin{lem}\label{lem:bg1UYH}
  If $G$ is a non-elliptic $D$-graph, then all the faces of $F$ are diffeomorphic to disks, and if it is non-empty, then at least one of the following situations happens:
  \begin{enumerate}[(1)]
  \item\label{it:bg1} the set $V^{\partial,>1}$ is non empty,
  \item\label{it:cap} there exists, two $\cap$'s (see figure~\ref{fig:UDgraphUHY}) (if $G$ consists of only one edge, the two $\cap$'s are actually the same one counted two times because it can be seen as a $\cap$ on its two sides),
  \item\label{it:3YH} there exists three $\lambda$'s or $H$'s (see figure~\ref{fig:UDgraphUHY}).
    \begin{figure}[ht]
      \centering
      \begin{tikzpicture}[scale=1.1]
        \begin{scope}[decoration={markings, mark=at
     position 0.7 with {\arrow{>}}},postaction={decorate}]
\coordinate (V1) at (-0.7,-0);
\coordinate (V2) at (0.7,-0);
\draw[dotted] (-1,0) arc (180:0:1 and 1);
\draw[gray, very thick] (170:1) .. controls +(0,0) and +(-0.3,0.05) .. (V1) ..controls +(+0.3,-0.05)  and +(-0.3,-0.05) .. (V2) .. controls +(+0.3,0.05) and +(0,0).. (10:1);
\fill (V1) circle (2pt);
\fill (V2) circle (2pt);
\draw (V1).. controls +(0.2,0.8) and +(-0.2, 0.8).. (V2);
\end{scope}
\begin{scope}[xshift = 4cm, decoration={markings, mark=at
     position 0.7 with {\arrow{>}}},postaction={decorate}]
\coordinate (V1) at (-0.7,-0);
\coordinate (V2) at (0.7,-0);
\coordinate (V3) at (0, 0.5);
\draw[dotted] (-1,0) arc (180:0:1 and 1);
\draw[gray, very thick] (170:1) .. controls +(0,0) and +(-0.3,0.05) .. (V1) ..controls +(+0.3,-0.05)  and +(-0.3,-0.05) .. (V2) .. controls +(+0.3,0.05) and +(0,0).. (10:1);
\fill (V1) circle (2pt);
\fill (V2) circle (2pt);
\fill (V3) circle (2pt);
\draw (V1) -- (V3);
\draw (V2) -- (V3);
\draw (V3) -- (100:1);
\draw (V3) -- (90:1);
\draw (V3) -- (80:1);
\end{scope}
\begin{scope}[xshift = 8cm, decoration={markings, mark=at
     position 0.7 with {\arrow{>}}},postaction={decorate}]
\coordinate (V1) at (-0.7,-0);
\coordinate (V2) at (0.7,-0);
\coordinate (V3) at (-0.5, 0.4);
\coordinate (V4) at (0.5, 0.4);
\draw[dotted] (-1,0) arc (180:0:1 and 1);
\draw[gray, very thick] (170:1) .. controls +(0,0) and +(-0.3,0.05) .. (V1) ..controls +(+0.3,-0.05)  and +(-0.3,-0.05) .. (V2) .. controls +(+0.3,0.05) and +(0,0).. (10:1);
\fill (V1) circle (2pt);
\fill (V2) circle (2pt);
\fill (V3) circle (2pt);
\fill (V4) circle (2pt);
\draw (V1) -- (V3);
\draw (V2) -- (V4);
\draw (V4) -- (V3);
\draw (V3) -- (130:1);
\draw (V3) -- (135:1);
\draw (V3) -- (140:1);
\draw (V4) -- (50:1);
\draw (V4) -- (45:1);
\draw (V4) -- (40:1);
\end{scope}
      \end{tikzpicture}
      \caption{From left to right: a $\cap$, a $\lambda$ and an $H$. The circle $\partial D^2$ is thick and grey, the $D$-graph is thin and black. Note that the vertices inside $D^2$ may have degree bigger than 3.}
      \label{fig:UDgraphUHY}
    \end{figure}
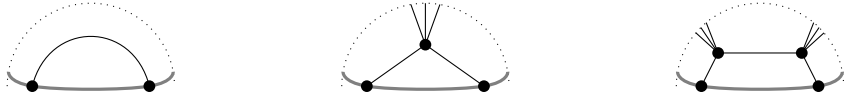
  \end{enumerate}
\end{lem}
\begin{proof}
  This is the same Euler characteristic-argument that we used in lemma~\ref{lem:UHYinNE}.
\end{proof}

\begin{dfn}
  A cut in a (not circled-) coloured $D$-graph is a simple oriented path $\gamma:[0,1] \to D$ such that:
  \begin{itemize}
  \item we have $\gamma(0)=x$ and $\gamma(1)=y$, therefore $I_\green$ is on the left and $I_\blue$ is on the right\footnote{We use the convention that the left and right side are determined when one scans $\gamma$ from $x$ to $y$.}. (see figure \ref{fig:cut}),
  \item for every face $f$ of $G$, $f\cap \gamma$ is connected,
  \item the path $\gamma$ crosses $G$ either transversely at edges joining a $\green$ vertex on left and a $\blue$ vertex on the right, or at vertices of $V^\partial$ whose colours do not fit with the intervals they lie on.
  \end{itemize}
  \begin{figure}[ht]
    \centering
    \begin{tikzpicture}[scale=1.1]
      \input{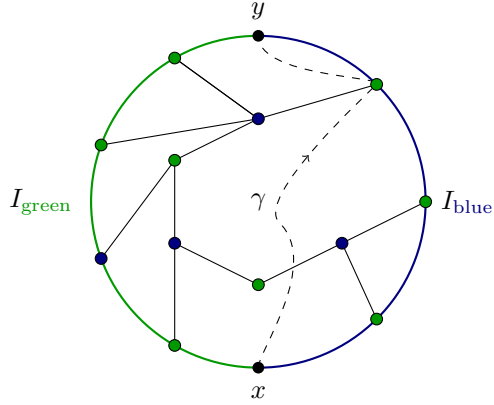}
    \end{tikzpicture}
    \caption{A cut in a coloured $D$-graph (note that $G$ is elliptic).}
    \label{fig:cut}
  \end{figure}
If $\gamma$ is a cut we denote by $\tensor[^{l(\gamma)}]{V}{}(G)$ and $\tensor[^{r(\gamma)}]{V}{}(G)$ the vertices located on the left (\resp on the right) of $\gamma$. (The vertices located on $\gamma$ are meant to be both on the left and on the right).
\end{dfn}

\begin{lem}\label{lem:techtech2}
  Let $G$ be is a non-elliptic (not circled-) coloured $D$-graph, then there exists a cut $\gamma$ such that: 
\[\#F(G) \geq 1+ \sum_{v\in \tensor[^{l(\gamma)}]{V}{_\green}} \frac1{n(v)} + \sum_{v\in \tensor[^{r(\gamma)}]{V}{_\blue}}\frac1{n(v)}. \]
\end{lem}

\begin{proof}
  The proof is done by induction on $s(G)\eqdef 3\# E(G) +4\# V^{\partial,>1}(G)$. If this quantity is equal to zero then the $D$-graph is empty, then we choose $\gamma$ to be any simple arc joining $x$ to $y$, and the lemma says $1\geq 1$ which is true.
We set:
\[
C(G,\gamma)\eqdef\sum_{v\in \tensor[^{l(\gamma)}]{V}{_\green}} \frac1{n(v)} + \sum_{v\in \tensor[^{r(\gamma)}]{V}{_\blue}}\frac1{n(v)}.
\]
 It is enough to check the situations (\ref{it:bg1}), (\ref{it:cap}) and (\ref{it:3YH})  described in lemma \ref{lem:bg1UYH}.

\paragraph*{Situation (\ref{it:bg1})}
Let us denote by $v$ a vertex of $V^{\partial, >1}$. There are two cases: the colour of $v$ fits with the colours of the intervals it lies on or not. 

If the colours fit, say both are $\green$, we consider $G'$ the same coloured $D$-graph as $G$ but with $v$ split into $v_1$, $v_2$ \dots $v_{\deg(v)}$ all in $V^{\partial,1}(G')$ (see figure~\ref{fig:GGpsit1}). We have $s(G')= S(G)-4< s(G)$ and $G'$ non-elliptic, therefore we can apply the induction hypothesis. We can find a cut $\gamma'$ with $\#F(G')\geq 1+ C(G',\gamma')$. Note that $\gamma'$ does not cross any $v'$, so that we can lift $\gamma'$ in the $D$-graph $G$. This gives us $\gamma$. We have:
\begin{align*}
C(G,\gamma)&= C(G',\gamma') + \frac1{n(v)} - \sum_{k=1}^{\deg(v)}\frac1{n(v_k)} \\
&=  C(G',\gamma') + 1 - \frac{\deg(v)}2 \\
&\geq C(G',\gamma').
\end{align*}
On the other hand $\#F(G)= \#F(G')$ so that we have $\#F(G)\geq 1+ C(G,\gamma)$.

\begin{figure}[ht]
  \centering
  \begin{tikzpicture}[scale=1.1]
    \input{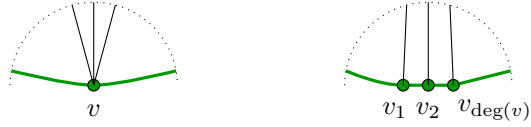}
  \end{tikzpicture}
  \caption{Local picture of $G$ (on the left) and $G'$ (on the right) around, when $v$ is $\green$ and lies on $I_\green$.}
  \label{fig:GGpsit1}
\end{figure}

If the colours do not fit (say $v$ is $\blue$), we construct $G'$ a coloured $D$-graph which is similar to $G$ every where but next to $v$. The vertex $v$ is pushed in $D^2$ (we denote it by $v'$) and we add a new vertex $v''$ on $\partial D^2$ and an edge $e$ joining $v'$ and $v''$. The coloured $D$-graph $G'$ is non-elliptic and $s(G')= s(G)-4+3<s(G)$ so that we can apply the induction hypothesis and find a cut $\gamma'$ with $\#F(G')\geq 1+ C(G',\gamma')$. 

\begin{figure}[ht]
  \centering
  \begin{tikzpicture}[scale=1.1]
    \input{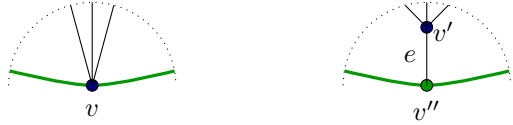}
  \end{tikzpicture}
  \caption{Local picture of $G$ (on the left) and $G'$ (on the right) around $v$, when $v$ is $\blue$ and lies on $I_\green$.}
  \label{fig:GGpsit1b}
\end{figure}

If $\gamma'$ does not cross $e$ we can lift $\gamma'$ in $G$ (this gives us $\gamma$). We have 
\begin{align*}C(G,\gamma)= C(G',\gamma')+ \frac{1}{n(v)} - \frac{1}{n(v')} =C(G',\gamma')+ 1- 1 = C(G',\gamma').
 \end{align*}
On the other hand, we have $F(G)= F(G')$, so that $\#F(G)\geq 1+ C(G,\gamma)$. 

Consider now the case where $\gamma'$ crosses $e$. Then we consider the cut $\gamma$ of $G$ which is the same as $\gamma$ far from $v$, and which around $v$ crosses $G$ in $v$ (see figure~\ref{fig:GGpsitc}). We have:
\begin{align*}
C(G,\gamma)&= C(G',\gamma')+ \frac{1}{n(v)} - \frac{1}{n(v')} - \frac{1}{n(v'')}\\
C(G,\gamma)&= C(G',\gamma') +1-1-\frac12\\
C(G,\gamma)&\geq  C(G',\gamma').
\end{align*}
But $\#F(G)= \#F(G')$, so that we have $\#F(G)\geq 1+ C(G,\gamma)$.
\begin{figure}[ht]
  \centering
  \begin{tikzpicture}[scale=1.1]
    \input{\imagesfolder/th_GGpsit1c}
  \end{tikzpicture}
  \caption{How to transform $\gamma'$ into $\gamma$.}
  \label{fig:GGpsitc}
\end{figure}

\paragraph*{Situation (\ref{it:cap})}
We now suppose that $G$ contains two $\cap$'s. Let us denote by $v_g$ (\resp $v_b$) the $\green$ (\resp $\blue$) vertex of the $\cap$ and by $e$ the edge of the cap. There are different possible configurations depending where $x$ and $y$ lies. As there are at least two caps, we may suppose $y$ is far from the $\cap$.

There are 3 different configurations (see figure~\ref{fig:3configsit2}):
\begin{itemize}
\item the point $x$ is far from the $\cap$,
\item the point $x$ is in the $\cap$ and $v_g\in I_\green$ and $v_b \in I_\blue$,
\item the point $x$ is in the $\cap$ and $v_g\in I_\blue$ and $v_b \in I_\green$,
\end{itemize}

\begin{figure}[ht]
  \centering
 \begin{tikzpicture}[scale=1.1]
    \input{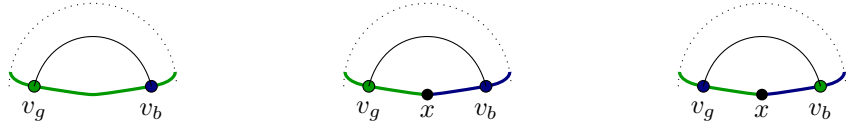}
  \end{tikzpicture} 
  \caption{The three possible configurations.}
  \label{fig:3configsit2}
\end{figure}

 We consider $G'$ the coloured $D$-graph similar to $G$ except that the $\cap$ is removed. The coloured $D$-graph $G'$ is non-elliptic and $s(G')= s(G)- 3 <s(G)$ so that we can apply the induction hypothesis and find a cut $\gamma'$ with $\#F(G')\geq 1+ C(G',\gamma')$.

Let us suppose first that $x$ is far from the $\cap$, then $v_b$ and $v_g$ both lie either on $I_\green$ or on $I_\blue$. By symmetry we may consider that they both lie on $I_\green$. We can lift $\gamma'$ in $G$ (this gives $\gamma$) so that it does not meet the $\cap$. We have:
\begin{align*}
C(G,\gamma)&= C(G',\gamma')+ \frac{1}{n(v_g)} \\
&= C(G',\gamma')+ \frac12.
\end{align*}
But $\#F(G) = \#F(G')+1$, hence   $\#F(G)\geq 1+ C(G,\gamma)$.

Suppose now that the point $x$ is in the $\cap$ and $v_g\in I_\green$ and $v_b \in I_\blue$. We can lift $\gamma'$ in $G$ so that it crosses $e$ (see figure~\ref{fig:ggpsit2b}). 
\begin{figure}[ht]
  \centering
 \begin{tikzpicture}[scale=1.1]
    \input{\imagesfolder/th_ggpsit2b}
  \end{tikzpicture} 
  \caption{How to transform $\gamma'$ into $\gamma$.}
  \label{fig:ggpsit2b}
\end{figure}
We have:
\begin{align*}
C(G,\gamma)&= C(G',\gamma')+ \frac{1}{n(v_g)} +\frac{1}{n(v_b)} \\
&= C(G',\gamma')+ \frac12+\frac12  \\
&=  C(G',\gamma')+1.
\end{align*}
But $\#F(G) = \#F(G')+1$, hence   $\#F(G)\geq 1+ C(G,\gamma)$.

Suppose now that the point $x$ is in the $\cap$ and $v_g\in I_\blue$ and $v_b \in I_\green$. We can lift  $\gamma'$ in $G$ so that it crosses\footnote{We could have chosen to cross $v_b$.} $v_g$ (see figure~\ref{fig:ggpsit2c}). 
\begin{figure}[ht]
  \centering
 \begin{tikzpicture}[scale=1.1]
    \begin{scope}[xshift = 0cm,decoration={markings, mark=at
     position 0.7 with {\arrow{>}}},postaction={decorate}]
\coordinate (V1) at (-0.7,-0);
\coordinate (V2) at (0.7,-0);
\coordinate (V) at (0, -0.1);
\draw[dotted] (-1,0) arc (180:0:1 and 1);
\draw[darkgreen, very thick] (170:1) .. controls +(0,0) and +(-0.3,0.05) .. (V1) ..controls +(+0.3,-0.05)  and +(-0.1,0) .. (V);
\draw[darkblue, very thick] (V) .. controls +(0.1,0) and +(-0.3,-0.05) .. (V2) .. controls +(+0.3,0.05) and +(0,0).. (10:1);
\fill (V) circle (2pt);
\draw[dashed] (V) -- (0,1);
\node at (0, -0.3) {$x$};
\node at (0, 1.3) {$\gamma'$};
\end{scope}
\node at (3,0.5) {$\leadsto$};
\begin{scope}[xshift = 6cm,decoration={markings, mark=at
     position 0.7 with {\arrow{>}}},postaction={decorate}]
\coordinate (V1) at (-0.7,-0);
\coordinate (V2) at (0.7,-0);
\coordinate (V) at (0, -0.1);
\draw[dotted] (-1,0) arc (180:0:1 and 1);
\draw[darkgreen, very thick] (170:1) .. controls +(0,0) and +(-0.3,0.05) .. (V1) ..controls +(+0.3,-0.05)  and +(-0.1,0) .. (V);
\draw[darkblue, very thick] (V) .. controls +(0.1,0) and +(-0.3,-0.05) .. (V2) .. controls +(+0.3,0.05) and +(0,0).. (10:1);
\filldraw[fill= darkblue] (V1) circle (2pt);
\filldraw[fill=darkgreen] (V2) circle (2pt);
\fill (V) circle (2pt);
\draw[dashed] (V) .. controls +(0,0.3) and +(-0.1,0.1) .. (V1).. controls +(0.1,0.1) and +(-0.2,-0.2) .. (-0.6, 0.6) .. controls +(0.2,0.2) and +(-0.4,0) .. (0,1);
\draw (V1).. controls +(0.2,0.8) and +(-0.2, 0.8).. (V2);
\node at (-0.7, -0.3) {$v_g$};
\node at (+0.7, -0.3) {$v_b$};
\node at (0, -0.3) {$x$};
\node at (0, 1.3) {$\gamma$};
\end{scope}
  \end{tikzpicture} 
  \caption{How to transform $\gamma'$ into $\gamma$.}
  \label{fig:ggpsit2c}
\end{figure}
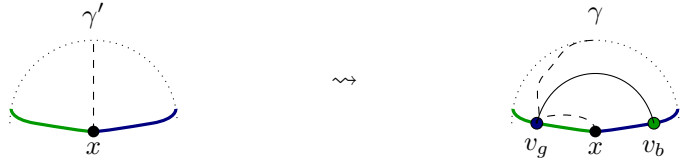
We have:
 \begin{align*}
C(G,\gamma)&= C(G',\gamma')+ \frac{1}{n(v_g)} \\
&=  C(G',\gamma')+1.
\end{align*}
But $\#F(G) = \#F(G') +1 $, hence   $\#F(G)\geq 1+ C(G,\gamma)$.

\paragraph{Situation (\ref{it:3YH})}
\label{sec:situation-3}

We suppose now that there are three $\lambda$'s or $H$'s. One can suppose that a $\lambda$ or an $H$ is far from $x$ and from $y$. 

Consider first that there is a $\lambda$ far from $x$ and $y$. Let us denote by $v_1$ and $v_2$ the two vertices of the $\lambda$ which belongs to $V^\partial(G)$, by $v$ the vertex of the $\lambda$ which is  in $V^\textrm{in}(G)$ and by $e_1$ (\resp $e_2$) the edge joining $v$ to $v_1$ (\resp $v_2$). We consider $G'$ the $D$-graph where the $\lambda$ is replaced by a single strand: the edges $e_1$ and $e_2$ and the vertices $v_1$ and $v_2$ are suppressed. The vertex $v$ is moved to $\partial D^2$ (and renamed $v'$). This is depicted on figure~\ref{fig:2confY}. The coloured $D$-graph $G'$ is non-elliptic and $s(G')<s(G)$ so that we can apply the induction hypothesis and find a cut $\gamma'$ with $\#F(G')\geq 1+ C(G',\gamma')$. 

The vertices $v_1$ and $v_2$ have the same colour, by symmetry we may suppose that they are both $\green$. It implies that $v$ and $v'$ are both $\blue$.

There are two different configurations:
\begin{itemize}
\item the vertices $v_1$ and $v_2$ lie on $I_\green$,
\item the vertices $v_1$ and $v_2$ lie on $I_\blue$.
\end{itemize}

\begin{figure}[ht]
  \centering
 \begin{tikzpicture}[scale=1]
    \input{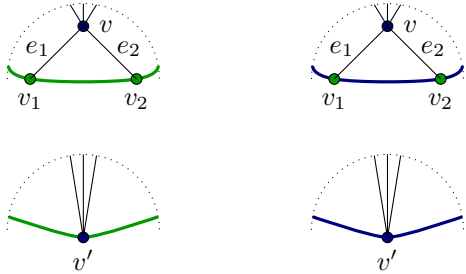}
  \end{tikzpicture} 
  \caption{On the center, the two possible configurations for a $\lambda$, on the sides, the $D$-graphs $G'$ obtained from $G$.}
  \label{fig:2confY}
\end{figure}

Let us first suppose that the vertices $v_1$ and $v_2$ lie on $I_\green$. If the cut $\gamma'$ does not cross $v'$ then we can canonically lift it in $G$. This gives us $\gamma$. We have:
\begin{align*}
C(G,\gamma)&= C(G',\gamma')+ \frac{1}{n(v_1)} +  \frac{1}{n(v_2)} \\
&=  C(G',\gamma')+\frac12 + \frac12.
\end{align*}
But $\#F(G) =\# F(G')+1$, hence $\#F(G)\geq 1+ C(G,\gamma)$.

If the cut $\gamma'$ crosses $v'$, we lift $\gamma'$ in $G$ so that it crosses $e_1$ and $e_2$ (see figure~\ref{fig:ggpsit2c}). 
\begin{figure}[ht]
  \centering
 \begin{tikzpicture}[scale=1.1]
    \begin{scope}[xshift = 0cm, decoration={markings, mark=at
     position 0.7 with {\arrow{>}}},postaction={decorate}]
\coordinate (V1) at (-0.7,-0);
\coordinate (V2) at (0.7,-0);
\coordinate (V3) at (0, 0.7);
\draw[dotted] (-1,0) arc (180:0:1 and 1);
\draw[darkgreen, very thick] (170:1) .. controls +(0,0) and +(-0.3,0.05) .. (V1) ..controls +(+0.3,-0.05)  and +(-0.3,-0.05) .. (V2) .. controls +(+0.3,0.05) and +(0,0).. (10:1);
\draw[dashed] (150:1) .. controls +(0.3,-0.2) and +(-0.3,-0.2) .. (30:1);
\filldraw[fill= darkgreen] (V1) circle (2pt);
\filldraw[fill= darkgreen] (V2) circle (2pt);
\filldraw[fill= darkblue] (V3) circle (2pt);
\draw (V1) -- (V3);
\draw (V2) -- (V3);
\draw (V3) -- (100:1);
\draw (V3) -- (90:1);
\draw (V3) -- (80:1);
\node at (-1.2, 0.6) { $\gamma$};
\end{scope}
\node at (-3,0.5) {$\leadsto$};
\begin{scope}[xshift = -6cm, decoration={markings, mark=at
     position 0.7 with {\arrow{>}}},postaction={decorate}]
\coordinate (V3) at (0, -0.10);
\draw[dotted] (-1,0) arc (180:0:1 and 1);
\draw[darkgreen, very thick] (170:1) .. controls +(0,0) and +(-0.2,0) .. (V3) .. controls +(+0.2,0) and +(0,0).. (10:1);
\filldraw[fill= darkblue] (V3) circle (2pt);
\draw[dashed] (150:1) .. controls (0,0) and +(-0.1, 0) ..(V3) .. controls +(0.1,0) and (0,0) .. (30:1);
\draw (V3) -- (100:1);
\draw (V3) -- (90:1);
\draw (V3) -- (80:1);
\node at (-1.2, 0.6) { $\gamma'$};
\end{scope}
  \end{tikzpicture} 
  \caption{How to transform $\gamma'$ into $\gamma$.}
  \label{fig:ggpYsit2}
\end{figure}
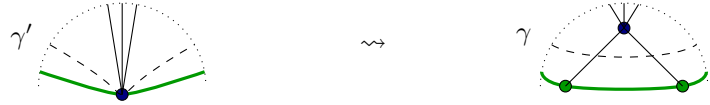
In this case we have:
\begin{align*}
C(G,\gamma)&= C(G',\gamma')+ + \frac{1}{n(v)} - \frac1{n(v')} +  \frac{1}{n(v_1)} +  \frac{1}{n(v_2)} \\
&=  C(G',\gamma')+1 -1+ \frac12 + \frac12.
\end{align*}
Hence, $\#F(G)\geq 1+ C(G,\gamma)$.

Now suppose that the vertices $v_1$ and $v_2$ lie on $I_\blue$, this implies that $\gamma'$ does not meet $v'$, so that we can lift $\gamma'$ canonically in $G$, this gives us $\gamma$, we have:
\begin{align*}
C(G,\gamma)&= C(G',\gamma')+ \frac{1}{n(v)} -  \frac{1}{n(v')} \\
&=  C(G',\gamma')+1-1.
\end{align*}
Hence $\#F(G)\geq 1+ C(G,\gamma)$.

We finally consider a $H$ far from $x$ and $y$. We take notation of the figure~\ref{fig:GGsitH} to denote vertices and edges of the $H$, we consider $G'$ the $D$-graph where the the $H$ is simplified (see figure~\ref{fig:GGsitH} for details and notation). The coloured $D$-graph $G'$ is non-elliptic and $s(G')= s(G) -3\times 3 + 2\times 4<s(G)$ so that we can apply the induction hypothesis and find a cut $\gamma'$ with $\#F(G')\geq 1+ C(G',\gamma')$. 

\begin{figure}[ht]
  \centering
 \begin{tikzpicture}[scale=1.4]
    \input{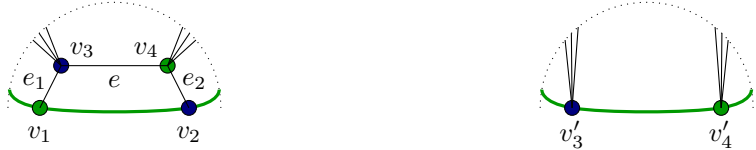}
  \end{tikzpicture} 
  \caption{How to transform $G$ into $G'$.}
  \label{fig:GGsitH}
\end{figure}

Up to symmetry there is only one configuration, therefore we may suppose that $v_1$ is $\green$ and lies on $I_\green$.  This implies that $v_2$ and $v_3$ are $\blue$  and that  $v_4$ is $\green$. Because of the colour condition, the cut $\gamma'$ does not cross $v'_4$ and may cross $v'_3$. If it does not cross $v'_3$, one can canonically lift $\gamma'$ in $G'$ and we have:
\begin{align*}
C(G,\gamma)&= C(G',\gamma')+ \frac{1}{n(v_1)} +  \frac{1}{n(v_4)} - \frac{1}{n(v'_4)} \\
&\geq  C(G',\gamma')+\frac12 +1 -1. \\
&\geq C(G',\gamma') +\frac12.
\end{align*}
But $\#F(G) =\# F(G')+1$, hence $\#F(G)\geq 1+ C(G,\gamma)$.

If the cut $\gamma'$ crosses $v'_3$, on lift it in $\G$ so that it crosses $e_1$ and $e_2$ (see figure~\ref{fig:GGsitHb}).
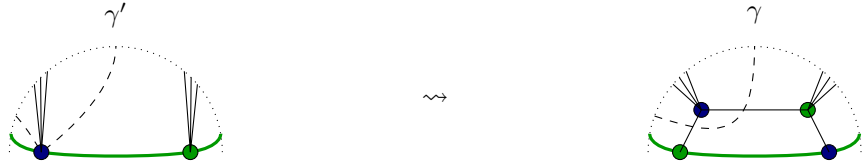
\begin{figure}[ht]
  \centering
 \begin{tikzpicture}[scale=1.4]
    \begin{scope}[xshift = -6cm, decoration={markings, mark=at
     position 0.7 with {\arrow{>}}},postaction={decorate}]
\coordinate (V3) at (-0.7,-0);
\coordinate (V4) at (0.7,-0);
\draw[dotted] (-1,0) arc (180:0:1 and 1);
\draw[darkgreen, very thick] (170:1) .. controls +(0,0) and +(-0.3,0.05) .. (V3) ..controls +(+0.3,-0.05)  and +(-0.3,-0.05) .. (V4) .. controls +(+0.3,0.05) and +(0,0).. (10:1);
\draw[dashed] (160:1).. controls +(0.1,-0.1) and +(0,0).. (V3) .. controls +(0,0) and +(0,-0.5) .. (0,1);
\filldraw[fill= darkblue] (V3) circle (2pt);
\filldraw[fill= darkgreen] (V4) circle (2pt);
\draw (V3) -- (130:1);
\draw (V3) -- (135:1);
\draw (V3) -- (140:1);
\draw (V4) -- (50:1);
\draw (V4) -- (45:1);
\draw (V4) -- (40:1);
\node at (0,1.3) {$\gamma'$};
\end{scope}
\node at (-3,0.5) {$\leadsto$};
\begin{scope}[xshift = 0cm, decoration={markings, mark=at
     position 0.7 with {\arrow{>}}},postaction={decorate}]
\coordinate (V1) at (-0.7,-0);
\coordinate (V2) at (0.7,-0);
\coordinate (V3) at (-0.5, 0.4);
\coordinate (V4) at (0.5, 0.4);
\draw[dotted] (-1,0) arc (180:0:1 and 1);
\draw[darkgreen, very thick] (170:1) .. controls +(0,0) and +(-0.3,0.05) .. (V1) ..controls +(+0.3,-0.05)  and +(-0.3,-0.05) .. (V2) .. controls +(+0.3,0.05) and +(0,0).. (10:1);
\draw[dashed] (160:1).. controls +(0.6,-0.2) and +(0,-0.9) .. (0,1);
\filldraw[fill= darkgreen] (V1) circle (2pt);
\filldraw[fill= darkblue] (V2) circle (2pt);
\filldraw[fill= darkblue] (V3) circle (2pt);
\filldraw[fill= darkgreen] (V4) circle (2pt);
\draw (V1) -- (V3);
\draw (V2) -- (V4);
\draw (V4) -- (V3);
\draw (V3) -- (130:1);
\draw (V3) -- (135:1);
\draw (V3) -- (140:1);
\draw (V4) -- (50:1);
\draw (V4) -- (45:1);
\draw (V4) -- (40:1);
\node at (0, 1.3) {$\gamma$};
\end{scope}
  \end{tikzpicture} 
  \caption{How to transform $\gamma'$ into $\gamma$.}
  \label{fig:GGsitHb}
\end{figure}
So that we have:
 \begin{align*}
C(G,\gamma)&= C(G',\gamma')+ \frac{1}{n(v_1)} +  \frac{1}{n(v_3)} - \frac{1}{n(v'_3)} + \frac{1}{n(v_4)} - \frac{1}{n(v'_4)}  \\
&\geq  C(G',\gamma')+\frac12 +1 -1 + 1 -\frac12. \\
&\geq C(G',\gamma') +1.
\end{align*}
But $\#F(G) =\# F(G')+1$, hence $\#F(G)\geq 1+ C(G,\gamma)$.
\paragraph{Conclusion}
\label{sec:conclusion}

For all situations, using the induction hypothesis we can construct a cut $\gamma$ such that: $\#F(G)\geq 1+ C(G,\gamma)$. This proves the lemma.

\end{proof}


\bibliographystyle{alpha}
\bibliography{biblio}
\end{document}